\documentclass[10pt,hidelinks]{article}

\usepackage{cancel}
\usepackage[normalem]{ulem}
\usepackage{etex}
\usepackage[left=2cm,top=2cm,right=2cm,bottom=2cm]{geometry}
\usepackage{amsmath, amsthm, amssymb}
\allowdisplaybreaks
\usepackage{pst-all}
\usepackage[textwidth=0.5in]{todonotes}
\usepackage[rightcaption]{sidecap}

\usepackage{relsize}
\usepackage{float}
\usepackage[caption = false]{subfig}

\usepackage[hypertexnames=false]{hyperref}

\usepackage[scaled=1]{helvet}
\usepackage{titlesec}
\usepackage{multido}
\usepackage{graphicx}
\usepackage{multirow}
\usepackage{blindtext}
\usepackage{color}
\usepackage{lipsum}
\usepackage{mathtools}
\usepackage[vcentermath]{youngtab}
\usepackage{young}
\usepackage[all,ps,dvips,graph]{xy}
\usepackage[misc,geometry]{ifsym}

\usepackage{titlesec}
\usepackage{lineno,hyperref}

\let\OLDthebibliography\thebibliography
\renewcommand\thebibliography[1]{
  \OLDthebibliography{#1}
  \setlength{\parskip}{0pt}
  \setlength{\itemsep}{0pt plus 0.3ex}
}

\titleformat{\section} {\normalfont\scshape \large \centering}{ \thesection}{1em}{}
\definecolor{morado}{rgb}{0.5,0,0.5}

\newcommand{\comu}{\mathbb C}

\newcommand{\cupdot}{\mathbin{\mathaccent\cdot\cup}}
\newcommand{\rank}{ {\color{black}{{\rm rk \, }}}}
\newcommand{\rankq}{ {\color{black}{{\rm rk}_q \, }}}
\newcommand{\rankWZeroq}{ {\color{black}{{\rm rk}_{W_0, q} \, }}}

\newcommand{\rankzq}{ {\color{black}{{\rm rk}_{z, q} \, }}}
\newcommand{\rankzeroq}{ {\color{black}{{\rm rk}_{z_0, q} \, }}}
\newcommand{\Soergelcat}{ {\mathcal D}_{(W,S)}}

\newcommand{\tildeSoergelcat}{ {{\mathcal D}}_{(W,S)}}
\newcommand{\tildeSoergelcatC}{ {\mathcal D}^{\comu}_{(W,S)}}

\newcommand{\F}{ { \mathcal F}}

\newcommand{\NB}{ {{\mathbb{NB}}_n}}
\newcommand{\NBgr}{ {{\mathbb{NB}_n^{gr}}}}
\newcommand{\NBnngr}{ {{\mathbb{NB}_{n-1}^{gr}}}}

\newcommand{\BRn}{ {\mathbb{B}^{x,y}_n}}
\newcommand{\BRgr}{ {\mathbb{B}^{gr, x,y}_n}}
\newcommand{\BRnn}{ {\mathbb{B}^{x,y}_{n-1}}}
\newcommand{\BRnngr}{ {\mathbb{B}^{gr, x,y}_{n-1}}}

\newcommand{\BRcinco}{ {\mathbb{B}^{x,y}_5}}

\newcommand{\JWk}{ {\mathbf{JW}_k}}
\newcommand{\JWn}{ {\mathbf{JW}_n}}
\newcommand{\JWnn}{ {\mathbf{JW}_{n-1}}}

\newcommand{\TL}{ {\mathbb{TL}}}

\newcommand{\Blob}{ {\mathbb{B}}}

\newcommand{\TRnn}{ {\mathbb{TL}_{n-1}}}
\newcommand{\TRk}{ {\mathbb{TL}_{k}}}
\newcommand{\TRn}{ {\mathbb{TL}_{n}}}
\newcommand{\TRcinco}{ {\mathbb{TL}_{5}}}

\newcommand{\Z}{\mathbb{Z}}

\newcommand{\B}{\mathbb{B}}

\newcommand{\CC}{ \mathbb C }

\newcommand{\spa}{{\rm span}}

\newcommand{\s}{\mathfrak{s}}

\newcommand{\T}{  \mathfrak{t}}

\newcommand{\tab}{{\rm Tab}}

\newcommand{\sign}{-}
\newcommand{\UU}{\mathbb{U}}

\newcommand{\Exp}{ {\rm \bf exp} }

\newcommand{\botts}[2]{ \mbox{Hom}_{{\mathcal D}} ( \underline{#1}, \underline{#2} ) }

\newcommand{\ese}{  {\color{red} s}}
\newcommand{\te}{ {\color{blue} t}}

\modulolinenumbers[5]

\newtheorem{theorem}{Theorem}[section]
\newtheorem{lemma}[theorem]{Lemma}

\newtheorem{definition}[theorem]{Definition}
\newtheorem{corollary}[theorem]{Corollary}

\newtheorem{remark}[theorem]{Remark}

\newenvironment{dem}{\noindent \textit{Proof:} }{\quad \hfill $\square$}

\numberwithin{equation}{section}

\newgray{plomoclaro}{.90}
\newgray{plomooscuro}{.70}
\newgray{blanco}{1}

\begin{document}
\Yvcentermath1
\sidecaptionvpos{figure}{lc}

\title{Graded sum formula for $\tilde{A}_1$-Soergel calculus and the nil-blob algebra. }

\author{
  \, Marcelo Hern\'andez Caro{\thanks{Supported in part by beca ANID-PFCHA/Doctorado Nacional/2019-21190827
  }} \, and  Steen Ryom-Hansen{\thanks{Supported in part by FONDECYT grant 1221112  }}}

\date{\vspace{-5ex}}
\maketitle
\begin{abstract}
  We study the representation theory of the Soergel calculus algebra
  $ {\color{black}{ \tilde{A}_w^{\comu}}} := \mbox{End}_{\Soergelcat} (\underline{w})
  $ over $ \comu$ in type $ \tilde{A}_1$. 
  We generalize the recent isomorphism between the nil-blob algebra $ \NB$ and
  {\color{black}{a diagrammatically definied subalgebra $ {A}_w^{\comu}$} of } 
  ${\color{black}{ \tilde{A}_w^{\comu}}}$ to deal with the two-parameter blob algebra.
  Under this generalization, the two parameters correspond
  to the two simple roots for $  \tilde{A}_1$. Using this,
  together with calculations 
  involving the Jones-Wenzl idempotents for the Temperley-Lieb subalgebra of $ \NB$, we obtain a
  concrete diagonalization of the matrix of the bilinear form on the cell module
$ \Delta_w(v) $ 
  for $ {\color{black}{ \tilde{A}_w^{\comu}}} $. The 
  entries of the diagonalized matrices turn out to be products of roots for
  $  \tilde{A}_1$. We use this to study Jantzen type filtrations of $ \Delta_w(v) $ 
  for $ {\color{black}{ \tilde{A}_w^{\comu}}}$. 
  We show that, at {\color{black}{an}} enriched Grothendieck group level, the corresponding
  sum formula has terms $ \Delta_w(s_{\alpha }v)[ l(s_{\alpha }v)- l(v)] $, where
  $ [ \cdot ] $ denotes grading shift. 
  \end{abstract}

\section{Introduction}
Cellular algebras were introduced by Graham and Lehrer in 1994 {\color{black}in the paper \cite{GL}}
as a framework for
studying the non-semisimple representation theory of many finite dimensional 
algebras, 

The motivating examples for cellular algebras were the Iwahori-Hecke algebras of type $ A_n $
and Temperley-Lieb algebras, but it has since been realized that many other finite
dimensional 
algebras fit into this
framework. 
For a cellular algebra one has 
a family of
{\color{black}{cell modules}}
$ \{ \Delta(\lambda)  \}$, endowed with bilinear forms
$ \langle \cdot, \cdot \rangle_ \lambda$, that together control the representation theory of
the algebra in question.
Unfortunately, the concrete analysis of these bilinear forms is in general 
difficult, but in this paper we give a non-trivial  
cellular algebra over $ \comu $ for which the bilinear forms 
$ \langle \cdot, \cdot \rangle_ \lambda$ can in fact be diagonalized over an integral form of the algebra, 
thus solving all relevant questions
concerning them, and therefore, by cellular algebra theory, concerning the representation theory
of the algebra itself.

\medskip
Our cellular algebra has two origins. Firstly it arises in the diagrammatic Soergel 
calculus of the Coxeter system $ (W,S) $ of type $ \tilde{A}_1 $
as the endomorphism algebra
$ \tilde{A}_w^{\comu} := \mbox{End} ({w}) $ of $ w :=
{\color{red} s} {\color{blue} t} {\color{red} s} {\color{blue} t} \cdots $ of length $n $,
where $ S= \{ {\color{red} s}, {\color{blue} t} \} $. 
An approach to Soergel calculus of universal Coxeter groups, in particular of type $ \tilde{A}_1 $,
has been developed 
recently by Elias and Libedinsky in
\cite{EL}, see also \cite{BLS}. For type $ \tilde{A}_1 $ this approach involves the
two-colour Temperley-Lieb algebra but unfortunately the
two-colour Temperley-Lieb algebra only captures the degree zero
part of $ \tilde{A}_w^{\comu} $,
whereas our interests lie in the full grading on $ \tilde{A}_w^{\comu} $. 

\medskip
The  second origin of our cellular algebra 
is as a certain idempotent truncation $\mathbb{NB}_{n-1} $
of Martin and Saleur's blob-algebra
{\color{black}{from \cite{Mat-Sal}}}. 
In \cite{LPR}
the algebras $ \tilde{A}_w^{\comu}  $ and $\mathbb{NB}_{n-1} $ 
were studied extensively 
and in particular presentations in terms of generators and relations were found for
each of them. Using this
it was shown that there is an isomorphism $ {A}^{\comu}_w  \cong  \mathbb{NB}_{n-1} $
where $ A_w^{\comu} $ is a natural diagrammatically defined subalgebra of $ \tilde{A}^{\comu}_w $,
whose dimension is half the dimension of $ \tilde{A}^{\comu}_w $. On the other hand, we show
in this paper that the representation theory of $ \tilde{A}^{\comu}_w $ can be completely recovered from
the representation theory of $ A_w^{\comu} $.

%% \sout{
%% {\color{black}{
%%     Note that the category associated with $\mathbb{NB}_{n}$ already appeared in \cite{Russel}
%%     where it was shown to be equivalent to}}}
%% \newline
%% \sout{
%% {\color{black}{the
%% Bar-Natan category of the thickened annulus.
%%     This equivalence is used in the recent work \cite{HRW} to construct a}}}
%% \newline
%% \sout{
%% {\color{black}{Kirby colour for Khovanov homology}}}

\medskip
Similarly to the original blob-algebra, the diagrammatics for $ \mathbb{NB}_{n-1} $
is given by blobbed (marked)
Temperley-Lieb diagrams, as in the following
examples
\begin{equation}
 \raisebox{-.43\height}{\includegraphics[scale=0.8]{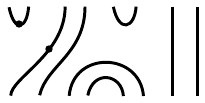}},
  \, \,   \, \,   \, \,   \, \, 
  \raisebox{-.43\height}{\includegraphics[scale=0.8]{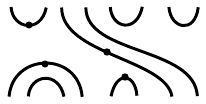}}
\end{equation}
although the rule for multiplying diagrams is different. Following \cite{LPR}, we call 
$ \mathbb{NB}_{n-1} $ the
nil-blob algebra, but in fact
$ \mathbb{NB}_{n-1} $ has also appeared 
in the literature under the name the \emph{dotted Temperley-Lieb algebra}, see \cite{Russel}. 
In \cite{Russel}
it was shown that the
associated \emph{dotted Temperley-Lieb category} is equivalent to the 
Bar-Natan category of the thickened annulus and 
this equivalence was used in the recent work \cite{HRW} to construct a
Kirby color for Khovanov homology.

\medskip
An important feature of $ \tilde{A}_w^{\comu} $ and $ \mathbb{NB}_{n-1} $,
and in fact of all cellular algebras appearing in this paper, is the fact that
they are $ \Z $-graded algebras, with explicitly given degree functions defined in terms of the diagrams.
They are $\Z$-graded cellular in the sense of Hu and Mathas,
{\color{black}{see \cite{hu-mathas}}}.

\medskip

In this paper our first new result is a construction of integral forms
$ {A}_w  $ and $ \BRnn $
for $ {A}^{\comu}_w  $ and $ \mathbb{NB}_{n-1} $
over the two-parameter polynomial algebra $ R:= {\mathbb C}[x,y] $ 
and a 
lift of the isomorphism
$ {A}^{\comu}_w  \cong  \mathbb{NB}_{n-1} $ to 
$ {A}_w  \cong   \BRnn $. 
The integral form $ A_w $ 
is in fact already implicit in the setup for Soergel calculus,
using the
{\color{black}{dual}} geometric realization of the Coxeter group $ W $ of 
type $ \tilde{A}_1$. Under this realization, the parameters $ x $ and $ y $ correspond to the
two simple roots for $W$.
The integral form for $ \mathbb{NB}_{n-1} $ is also a well-known object, since it is
simply the two-parameter blob-algebra $ \BRnn$ with
blob-parameter $ x $ and 
marked loop parameter 
$ y $. 
Thus the novelty of our result lies primarily in the isomorphism between these
integral forms, which on the other hand has the quite 
surprising consequence of rendering 
a Coxeter group theoretical meaning to the two blob algebra parameters
for $ \BRnn$, since they become nothing but the simple roots for $ W $. 

\medskip
Our main interests lies in the representation theory of $ A_w $ which via the
above isomorphism is equivalent to the representation theory of $ \BRnn$. 
For several reasons the representation theory of $ \BRnn$ is more convenient to handle. 
Both algebras are cellular algebras with diagrammatically defined cellular bases, but  
the straightening rules for expanding the product of two cellular basis
elements in terms of the cellular basis are
{\color{black}{easier}}
in the $ \BRnn$ setting. Secondly, there is a natural Temperley-Lieb subalgebra 
$ \TRnn$ of $ \BRnn$ whose associated restriction functor $ {\rm Res} $ is very useful for our purposes.
We show 
{\color{black}{in section \ref{restriction} of our paper}}
that $ {\rm Res} $ maps a cell module $ \Delta^{\B}_{n-1}(\lambda) $ for $ \BRnn$ to a module with a cell module
filtration for 
$ \TRnn$ and that the sections of this filtration induce a diagonalization
of the bilinear form $ \langle \cdot, \cdot \rangle^{\B}_{n-1,\lambda} $ on $ \Delta^{\B}_{n-1}(\lambda) $.

\medskip
This leaves us with the task of calculating the values of the bilinear form on these sections. For
$ \lambda \ge 0 $ and $ n-1 = 2k + \lambda $ we 
show that this task is equivalent to calculating the coefficient of the identity
{\color{black}{
$ \raisebox{-.3\height}{\includegraphics[scale=0.9]{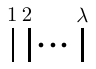}}  $
in the expansion of 
 \begin{equation}
 \raisebox{-.5\height}{\includegraphics[scale=0.9]{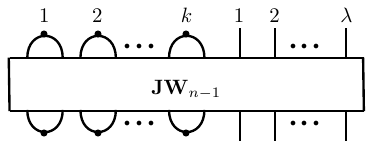}}  
 \end{equation}
 where $ \mathbf{JW}_{n-1} $ is the Jones-Wenzl idempotent for
 $ \TRnn$. Similarly, for $ \lambda < 0 $ and $ n-1 = 2k + |\lambda |$ it is equivalent to calculating the coefficient of 
 $ \raisebox{-.3\height}{\includegraphics[scale=0.9]{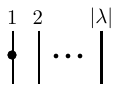}}  $ in the expansion of
  \begin{equation}
 \raisebox{-.5\height}{\includegraphics[scale=0.9]{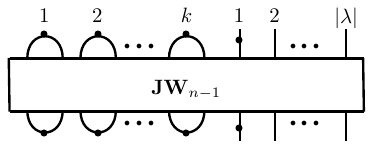}}  
\raisebox{-26\height}{\color{black}{.}}
  \end{equation}}}
The determination of these coefficients constitutes the main
calculatory ingredient of our paper and is done in section \ref{deter}. The result is 
given in Theorem \ref{cifactorA} for $ \lambda \ge 0 $ and in Theorem \ref{cifactorB}
for $ \lambda < 0 $. Although one may possibly not have expected Coxeter theory to 
appear in this calculation, the result turns out to be a 
nice product of positive roots for $ W$.

\medskip
This work is partially motivated by the paper \cite{steen} by the second
author
in which
a diagonalization of the bilinear form for the cell module for $ \tilde{A}_w $ is obtained,
in fact
{\color{black}{the results of \cite{steen} are valid}}
for a general
Coxeter system. Unfortunately, as was already mentioned in \cite{steen},
the diagonalization process in that paper
does not work over $R$ itself, but only 
over the fraction field $ Q $ of $R$,
since it relies on certain {\color{black}{Jucys-Murphy}} elements for $ A_w $ that 
are of degree 2,
and not 0.  As a consequence 
the $\Z$-graded structure on the cell module for $ \tilde{A}_w $ breaks down under the
diagonalization process in \cite{steen}. The diagonalization process of the present paper, however,
which is based on the Jones-Wenzl idempotents that are of degree 0, resolves
this problem at least for type $\tilde{A}_1$.

\medskip
In the final section \ref{Jantzen} of the paper we
use the results of the previous sections
to 
set up a
%\cancel{\color{black}{graded}}
version $  \Delta^{gr , \comu}_{w}(v)  \supseteq  \Delta^{gr , 1, \comu}_{w}(v) \supseteq 
\Delta^{gr , 2, \comu}_{w}(v) \supseteq  \ldots $ 
of the Jantzen filtration formalism
for the graded cell module $\Delta^{gr , \comu}_{w}(v) $,
{\color{black}{using the bilinear form
$ \langle \cdot, \cdot \rangle_{w}^v $ on $  \Delta^{gr , \comu}_{w}(v) $.
Since $ \langle \cdot, \cdot \rangle_{w}^v $ is a graded bilinear form,
this filtration consists of graded submodules of $ \Delta^{gr , \comu}_{w}(v) $, and in our final
Corollary \ref{mimickingA} we show that the following identity holds 
at {\it enriched Grothendieck} group level
\begin{equation}\label{bothsidesA}
  \sum_{k>0}   \langle \Delta^{gr, k, \comu}_{W_0}(v) \rangle_q   =
  \sum_{\substack{ \alpha >0 \\  v <s_{\alpha} v \le w}}
  \langle \Delta^{gr,  \comu}_{W_0}(s_{\alpha}v )
 [l(s_{\alpha}v) -l(v)]   \rangle_q
\end{equation}  
where $ \alpha >0 $ refers to the positive roots for $ W $ and $ [ \cdot ] $ to grading shift.
This is our \emph{graded} Jantzen sum formula.}}
Analogues of (ungraded) Jantzen filtrations with associated sum formulas
exist in many module categories of Lie type and 
give information on the irreducible modules
for the category in question, see for example \cite{A}, \cite{A1}, \cite{MathasJan}. 
But although graded representation theories
in Lie theory have been known since the beginning of the nineties,
see for example 
\cite{AJS}, \cite{BGS}, \cite{EhrigStroppel},
\cite{FL}, \cite{Hu}, \cite{Lobos-Ryom-Hansen} and \cite{Stroppel}, 
to our knowledge graded sum formulas
{\color{black}{in the sense of \eqref{bothsidesA}}}
have so far not been available, even though they would be very useful for
calculating decomposition numbers.
We believe
\eqref{bothsidesA} gives an interesting indication of the possible form of graded sum formulas in 
representation theory.

\medskip
The layout of the paper is as follows. In the next section we introduce the notation that
shall be used throughout the paper
and recall
the various algebras that play a role throughout: the Temperley-Lieb algebra
$\TRn$, the blob algebra $\BRn $, the nil-blob algebra $\NB$ and the Soergel algebra
$ \tilde{A}_w$. We also recall how each of them fits into the cellular algebra language. In section
\ref{isomorphism theorems} we introduce the subalgebra 
$ {A}_w $ of $  \tilde{A}_w $ and 
show the isomorphism $ \BRnn \cong  {A}_w $ that was mentioned above. We also show how
the cellular algebra structure on $ \tilde{A}_w $ induces a cellular structure on $ {A}_w $
and that there is an isomorphism $ \Delta^{\Blob}_{n-1}(\lambda) \cong \Delta_w(v)$ between
the respective cell modules for $ \BRnn $ and ${A}_w$. In section \ref{restriction} we consider
a natural  filtration
of ${\rm Res} \Delta^{\Blob}_{n}(\lambda) $ where ${\rm Res} $ 
is the restriction functor from $ {\BRnn} $-modules to $ {\TRnn} $-modules.
We show that the Jones-Wenzl idempotents $\JWk $ for $\TRk $ where $ k \le n-1 $ can be used
to construct sections for this filtration and to diagonalize the bilinear form $ \langle \cdot, \cdot
\rangle_{n-1, \lambda}^{\B} $ on $\Delta^{\Blob}_{n-1}(\lambda)$. In section \ref{deter}
we prove the key Theorems \ref{cifactorA} and \ref{cifactorB}, that were already mentioned above.
They allow us to give concrete expressions for the diagonal
elements of the matrix for $ \langle \cdot, \cdot\rangle_{n-1, \lambda}^{\B}$, which, as already
mentioned, turn out to be products of
positive roots $ \alpha $ for $W$. 
In section \ref{charof} we give a description of the reflections $ s_{\alpha} $ in $ W $ that correspond to
the positive roots of section \ref{deter}. Finally, in section \ref{Jantzen} we use the results of the previous
sections to give the graded Jantzen filtration with corresponding graded sum formula.

\medskip
    {\color{black}{The authors wish to express their gratitude to P. Wedrich for 
        useful conversations and for pointing out that the nil-blob algebra and the
        dotted Temperley-Lieb algebra are the same. 
        They also wish to thank the two anonymous referees for detailed reports
        that greatly helped improving the presentation and accuracy of the paper.}}

\section{Blob algebras and Soergel calculus for $\tilde{A}_1$}
Throughout we use as ground field the complex numbers $\comu$, although several of 
our results hold in greater generality.
We set
\begin{equation}\label{defR}
R:=\comu[x,y].
\end{equation}
We consider $ R $ to be a (non-negatively) $\Z$-graded $\comu$-algebra via
\begin{equation}
  {\rm deg}(x) =   {\rm deg}(y) = 2.
\end{equation}

In this paper we shall consider several diagram algebras. Possibly the oldest and most studied 
diagram algebra is the Temperley-Lieb algebra. It arose in statistical mechanics in the seventies. 
In the present paper we shall use the following variation of it. 

\begin{definition}\label{deftwoTL}
The Temperley-Lieb algebra $\TRn$ with loop-parameter $ -2$
is the $R$-algebra on the
  generators $\UU_1, \ldots , \UU_{n-1} $
  subject to the
  relations
	\begin{align}
\label{eq oneTL}	\UU_i^2	& =  \sign 2 \UU_i  &   &  \mbox{if }  1\leq i < n \\
\label{eq twoTL}	\UU_i\UU_j\UU_i & =\UU_i    &  & \mbox{if } |i-j|=1
%\cancel{\color{black}{  \mbox{ and } i,j >0} }
 \\
 \label{eq threeTL}	\UU_i\UU_j& = \UU_j\UU_i    &  &   \mbox{if } |i-j|>1
 {\color{black}{.}}
	\end{align}
\end{definition}

The blob algebra
was introduced by Martin and Saleur
in \cite{Mat-Sal}, 
{\color{black}{as a 
a way of considering boundary conditions in 
the statistical mechanical model of the Temperley-Lieb algebra. 
Since its introduction, the blob algebra has been the subject of much research activity in mathematics
as well as physics, see for example  \cite{GJSV}, \cite{Lobos-Ryom-Hansen}, 
 \cite{Mar-Wood}, \cite{Michailidis}, \cite{Plaza1}, \cite{Plaza}, 
\cite{PlazaRyom}}}.
In this paper, we shall use the following variation of it.

\begin{definition}\label{deftwoblob}
The two-parameter blob algebra $\BRn$,  or more precisely
the blob algebra with loop-parameter $ -2$, marked loop parameter 
$ y $ and blob-parameter $ x $, 
is the $R$-algebra on the
  generators $\UU_0,\UU_1, \ldots , \UU_{n-1} $
  subject to the relations 
	\begin{align}
\label{eq one}	\UU_i^2	& =  \sign 2 \UU_i  &    &  \mbox{if }  1\leq i < n \\
\label{eq two}	\UU_i\UU_j\UU_i & =\UU_i      & & \mbox{if } |i-j|=1 \mbox{ and } i,j >0  \\
\label{eq three}	\UU_i\UU_j& = \UU_j\UU_i     &  &  \mbox{if } |i-j|>1 \\
\label{eq four}	\UU_1\UU_0\UU_1& =  y \UU_1    & &   \\
\label{eq five}			\UU_0^2& = x \UU_0 {\color{black}{.}}   &&     
	\end{align}  
\end{definition}

The nil-blob algebra $ \NB $, that was introduced and studied extensively in \cite{LPR}, may be recovered from $ \BRn $ via
specialization, that is 
\begin{equation}
\NB \cong \BRn \otimes_R \comu
\end{equation}  
where $\comu$ is made into an $ R$-algebra via $ x\mapsto 0 $ and $ y\mapsto 0$. In other words, 
$ \BRn $ may be considered a deformation of $ \NB$, and in fact this shall be the point of view of the 
present paper.

\medskip
Another interesting specialization of $ \BRn $ is $ \widetilde{\cal T}_n  $ defined 
as $\widetilde{\mathcal T}_n  := \BRn \otimes_R \comu$
where $ \comu $ is made into an $ R $-algebra via $ x\mapsto 1 $ and $ y\mapsto -2$.
Let $ {\mathcal  I}:= \langle \UU_0 -1\rangle $ be the two-sided
ideal in $\widetilde{\cal T}_n $ generated by $ \UU_0-1 $.
Then
\begin{equation}\label{istheTL}
{\mathcal  T}_n:= \widetilde{\mathcal T}_n/ {\mathcal  I} 
\end{equation}
is the Temperley-Lieb algebra from Definition \ref{deftwoTL}, but defined over $ \comu$.

\medskip
Just as is the case for $\NB$, one easily checks that $ \BRn $ is a $ \Z$-graded algebra.
\begin{lemma}
  The rules $ {\rm deg}(\UU_i) = 0 $ for $ i > 0 $ and $ {\rm deg}(\UU_0)  = 2 $ 
  define a (non-negative) $ \Z$-grading on $\BRn$.
\end{lemma}  
\begin{proof}
The relations are easily seen to be homogeneous with respect to ${\rm deg} $.
\end{proof}  
As already indicated, $ \TRn $ and $ \BRn$ are diagram algebras.
This fact plays an important role in our paper, and let us briefly explain it.
The diagram basis for $ \TRn$ consists of {\it Temperley-Lieb diagrams} on $ n $ points,
which are planar pairings between $ n $ northern points and $ n $ southern points of a rectangle. 
The diagram basis for $ \BRn$ consists of blobbed (marked) Temperley-Lieb diagrams on $ n $ points,
or {\it blob diagrams} on $ n $ points, 
which are marked planar pairings between $ n $ northern points and $ n $ southern points of a rectangle, 
where 
only pairings exposed to the left side of the rectangle may be marked, and at most once.
There is thus a natural embedding of Temperley-Lieb diagrams into blob diagrams.
The multiplication  
{\color{black}{$ D_1 D_2 $ of two diagrams $ D_1 $ and $ D_2 $ 
is given by concatenation of them, with $ D_2 $ on top of $ D_1$.}} This concatenation process
may give rise to internal 
marked or unmarked loops, as well as diagrams with more than one mark.
Internal unmarked loops are removed from a diagram
by multiplying it by $  -2 $, whereas internal marked loops are removed
from a diagram by multiplying it by $y $. Finally, any diagram with $r>1$ marks on a diagram is set
equal to the same diagram multiplied by $ x^{r-1} $, but with the $(r-1)$ extra marks removed. 
{\color{black}{For example, for
\begin{equation}
  D_1 \,  =  \, \raisebox{-.43\height}{\includegraphics[scale=0.8]{blobExampleB.pdf}},
  \, \,   \, \,   \, \,   \, \, 
    D_2 \,  = \,  \raisebox{-.43\height}{\includegraphics[scale=0.8]{blobExampleA.pdf}}
\end{equation}
we have that 
\begin{equation}
  D_1 D_2 \,   = \,  \raisebox{-.43\height}{\includegraphics[scale=0.8]{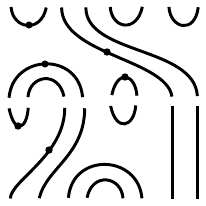}} \,  = \, 
x^2 y  \raisebox{-.43\height}{\includegraphics[scale=0.8]{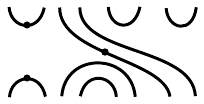}} 
\raisebox{-13.5\height}{\color{black}{.}}
\end{equation}
Later on, we shall give many more examples. 
}}

\medskip
For the proof of the isomorphisms between $ \BRn$
%\cancel{\color{black}{and}} 
and its diagrammatic version, one may 
consult
the appendix of \cite{blob positive} or else adapt the more self-contained proof given in \cite{LPR},
and similarly for $ \TRn $. 
Under the isomorphism we have that 
\begin{equation}{\label{isomorphism}} \! \!\! \!\! \! \! \!
  1 \mapsto \quad \raisebox{-.43\height}{\includegraphics[scale=0.8]{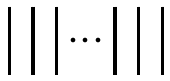}} 
\end{equation}  
and that 
\begin{equation}{\label{isomorphism1}} \! \!\! \!\! \! \! \!
  \UU_0 \mapsto \quad \raisebox{-.43\height}{\includegraphics[scale=0.8]{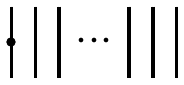}} \,,
   \, \,\,\, \,
{\color{black}{   \UU_i \mapsto \quad \raisebox{-.35\height}{\includegraphics[scale=0.8]{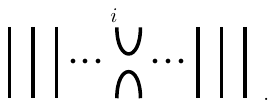}} }}
\end{equation}

\noindent
The number of Temperley-Lieb diagrams and blob diagrams on $ n $ points is
the Catalan number $  \frac{1}{n+1} \binom{2n}{n} $ and 
$ \binom{2n}{n} $. In particular $\TRn$ and 
$\BRn$ are free over $ R $ of {\color{black}{rank}}
\begin{equation}\label{rankblob}
  \rank \TRn = \frac{1}{n+1}  \binom{2n}{n}\,\, \,  \mbox{and  } \rank \BRn = \binom{2n}{n}
  {\color{black}{.}}
\end{equation}

All algebras considered in the present paper fit into the general language of {\it cellular algebras},
introduced by Graham and Lehrer. 
\begin{definition}\label{defcellularalgebra}
Let $ \cal A $ be a finite dimensional algebra 
over a commutative ring $ \Bbbk     $ with unity.
Then a cellular basis structure for $\cal A$ consists of a triple $ (\Lambda, {\tab},  C ) $ such that $ \Lambda $ is a poset, 
$\tab $ is a function on $ \Lambda $ with values in finite sets
and $ C:\coprod_{\lambda \in \Lambda} {\tab}(\lambda)  \times {\tab}(\lambda) \rightarrow \cal A$
is an injection 
such that 
$$ \{ C_{st}^{\lambda} \, |\,  s,t \in {\tab}(\lambda), \,  \lambda \in \Lambda \} $$
is a {\color{black}{$ \Bbbk$}}-basis for
{\color{black}{$\cal  A $}}:
the cellular basis for $ \cal A$. The
rule $ (C_{st}^{\lambda})^{\ast} := C_{ts}^{\lambda}$ defines 
a {\color{black}{$ \Bbbk$}}-linear antihomomorphism of $\cal A$ and 
the structure constants for $\cal A$ with respect to $ \{ C_{st}^{\lambda} \}$ 
satisfy the following condition with respect to the partial order: for all $ a \in \cal A $ we have
$$  a C_{st}^{\lambda} = \sum_{u \in \tab(\lambda) }
r_{usa} C_{ut}^{\lambda} + \mbox{ lower  terms} $$
where lower terms means a linear combination of $ C^{\mu}_{ab} $ where $ \mu < \lambda $ 
and where $ r_{usa} \in {\color{black}{ \Bbbk    }} $.
\end{definition}

To make $ \TRn $ fit into this language we choose $ \Bbbk = R$, 
$ \Lambda= \Lambda_n := \{ n,  n-2, \ldots,  1 \} $  
(or $ \Lambda := \Lambda_n = \{  n,  n-2, \ldots,  2, 0 \} $)
if $ n $ is odd (or even), 
with poset structure inherited from $ \mathbb Z$. 
For $ \lambda \in \Lambda_n $ we choose $ {\tab}(\lambda) $ to be Temperley-Lieb
{\it half-diagrams} with $ \lambda $ 
{\it propagating} lines, that is
Temperley-Lieb diagrams on $  \lambda  $ northern and $ n $ southern points
in which each northern point is paired with a southern point. For $ s, t \in {\tab}(\lambda) $
we define $ C_{st}^{\lambda} $ to be the diagram obtained from gluing $ s $ and the
horizontal reflection of $ t $, with $ s$ on the bottom. Here is an example of this gluing process with
$ n = 8 $ and 
$ s, t \in {\tab}(2) $.

\begin{equation}{\label{isomorphism2}} \! \!\! \!\! \! \! \!
  (s,t) = \,\, \left(\raisebox{-.48\height}{\includegraphics{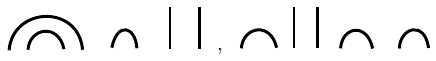}}\right) \mapsto
  \raisebox{-.48\height}{\includegraphics{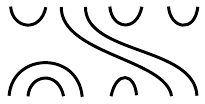}}
  \raisebox{-15\height}{.}
  \end{equation}
\begin{theorem}\label{firstcellular1}
The above triple $ (\Lambda_n, {\tab},  C ) $ makes $ \TRn$ into a cellular algebra. 
\end{theorem}
\begin{dem}
This follows directly from the definitions.
\end{dem}

\medskip
To make $ \BRn $ fit into the cellular algebra language we choose $ \Bbbk = R$, 
$ \Lambda= \{ \Lambda_{\pm n} := \{ \pm n, \pm (n-2), \ldots, \pm 1 \} $  
(or $ \Lambda := \Lambda_{\pm n} = \{ \pm n, \pm (n-2), \ldots, \pm 2, 0 \} $)
if $ n $ is odd (or even), 
with poset structure given by $ \lambda < \mu $ if $ |\lambda| < |\mu| $ or if
$ |\lambda| = |\mu| $ and $ \lambda < \mu $. For example, for $ n= 6 $ we have 
\begin{equation}
\Lambda_{\pm 6} := \{ \pm 6, \pm 4, \pm 2,0 \}, \, \, \,\,  \, \, \,\, 0<-2 < 2 <-4 < 4 <-6 < 6
\end{equation}  
For $ \lambda \in \Lambda_{\pm n} $ we choose $ {\tab}(\lambda) $ to be blob {\it half-diagrams} with $ |\lambda |$ 
{\it propagating} lines, that is
marked Temperley-Lieb diagrams on $  |\lambda | $ northern and $ n $ southern points
in which each northern point is paired with a southern point and in which 
only non-propagating pairings exposed to the left side of the rectangle may be marked. For $ s, t \in {\tab}(\lambda) $
we define $ C_{st}^{\lambda} $ to be the diagram obtained from gluing $ s $ and the
horizontal reflection of $ t $, with $ s$ on the bottom, and marking the leftmost propagating line if
$\lambda < 0$. Here is an example of this gluing process with
$ n = 8 $ and 
$ s, t \in {\tab}(-2) $.

\begin{equation}{\label{isomorphism3}} \! \!\! \!\! \! \! \!
  (s,t) = \,\, \left(\raisebox{-.48\height}{\includegraphics{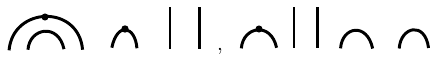}}\right) \mapsto
    \raisebox{-.48\height}{\includegraphics{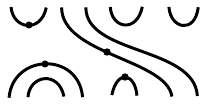}}
  \raisebox{-19\height}{.}
\end{equation}

\medskip
With this notation we have the following Theorem.
\begin{theorem}\label{firstcellular2}
The above triple $ (\Lambda_{\pm n}, {\tab},  C ) $ makes $ \BRn$ into a cellular algebra. 
\end{theorem}
\begin{dem}
This also follows directly from the definitions.
\end{dem}  

\medskip
For $ \lambda \in \Lambda_{n} $ or for $ \lambda \in \Lambda_{\pm n} $
we define $ t_{\lambda} \in {\tab}(\lambda) $ to be the following half-diagram
\begin{equation}{\label{tlambda1}} \! \!\! \!\! \! \! \!
  t_{\lambda} = \,\, \raisebox{-.5\height}{\includegraphics{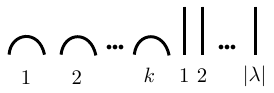}}
\end{equation}
and set $ C_\lambda = C^{\lambda}_{t_{\lambda}t_{\lambda}} $,  that is 
\begin{equation}{\label{tlambda2}} 
  C_{\lambda} = \left\{
  \begin{array}{l} \UU_1 \UU_3 \cdots \UU_{2k-1} =
    \raisebox{-.5\height}{\includegraphics{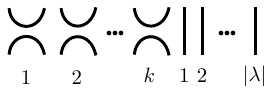}}   \mbox{ if }  \lambda \ge 0 \\        \\
  (\UU_1 \UU_3 \cdots \UU_{2k-1}) \UU_0 (\UU_2 \UU_4 \cdots \UU_{2k})   \UU_1 \UU_3 \cdots \UU_{2k-1} =
    \raisebox{-.5\height}{\includegraphics{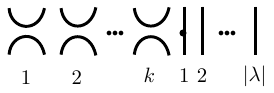}}   \mbox{ if }  \lambda < 0 
  \end{array}
  \right.
\end{equation}
where $ k = (n-|\lambda|)/2$.
Thus $ C_\lambda $ is an element of $ \TRn$ or $ \BRn$, depending on the context.
With this we can 
define the two-sided {\it cell} ideals $ A^{\lambda} $ and
$ A^{< \lambda} $ of $ \TRn$ (resp. $ \BRn$) via
\begin{equation}
  A^{\lambda} := \{ a C_{\lambda} b \, | \, a,b \in \TRn \, (\mbox{resp. } \BRn) \}  \mbox{ and }  A^{< \lambda} :=
\{ a C_{\mu} b \, | \, a,b \in \BRn \, (\mbox{resp. } \BRn) \mbox{ and } \mu < \lambda \} {\color{black}{.}}
\end{equation}
It follows from Definition \ref{defcellularalgebra} that $ A^{< \lambda} \subset A^{ \lambda} $ and
so $ A^{ \lambda}/ A^{< \lambda}  $ is a $ \TRn$-bimodule (resp.$\,\,  \BRn$-bimodule).
Let $ \overline{C}_{\lambda}:=C_{\lambda}+A^{< \lambda}$. 
We define the {\it cell module}
$ \Delta^{\TL}_n(\lambda) $ for $ \TRn$ (resp. $ \Delta^{\Blob}_n(\lambda) $ for $ \BRn$) as 
\begin{equation}\label{deficellmodule}
  \Delta^{\TL}_n(\lambda) := \TRn \overline{C}_{\lambda} \, (\mbox{resp. }
\Delta^{\Blob}_n(\lambda) :=
  \BRn \overline{C}_{\lambda})
  \subseteq A^{ \lambda}/A^{< \lambda} {\color{black}{.}}
\end{equation}
We now have the following Theorem. 
\begin{theorem}\label{cellblob}
  Let $ \Delta_n^{\TL}(\lambda) $ be the cell module for $ \TRn$. 
Then $ \Delta_n^{\TL}(\lambda) $ is free over $ R $ with basis
$ \{ \overline{C}_{s t_{\lambda}}\, | \, s \in {\tab}(\lambda) \} $
where $ \overline{C}_{s t_{\lambda}} := {C}_{s t_{\lambda}} +
A^{< \lambda} $. A similar statement holds for $ \Delta_n^{\Blob}(\lambda) $.
\end{theorem}
\begin{dem}
  This follows from the algorithm given in the proof Theorem 2.5 of \cite{LPR}. 
\end{dem}  
  
\begin{remark}\rm\label{remarkoncellmodule}
Our definition \eqref{deficellmodule} of the cell modules for $  \TRn $ and 
$ \BRn  $ differs slightly from the definition given in \cite{GL}, but in view of
the Theorem the two definitions coincide.
\end{remark}

\medskip

We now come to Soergel calculus.
In \cite{LPR}, an isomorphism $  {A}_w \cong \NB$
was established, where $  {A}_w $
is the endomorphism algebra of a Bott-Samelson object in the Soergel calculus of type $\tilde{A}_1$. 
We aim at generalizing this result to an isomorphism involving $ \BRn$. 

\medskip
To any Coxeter system $ (W, S) $, 
Elias and Williamson associated in \cite{EW}, a diagrammatic 
category $ \tildeSoergelcat  $. 
We fix $ S  := \{ {\color{red} s}, {\color{blue} t} \} $ and let $ W $ be the Coxeter group on
$ S $ given by 
\begin{equation} \label{presentation W}
W := \langle  {\color{red} s}, {\color{blue} t} \, |\,   {\color{red} s}^2 ={\color{blue} t}^2=e \rangle 
\end{equation}
that is, $ W $ is the affine Weyl group of type  $\tilde{A}_1$, or the infinite dihedral group
with Bruhat order $ < $ chosen such that $ 1 $ is the minimal element. 
The definition of ${\mathcal D}_{(W,S)}$ depends on the choice of 
a \emph{realization} $\mathfrak{h}$ of $(W,S)$,
which is a representation $ \mathfrak{h} $ of $ W $ over $ \comu$, arising from a choice of 
{\it simple roots} and {\it simple coroots}, see \cite[Section 3.1]{EW}.
In \cite{LPR}, the realization $ \mathfrak{h} $ was chosen to be the  \emph{geometric representation}
of $W$ defined over $ \comu$, see  \cite[Section 5.3]{H1}, 
with coroots 
$ \alpha^\vee_\ese,  \alpha^\vee_\te  $ being a basis for $ \mathfrak{h} $ and roots 
$ \alpha_\ese ,  \alpha_\te \in \mathfrak{h}^{\ast}  $ given by 
\begin{equation}  \label{equations realization}
\alpha_{\ese } ( \alpha^\vee_\ese )= 2,   \, \, \, \,  \alpha_\te (\alpha^\vee_\ese )= -2,     \, \, \, \,
\alpha_\ese (\alpha^\vee_\te ) = -2,  \, \, \, \,      \alpha_\te (\alpha^\vee_\te )= 2
\end{equation}
that is $ \alpha_{\ese} = -\alpha_\te  $. With this choice of realization of $ (W,S)$, 
the symmetric algebra of the dual representation is 
$R := S(\mathfrak{h}^\ast) = \comu[\alpha_{\ese}]$, or simply a one-variable polynomial algebra.

\medskip

\medskip
In this paper, we choose {\color{black}{for realization $ \mathfrak{h} $ of $ (W,S) $
the {\it dual} of the geometric representation. 
To be precise, we choose}}
$ \mathfrak{h} $ to be the $  \comu$-vector space of dimension two, containing
an element $ \alpha^\vee_\ese = -  \alpha^\vee_\te $,  
such that for a basis 
$ \alpha_\ese,  \alpha_\te $ for $ \mathfrak{h}^{\ast} $ the relations (\ref{equations realization})
hold.
For this choice of $ \mathfrak{h} $ we have that
\begin{equation}\label{seconddefR}
R = S(\mathfrak{h}^\ast) =  \comu[\alpha_{\ese }, \alpha_\te ]
\end{equation}  
that is $ R$ is a two-variable polynomial algebra. We shall use the identifications
\begin{equation}
x := \alpha_{\ese}, \, y:=\alpha_\te 
\end{equation}
such that the two definitions of $ R $ in
\eqref{seconddefR} and \eqref{defR} coincide, 
{\color{black}{but shall in general use the symbols $ x, y $ when referring to 
$ \BRn$ and $ \alpha_{\ese}, \alpha_\te $ when referring to Soergel calculus}}.
    
\medskip
We consider $ R $ to be a $\Z$-graded algebra where
$ \deg(\alpha_{\ese }) = \deg(\alpha_\te) =2 $.
The action of $ W $ on
{\color{black}{$ \mathfrak{h}^{\ast} $ is given by
the formulas}}
\begin{equation}\label{W-action}
\ese   \alpha_\ese = -\alpha_\ese,   \, \, \, \,  \ese \alpha_\te = \alpha_\te + 2 \alpha_\ese,     \, \, \, \,
\te   \alpha_\te = -\alpha_\te,   \, \, \, \,  \te \alpha_\ese = \alpha_\ese + 2 \alpha_\te {\color{black}{.}}
\end{equation}
{\color{black}{It}}
extends to an action of $ W $ on $ R $ and so
we have 
the \emph{Demazure operators} $\partial_\ese , \partial_\te :R \rightarrow R$
{\color{black}defined by:}
\begin{equation}
  \partial_\ese (f) = \frac{f-\ese f}{\alpha_\ese},  \, \, \,\, \, \,\, \, \,   \partial_\te (f) = \frac{f-\te f}{\alpha_\te}
  {\color{black}{.}}
\end{equation}
We have that
\begin{equation}
  \partial_\ese ( \alpha_{\ese})= \partial_\te ( \alpha_{\te})= 2, \, \, \,\, \, \,\, \, \,
    \partial_\ese ( \alpha_{\te})= \partial_\te ( \alpha_{\ese})= -2 {\color{black}{.}}
\end{equation}

Let us now briefly explain the definition of the diagrammatic Soergel category $ \tildeSoergelcat $ for
our choices. 
Let $\Exp $ be the set of 
{\it expressions} over $ S $, that is {\it words}
$ \underline{w}= s_{i_1} s_{i_1} \cdots s_{i_N} $ in the alphabet $ S $.
We consider the empty expression $ \underline{w} = \emptyset$ to be an element of $\Exp $.

\begin{definition}
  Let $ (W,S) $ be as above.
  A Soergel diagram for $ (W,S) $ is a finite graph 
  embedded in $\mathbb{R}\times [0,1]$.
  The arcs of a Soergel diagram are coloured by either colour $\color{red}{red} $
  or colour $\color{blue}{blue}$,
  corresponding to the elements of $S$.
The vertices of a Soergel diagram are of the
{\color{black}{four}} possible types indicated below, univalent vertices (dots) and
trivalent vertices where all three incident arcs are of the same colour. 
 \begin{equation}\label{Vertices}
\raisebox{-.5\height}{\includegraphics[scale=1.0]{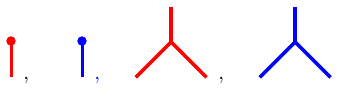}}   
  \raisebox{-15\height}{.}
 \end{equation}

\noindent
A Soergel diagram has its regions, that is the connected components of the
complement of the graph in $\mathbb{R}\times [0,1] $,
decorated by elements of $R$. 
For simplicity, we omit the decoration $ 1 \in R $ when drawing
Soergel diagrams. 
\end{definition}

\noindent
{\color{black}{A vertex of an arc of a Soergel diagram that belongs to}} the boundary of the strip $\mathbb{R}\times [0,1] $
is called a \emph{boundary point}.
{\color{black}{We say that an arc $ l $ of $ D $ is a \emph{boundary dot arc} if one of its vertices
is a dot and the other one is a boundary point.}}
The left to right reading of the boundary points gives rise to two elements of $ \Exp$ 
called the \emph{bottom boundary} and \emph{top boundary} of the diagram, respectively.

\begin{definition}\label{defin endo BS}
  The diagrammatic Soergel category $\tildeSoergelcat$ is the
  monoidal category whose objects are the elements of $ \Exp$ and whose morphisms 
    $ \botts{x}{y} $ are the $R$-modules generated by all Soergel
  diagrams with bottom boundary $\underline{x}$ and top boundary $\underline{y}$, modulo
isotopy and modulo the following local relations

\begin{equation} \label{sgraphA}
  \raisebox{-.5\height}{\includegraphics[scale=1.0]{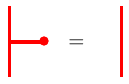}}   
\end{equation}

 \begin{equation}   \label{sgraphB}
\, \, \, \,\, \,   \raisebox{-.5\height}{\includegraphics[scale=1.0]{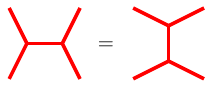}}   
 \end{equation}

\begin{equation}   \label{sgraphC}
 \raisebox{-.5\height}{\includegraphics[scale=1.0]{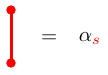}}
\end{equation}

\begin{equation}   \label{sgraphD}
    \raisebox{-.5\height}{\includegraphics[scale=1.0]{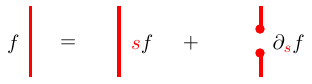}}
\end{equation}

 \begin{equation}   \label{sgraphE}
       \raisebox{-.5\height}{\includegraphics[scale=1.0]{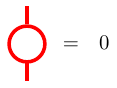}}
\end{equation}

 \medskip
 \noindent
where the relations {\rm (\ref{sgraphA})--(\ref{sgraphE})} also hold if red is replaced by blue.

\medskip
For $ f \in R $ and $ D $ a Soergel diagram, 
the product $ f D$ is defined as the diagram obtained
from $ D $ by replacing the polynomial $ g$ of the leftmost region of $D$ by $fg$.
The multiplication {\color{black}{$ D_1 D_2 $}}
{\color{black}{of diagrams $ D_1 $ and $ D_2 $}}
is given by vertical concatenation 
{\color{black}{with $ D_2 $ on top of $ D_1$}}, where regions containing two polynomials are
replaced by the same regions, but containing the product of these polynomials.
The monoidal structure is given by horizontal concatenation.
There is a natural $ \Z$-grading on
$ \tildeSoergelcat $, extending the
grading on $ R $, in which dots, that is the first two diagrams in (\ref{Vertices}) have degree $ 1$,  and
the trivalents, that is 
the last two diagrams in (\ref{Vertices}), have degree $-1$.
\end{definition}

\begin{remark}\rm\label{cyclotomic}
  For more details concerning the definition 
  %\cancel{\color{black}{of \cite{LPR}}}
  of $ \tildeSoergelcat$ one should consult
  \cite{LPR} or the original paper \cite{EW}. Note that 
  apart from the choice of realization of $ (W,S)$, the relations appearing in 
  Definition (\ref{defin endo BS}) also 
  differ slightly from the ones appearing in the corresponding Definition 3.2 in \cite{LPR}. To be precise, 
  in Definition 3.2 of \cite{LPR}, there is a final relation
 \begin{equation}   \label{sgraphfian}
       \raisebox{-.5\height}{\includegraphics[scale=1.0]{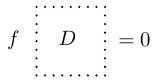}}
\end{equation}
 where $ D$ is any Soergel diagram and $ f  $ is any homogeneous polynomial of 
 strictly positive degree, multiplied on the left on $D$. 
\end{remark}

\begin{definition}\label{defin endo BS tilde}
We define $ \tildeSoergelcatC $ to be the category
obtained from $ \tildeSoergelcat $ by adding relation \eqref{sgraphfian}. 
\end{definition}  

\begin{remark}\rm
{\color{black}{It is shown in \cite{EW} that there is an equivalence between 
$ \tildeSoergelcat $ and the category of Soergel bimodules for $ (W,S) $. It induces an equivalence between 
$ \tildeSoergelcatC $ and the category of Soergel modules. }}
\end{remark}
  
Let now $ n $ be a fixed non-negative integer and define $\underline{w}
\in \Exp $ of length $ n $ via 
\begin{equation}\label{furthernotice}
  \underline{w}  := {\color{red} s} {\color{blue} t} {\color{red} s} \cdots  \, \, \, \, \,   
\end{equation}  
such that the first generator of
$  \underline{w}$
is ${\color{red} s} $ but the last generator
depends on the parity of $n$.   
We then introduce $ \tilde{A}_w $ as follows 
\begin{equation}\label{dfinAw}
  \tilde{A}_w :=   \mbox{End}_{\tildeSoergelcat} (\underline{w}) {\color{black}{.}}
\end{equation}
By the definitions, $ \tilde{A}_w $
is an $ R $-algebra with multiplication $ D_1 D_2 $ given by concatenation of $ D_1 $ and $D_2$ and
scalar product $f D $ by multiplication of $ f $ with the polynomial appearing in the leftmost region
of $D$. Its one-element is denoted $1$, and is as follows 
\begin{equation}
  1 := \quad     \raisebox{-.4\height}{\includegraphics[scale=0.8]{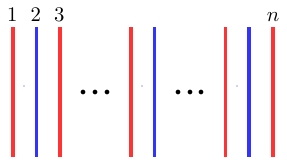}}
    \raisebox{-15\height}{ {\color{black}{.}}}
\end{equation}

\medskip

As in \cite{LPR} we introduce certain elements of $ \tilde{A}_w $ that shall play a key role
throughout.
For $ i =1,\ldots, n-2 $, let $ U_i $ 
be the following element of $ \tilde{A}_w $
\begin{equation}
     U_i := \quad     \raisebox{-.4\height}{\includegraphics[scale=0.8]{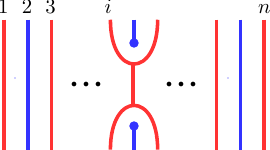}}
\end{equation}
\noindent
and for $ i=0$, set 
\begin{equation}
     U_0 := \quad     \raisebox{-.4\height}{\includegraphics[scale=0.8]{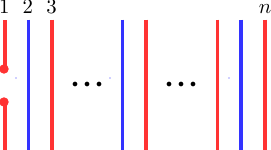}}
    \raisebox{-15\height}{ {\color{black}{.}}}
\end{equation}

\section{Isomorphism Theorems}\label{isomorphism theorems}
The following Theorem is fundamental {\color{black}{for this paper}}. 
\begin{theorem} \label{fundamental}
  There is a homomorphism of $ R $-algebras $ \varphi: \BRnn \rightarrow  \tilde{A}_w$ given 
  by $\mathbb{U}_i\mapsto U_i $ for $ i = 0, 1, \ldots, n-2$.
\end{theorem}
\begin{dem}
 The proof is almost identical to the proof of Theorem 3.4 in \cite{LPR}.
  We must check that $U_0, U_1, \ldots , U_{n-2} $ satisfy the relations
  given by the $ \mathbb{U}_i $'s in Definition (\ref{deftwoblob}).
The verification of the relations \eqref{eq one}, \eqref{eq two} and
\eqref{eq three} is done exactly as in \cite{LPR} whereas 
relation \eqref{eq four}, for example for $ n=4 $, is verified as follows 

\medskip 
\begin{equation}
U_1 U_0 U_1 \, \, = \, \,
\raisebox{-.45\height}{\includegraphics[scale=0.6]{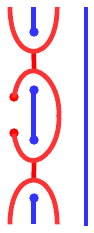}}  
\, \, \, \,= \, \, \,
\raisebox{-.45\height}{\includegraphics[scale=0.6]{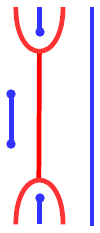}}  
\, \, \, \,= \, \, \,
\raisebox{-.45\height}{\includegraphics[scale=0.6]{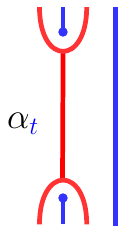}}
 \, \, = \, \,{ \color{black}{ y  U_1 }}
\end{equation} 
and \eqref{eq five}, for example for $ n=4 $, is verified as follows 

%\LARGE $  \alpha_{ {\color{red}t }} $ 

\medskip 
\begin{equation}
U_0^2 \, \, = \, \,
\raisebox{-.45\height}{\includegraphics[scale=0.6]{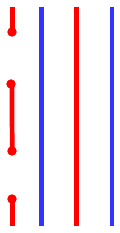}}  
\, \, \, \,= \, \, \,
\raisebox{-.45\height}{\includegraphics[scale=0.6]{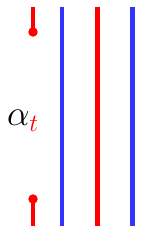}}  
 \, \, = \, \, {\color{black}{  x U_0}}.
\end{equation} 
This proves the Theorem.
\end{dem}

\medskip
For general $(W,S)$ there is a, somewhat unwieldy,  
recursive procedure for constructing
an $ R$-basis for the morphisms  $\botts{x}{y} $, for any
$ \underline{x}, \underline{y} \in \Exp$. It is a diagrammatic
version of Libedinsky's {\it double
  %\cancel{\color{black}{light}}
  leaves basis}
for Soergel bimodules and the basis elements are
also called double
%\cancel{\color{black}{light}}
leaves. 
For $ W $ the infinite dihedral group, 
there is however a non-recursive description of the double
%\cancel{\color{black}{light}}
leaves basis that was used
extensively in \cite{LPR}, and that also plays an important role in the present paper.

\medskip
The double leaves diagram basis elements for $\tilde{A}_{w}$ 
are built up from top and bottom `half-diagrams', similarly to Temperley-Lieb and
blob diagrams. 
Let us explain
these half-diagrams. 

\medskip
    {\color{black}{By the isotopy relations for Soergel diagrams, we have the following diagram identities
\begin{equation}\label{33} 
  \raisebox{-.43\height}{\includegraphics[scale=0.8]{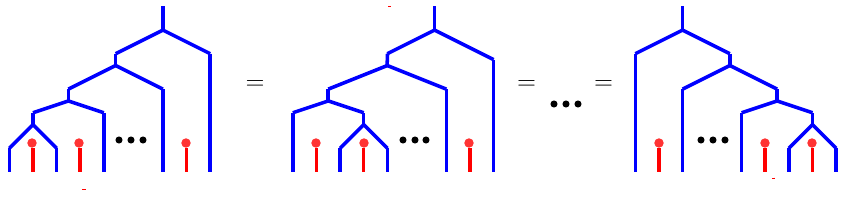}}      
\end{equation} 
and we in shall in general represent these diagrams as follows
\begin{equation}\label{34}
  \raisebox{-.43\height}{\includegraphics[scale=0.8]{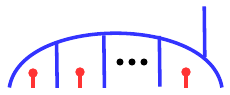}}      
    \raisebox{-12\height}{ {\color{black}{.}}}
\end{equation} 
The diagrams in \eqref{34} are called \emph{hanging full birdcages}. We shall also consider 
\emph{non-hanging full birdcages} that look as follows
    }}
\begin{equation}\label{35}
  \raisebox{-.43\height}{\includegraphics[scale=0.8]{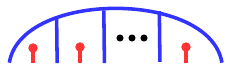}}      
    \raisebox{-8\height}{ {\color{black}{.}}}
\end{equation}
{\color{black}We sometimes omit the word 'full' when referring to 'full birdcages'.}
{\color{black}{We say that the birdcages in 
    \eqref{33}, \eqref{34} and \eqref{35} are of \emph{colour blue}, but shall
    {\color{black}{also}} allow birdcages of
    \emph{colour red}, as follows
\begin{equation}\label{36}
  \raisebox{-.43\height}{\includegraphics[scale=0.8]{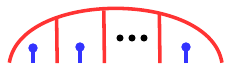}}      
    \raisebox{-8.5\height}{{\color{black}{.}}}
\end{equation}}}
{\color{black}{We further allow \emph{degenerate}, non-hanging and hanging, full birdcages as follows}}
\begin{equation}
\raisebox{-.20\height}{\includegraphics[scale=0.8]{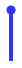}}\, ,  \, \,  \, \,   \, \,  \, \,  \, \, 
\raisebox{-.11\height}{\includegraphics[scale=0.8]{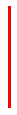}} \, \,  \, \, .
\end{equation}
We define the \emph{length} of a full birdcage to be the number of {\color{black}{\emph{enclosed}}}
dots in it, where 
the birdcages of length {\color{black}{zero}} are the degenerate ones. 
A full birdcage which is not degenerate is called
\emph{non-degenerate}.
{\color{black}{In the following examples, the first birdcage is
    of length 4 whereas the last two are {\color{black} of length zero}.
\begin{equation}
  \raisebox{-.20\height}{\includegraphics[scale=0.8]{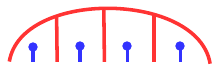}}\, ,  \, \,  \, \,   \, \,  \, \,  \, \, 
  \raisebox{-.20\height}{\includegraphics[scale=0.8]{jaulaEx2.pdf}}\, ,  \, \,  \, \,   \, \,  \, \,  \, \, 
  \raisebox{-.15\height}{\includegraphics[scale=0.8]{jaulaEx3.pdf}}
    \raisebox{-1\height}{ {\color{black}{.}}}
\end{equation} 
}}
We also consider \emph{top full birdcages}, that are obtained from
bottom full birdcages by reflecting through a horizontal axis.

\medskip
{\color{black}{We now introduce the operation of}} 
replacing a degenerate non-hanging full birdcage, {\color{black}{in other words a boundary dot arc}}, by a non-hanging
non-degenerate full birdcage as follows

\begin{equation}\label{insertionStepFirst}
\! \! \! \! \! \! \! \! \! \! \! \! \! \!
 \raisebox{-.43\height}{\includegraphics[scale=0.8]{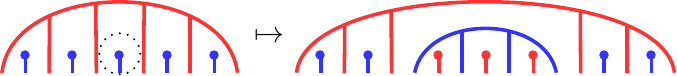}}      
    \raisebox{-11\height}{ {\color{black}{.}}}
\end{equation}

{\color{black}{In the notation of \cite{LibedinskyGentle}}}, a \emph{birdcagecage}
is any diagram
that can be obtained by performing 
the above operation repeatedly a number of times, 
for example

\begin{equation}\label{insertionStep}
   \raisebox{-.43\height}{\includegraphics[scale=0.8]{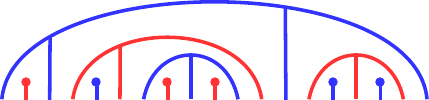}}      
    \raisebox{-13\height}{ {\color{black}{.}}}
\end{equation}

A {\it light leaf diagram} element $ D $ for
$\tilde{A}_{w}$
is by definition a horizontal concatenation of 
birdcagecages as indicated below in (\ref{birdcagecagezonesA}),
with bottom boundary $ \underline{w}$. 
Zone A consists of a number of non-hanging birdcagecages whereas 
zone B consists of a number of hanging birdcagecages, but 
zone C
{\color{black}{may consist}} of at most one 
non-hanging birdcagecage.

\begin{equation}{\label{birdcagecagezonesA}}
 \! \!  \! \! \!  \! \! \!  \! \! \!
  \raisebox{-.43\height}{\includegraphics[scale=0.7]{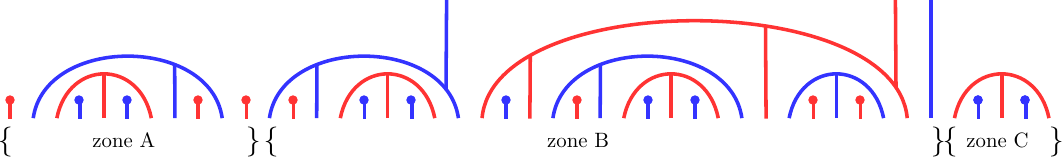}}      
    \raisebox{-11\height}{ {\color{black}{.}}}
\end{equation}

\medskip
The 
hanging birdcagecages in zone B of a light leaves diagram define an element $ v \in W $.
In the
above example we have $ v = \color{blue} t \color{red} s \color{blue} t$. 
The \emph{double leaves basis of $ \tilde{A}_w $} is obtained by running over all $ v \le w  $
and over all pairs of light leaves for $  \tilde{A}_w $ that are 
associated with that $ v $. For each such pair $ (D_1, D_2) $ 
the second component $ D_2 $ is reflected through a horizontal axis, and the two
components are glued together. The resulting diagram is a double leaf, 
for example

\begin{equation}{\label{birdcagecagezonesB}}
 \! \!  \! \! \!  \! \! \!  \! \! \!
  \raisebox{-.43\height}{\includegraphics[scale=0.7]{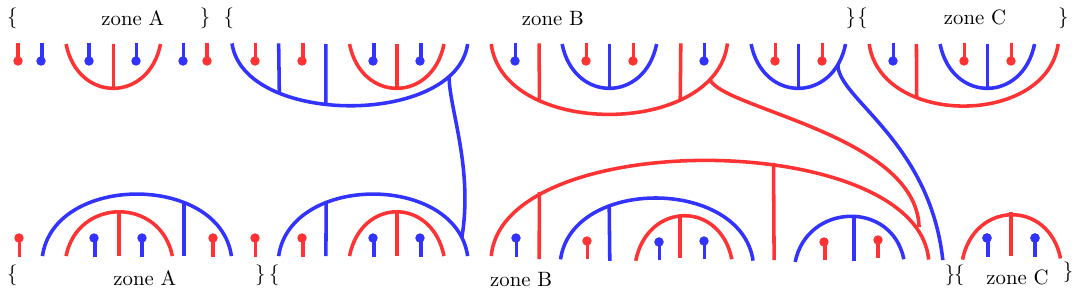}}      
    \raisebox{-30\height}{ {\color{black}{.}}}
\end{equation}

\medskip
The fundamental result concerning double leaves is the fact that they form
an $R$-basis for $\tilde{A}_w$. In fact, a stronger Theorem holds.
Let $ \tilde{\Lambda}_{{w}} := \{ v \in W | v \le w \} $, endowed with poset structure via the Bruhat order $ < $.
For $ v \in  \tilde{\Lambda}_{w} $ let 
$ {\tab}_w(v )  $ be the set of light leaves for $ \tilde{A}_w $ defining $ v$ in the above sense,
and for $ D_1, D_2  \in {\tab}_w(v ) $
let $ C_{D_1, D_2}^{v} \in \tilde{A}_w $ be the double leaf obtained by gluing as above.
We now have the
following Theorem. 
\begin{theorem}\label{thenwehavethe}
  The triple $ ( \tilde{\Lambda}_{{w}}, {\tab}_w(v), C) $
  defines a cellular basis structure on $ \tilde{A}_w$. 
\end{theorem}  
\begin{dem}
See \cite{LibedinskyGentle} and \cite{EW}.
\end{dem}

\medskip
Following \cite{LPR},  we now introduce $ A_w \subseteq \tilde{A}_w$
as follows.

\begin{definition}\label{defAw} 
Let $ A_w $ be the span in $ \tilde{A}_w$ of all
double leaves {\color{black}{with}} empty zone C, {\color{black}{or equivalently, $ A_w $ is the free $ R$-module
    with basis given by the double leaves of empty zone C.}}
\end{definition}

Our next Theorem is an analogue of Theorem 3.8 and Corollary 3.9 of \cite{LPR},
%% {\color{green}{\sout{but 
%% note that the proofs in \cite{LPR} that depend on dimension arguments in linear algebra
%% must be modified}}},
{\color{black}{although the proofs of parts ${\bf b)}$ and ${\bf c)}$ of the Theorem are different
    from the proofs of the corresponding statements in \cite{LPR},
    since the algebras considered in \cite{LPR} are $ \CC$-algebras
    whereas the algebras in the present paper are $R$-algebras. 
    Therefore, in the present article some extra care is necessary since nonzero coefficients of $ R $ need not be invertible. 
    Moreover, the arguments in \cite{LPR} depend on the linear algebra fact 
    that injective linear transformations $ f: V \rightarrow W $
    between vector spaces $ V $ and $ W $ of the same finite dimension are isomorphisms.
    The analogous statement 
    is false for free $ R$-modules, which is the main reason why the proofs of the present paper are
    different from the ones in \cite{LPR}.}}
%% {\color{green}{\sout{
%%       Note also that unlike \cite{LPR}, we avoid using JM-elements.}}}
\begin{theorem}\phantomsection\label{mainTheoremSection3new}
  \begin{description}
  \item[a)] The cardinality of double leaves of empty zone C
    is $ \binom{2n}{n}$ and so ${A}_w$ is
    free over $ R$ of {\color{black}{rank}} $  \rank   A_w  = \rank  \BRn = \binom{2n}{n}$.
  \item[b)] ${A}_w$ is an $ R$-algebra. It is the subalgebra of $\tilde{A}_w$
    generated by $U_0, U_1, \ldots, U_{n-2}$.
  \item[c)]  The homomorphism $ \varphi: \BRnn \rightarrow  \tilde{A}_w$ from
Theorem \ref{fundamental} induces an isomorphism 
$ \varphi: \BRnn \rightarrow  {A}_w$.
    \end{description}
 \end{theorem}
\begin{dem}
  {\color{black}{To show
      ${\bf a)}$ we first note that the cardinality of the set of double leaves of empty zone C 
      is given by Definition 3.7 in \cite{LPR} and Theorem 3.8 ${\bf c)}$ in \cite{LPR},
      and so the statements about $ A_w $ 
      are a direct consequence of Definition \ref{defAw} and \eqref{rankblob}.}}
%  {\color{green}{\sout{Part ${\bf a)}$ is Theorem 3.8 c) of \cite{LPR}.}}}

         \medskip
  In order to prove ${\bf b)}$ we define $ {A}_w^{\prime}$ to be
  the subalgebra of $  \tilde{A}_w $
  generated by $U_0, U_1, \ldots, U_{n-2}$ and must show that $  {A}^{\prime}_w  =  {A}_w $. 
 We first show that $  {A}^{\prime}_w  \supseteq {A}_w $. 

\medskip
 Recall the diagram $ C_{\lambda} \in \BRn$ from {\eqref{tlambda2}}. 
For $ \lambda \ge 0 $ we have that 
\begin{equation}{\label{Csoergel}}
\varphi(C_{\lambda}) = U_1 U_3 \cdots U_{2k-1} =   \raisebox{-.5\height}{\includegraphics[scale=0.8]{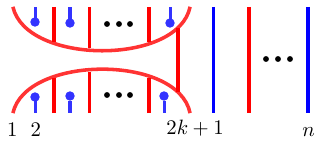}}      
\end{equation}
and for $ \lambda < 0 $ we have 
\begin{equation}{\label{CsoergelAB}}
 \begin{split} 
\varphi(C_{\lambda})  =  (U_1 U_3 \cdots U_{2k-1}) U_0 (U_2 U_4 \cdots U_{2k})   U_1 U_3 \cdots U_{2k-1} = \\
\raisebox{-.52\height}{\includegraphics[scale=0.8]{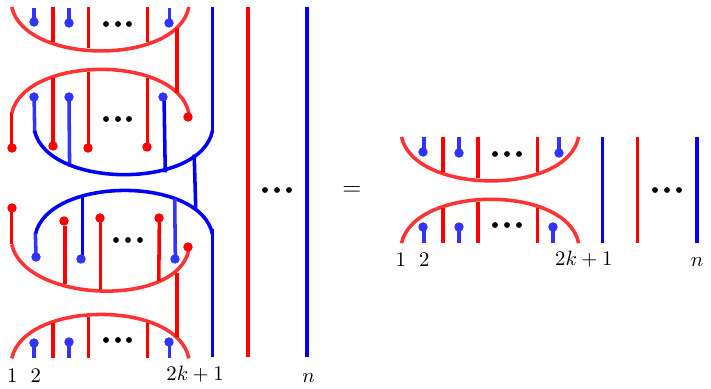}}
    \raisebox{-21\height}{ {\color{black}{.}}}
 \end{split} 
 \end{equation}
{\color{black}{We also need the diagrams
\begin{equation}{\label{CsoergelABC}}
  \varphi(U_i U_{i+2} \cdots U_{i+2k} ) =
   \raisebox{-.5\height}{\includegraphics[scale=0.8]{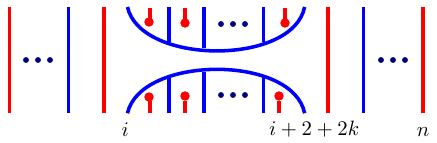}}      
    \raisebox{-17\height}{ {\color{black}{.}}}
\end{equation}
}}

\medskip
Now, multiplying together appropriate diagrams of the form {\eqref{Csoergel}} and of the form {\eqref{CsoergelAB}}
we deduce that any diagram of the form
\begin{equation}{\label{symm1}}
\raisebox{-.5\height}{\includegraphics[scale=0.7]{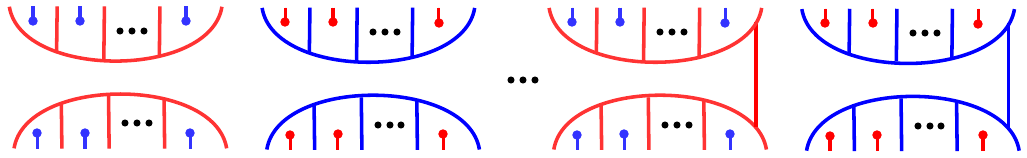}}
\end{equation}
belongs to $ {A}_w^{\prime}$
where the number of hanging birdcages on the right is at least one.
The number of non-hanging birdcages on top of {\eqref{symm1}} is the same
as on the bottom, but we need to break this symmetry, 
%% {\color{green}{\sout{This was achieved in \cite{LPR} using
%% the $ L_i$'s, but in view of the relations \eqref{sgraphD} and \eqref{sgraphE}
%% it can in fact also be achieved
%% via multiplication on top or bottom with appropriate diagrams of the form given in 
%% {\eqref{Csoergel}} and {\eqref{CsoergelAB}}. 
%% To illustrate this, we give the following example where $ n=10$. }}}
{\color{black}{that is, we must show that diagrams $ D $ as in 
{\eqref{symm1}}, but with unequal numbers of top and bottom
    non-hanging birdcages in zone A, also belong to $ {A}_w^{\prime}$. 
Note that the number of top and the number of
 bottom
 non-hanging birdcages in zone A are always of the same parity and so, in order
 to break this symmetry in zone A, we first give a procedure for 
 splitting any non-hanging and non-degenerate birdcage in zone in A in three non-hanging birdcages, and still stay in $ {A}_w^{\prime}$. 

\medskip
 If the non-hanging birdcage is the leftmost one, it can easily be split in three parts via multiplication
 by $ U_0 $, as illustrated below
 }}
\begin{equation}{\label{symm2}}
  \raisebox{-.5\height}{\includegraphics[scale=0.8]{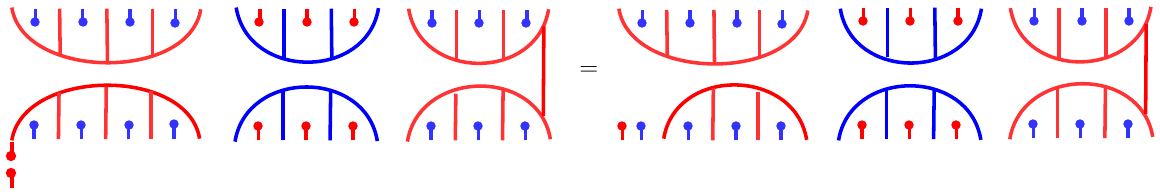}} 
 \raisebox{-17\height}{ .}
\end{equation}
    {\color{black}{Note that this belongs to
$ {A}_w^{\prime}$. 
In the non-hanging birdcage is not the leftmost one, 
we first notice that 
multiplication with appropriate $ U_i$'s has the effect of 'moving'
a dot from one birdcage to its neighbouring birdcage,
as illustrated below in the following two examples. 
\begin{equation}\label{neighbouring}
\begin{array}{l}
  \raisebox{-.5\height}{\includegraphics[scale=0.8]{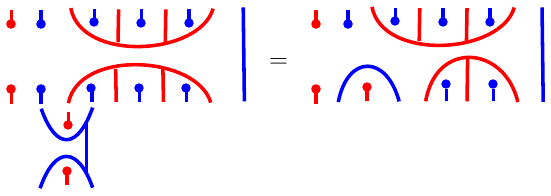}}, \\ \\
\raisebox{-.5\height}{\includegraphics[scale=0.8]{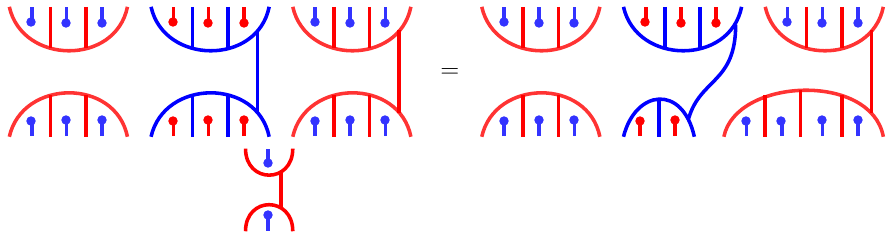}}
    \raisebox{-7\height}{ {\color{black}{.}}}
\end{array}
\end{equation}  
Using this we can also 'move' any non-hanging birdcage to the leftmost position
and then multiply it by $ U_0 $, to split it in three birdcages. Next, we 
use \eqref{neighbouring} to move the birdcages to the desired positions. 
The result of this belongs to $ {A}_w^{\prime}$ and so the symmetry in zone A has been broken.

In fact, once we have the right number of birdcages in zone A, we can 
use successive multiplications of the types given in \eqref{neighbouring},
to obtain any combination of desired birdcage lengths, as opposed to
\eqref{symm2} that always produces degenerate birdcages. 

    }}

    \medskip
{\color{black}{    
Now in \eqref{symm1}, we must also break the symmetry with respect
to birdcage lengths in zone B. 
But in order to break this symmetry we proceed just as we did in the last step for zone A, applying
\eqref{neighbouring} successively, adjusting
the lengths of the relevant birdcages until they are as desired. }}

\medskip
All in all we have now shown that any diagram of the form described in \eqref{318}
belongs to $ {A}_w^{\prime}$, where the number and the lengths of the top and bottom birdcages in zone A
may differ, as may the lengths of birdcages in zone B. 

\begin{equation}\label{318}
  \raisebox{-.5\height}{\includegraphics[scale=0.7]{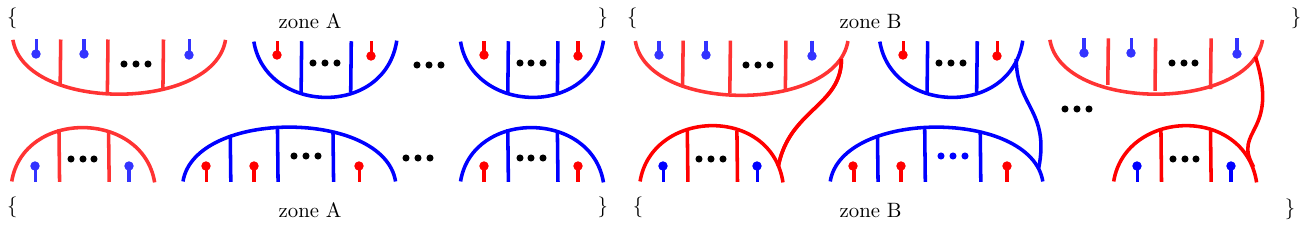}}
\end{equation}

\medskip

Finally, to conclude the proof of $  {A}^{\prime}_w  \supseteq {A}_w $, we observe that the process, 
illustrated in \eqref{insertionStepFirst},  
of replacing a degenerate non-hanging full birdcage by a non-hanging
non-degenerate full birdcage 
can be realized as the multiplication on top or bottom with a diagram
of the form {\eqref{CsoergelABC}}. Below we give an example. 
\begin{equation}\label{belowwegive}
  \raisebox{-.5\height}{\includegraphics[scale=0.8]{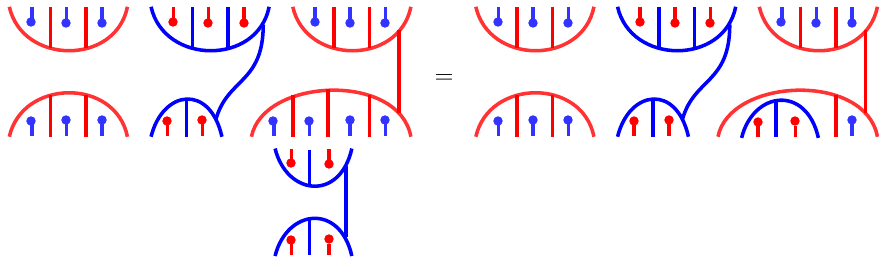}}
    \raisebox{-2\height}{ {\color{black}{.}}}
\end{equation}

\medskip
In order to prove the other inclusion $  {A}^{\prime}_w  \subseteq {A}_w $
we need to change the argument of the corresponding statement in \cite{LPR}
since it, {\color{black}{as already mentioned}}, depends on dimension arguments in linear algebra (over $ \CC$). 

We first observe that
$ U_i \in {A}_w $ for $ i =0,1,\ldots, n-2 $.
{\color{black}{Hence, to show $  {A}^{\prime}_w  \subseteq {A}_w $}}
it is enough to verify that $ {A}_w $ is invariant 
under left and right multiplication by the $ U_i$'s.

{\color{black}{Let therefore $ D$ be a diagram for $  {A}_w $.
    We first choose $ i > 0 $ and proceed to give a description of
the effect of
multiplying $ U_i   $ below on $ D $, that is we describe $ U_i  D$ in terms of $ D$.
Let $ l_1,  l_2 $ and $ l_3 $ be the arcs in $ D $ that has bottom boundary points $ i ,  i+1 $
and $ i+2$ of $ D$, respectively. Without loss of generality we may assume
that $ l_1 $ and $ l_3 $ are red,
and that $ l_2 $ is blue. {\color{black}{Recall}} that an arc $ l $ of $ D $ is a boundary dot arc if
one of its vertices
is a dot and the other one is a boundary point.

\medskip
\noindent
    {\bf Case 1:} This is the case where $ l_1 $ and $ l_2 $ have a common vertex in $D$,
    or, more generally, that $ l_1 $ and $ l_2 $ are connected in $ D$. We then have that 
    $ l_2 $ is a boundary dot arc. 
In now follows from the isotopy relations
\eqref{33} and the relation $ U_i^2 = -2 U_i $, see Theorem
\ref{fundamental}, 
that $ U_i D = -2 D$,
since relation \eqref{sgraphD} implies that the scalar $ -2 $ moves to the left of $ D$. 
Hence we get that 
$ U_i D \in A_w$. 
Here is an illustration, where we for simplicity leave out the top
part of the diagrams. 
\begin{equation}
  \raisebox{-.5\height}{\includegraphics[scale=0.8]{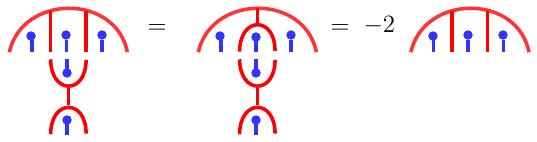}}
    \raisebox{7\height}{ {\color{black}{.}}}
\end{equation}

\medskip
\noindent
{\bf Case 2:} 
This is the case where $ l_2 $ is a boundary dot arc, whereas $ l_1 $ and $ l_3 $ are the rightmost and leftmost
arcs of birdcagecages, respectively.
Note that in this situation the bottom boundary points $ i  $ and $ i+1 $ belong to zone A
of $ D$. 
We here get that $ U_i D = \alpha_{\te} D_1 $ where
$ D_1 $ is the diagram obtained from $ D$ by concatenating horizontally the two birdcagecages, with
an extra dot in the middle. Since $ D_1 \in A_w $, this case is also OK. 
Here is an illustration, once again without the top parts of the diagrams. 
\begin{equation}
  \raisebox{-.5\height}{\includegraphics[scale=0.8]{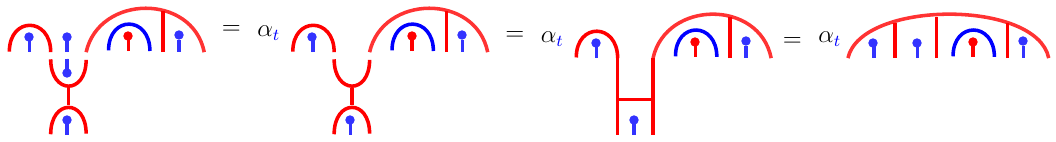}}
    \raisebox{5\height}{ {\color{black}{.}}}
\end{equation}

\medskip
\noindent
    {\bf Case 3:} This is the case where $ l_2 $ has a trivalent vertex, whereas
    $ l_1 $ and $ l_3 $ are the rightmost and leftmost arcs of birdcagecages, respectively.
    In this case we get $ U_i D =  D_1 $ where $ D_1 $ is obtained from $ D $ by eliminating
    $ l_2 $ and joining the birdcagecages involving $ l_1 $ and $ l_3 $, with an extra dot in the middle.
    Once again,
    we get that $ D_1 \in A_w  $, and so this case is also OK. Here is an illustration. 
\begin{equation}
  \raisebox{-.5\height}{\includegraphics[scale=0.8]{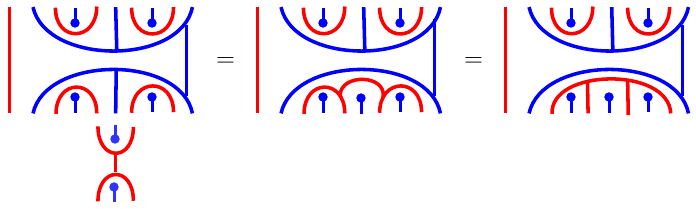}}
    \raisebox{-2\height}{ {\color{black}{.}}}
\end{equation}

\medskip
\noindent
    {\bf Case 4:} This is the case where $ l_1 $ is the leftmost arc of a birdcagecage, whereas $ l_2 $
    has a trivalent vertex $V$. Let $ B_1 $ be the birdcagecage that lies to the southwest of $ V$ and
    let $ B_2 $ be the birdcagecage that lies to the southeast of $V$. Then $ U_i D = D_1 $ where
    $ D_1 $ is obtained from $ D$ by joining $ B_1 $ and $ l_1 $ and adding a boundary dot arc to the
    left of $ B_1 $. We have that $ D_1 \in A_w $ and so this case is done. Here is an example, without
    the top parts of the diagrams. 
\begin{equation}
  \raisebox{-.5\height}{\includegraphics[scale=0.8]{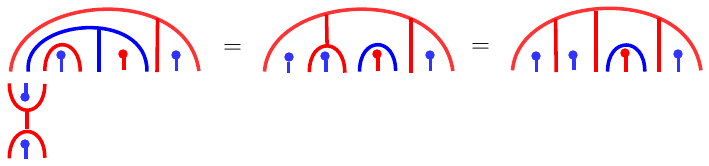}}
      \raisebox{3.5\height}{ {\color{black}{.}}}
\end{equation}

\medskip
\noindent
    {\bf Case 5:} This is the identical to case 4 except that both vertices of $l_2 $ are supposed
    to be boundary points. Let $ B $ be the birdcagecage that lies below $ l_2 $. Then
    $ U_i D = D_1 $ where $ D_1 $ is obtained from $ D $ by splitting $ l_2 $ in two dot boundary arcs,
    and $ B $ is joined with $l_1$. Since $ D_1 \in A_w $ we are done in this case, as well. Here is an example. 
\begin{equation}
  \raisebox{-.5\height}{\includegraphics[scale=0.8]{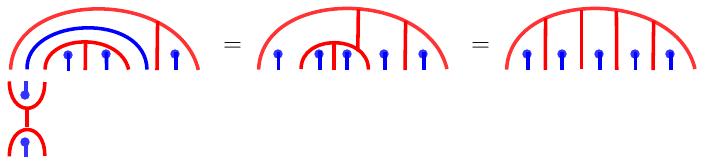}}
        \raisebox{3.5\height}{ {\color{black}{.}}}
\end{equation}

\medskip
\noindent
    {\bf Case 6:} This is the case where
a vertex of 
    $ l_2 $  is a top boundary point of $ D$. Then
    $ l_1 $ and $ l_2 $ are the rightmost and leftmost arcs of two birdcagecages $ B_1 $ and $B_2 $.
    We get that $ U_i D = D_1 $ where $ D_1 $ is obtained from $ D $ by joining $ B_1 $ and $ B_2 $
    and splitting $ l_2 $ in two dot boundary arcs. Since $ D_1 \in A_w $ we have that this case is OK as well. 
    Here is an example. 
\begin{equation}
  \raisebox{-.5\height}{\includegraphics[scale=0.8]{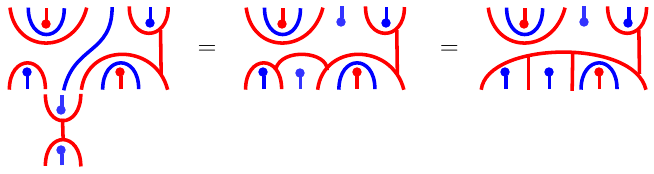}}
        \raisebox{-2\height}{ {\color{black}{.}}}
\end{equation}

\medskip
\noindent
    {\bf Case 7:} This is the case where $ l_2 $
is the leftmost arc of a hanging birdcagecage. 
This case resembles case 6.
We have that 
$ l_1 $ and $ l_2 $ are the rightmost and leftmost arcs of two birdcagecages $ B_1 $ and $B_2 $.
Then we get that $ U_i D = D_1 $
where $ D_1 $ is obtained from $ D $ by joining $ B_1 $ and $ B_2 $
and replacing $ l_2 $ by a bottom dot boundary arc. We have $ D_1 \in A_w $
and so this case is OK as well. Here is an example. 
\begin{equation}
  \raisebox{-.5\height}{\includegraphics[scale=0.8]{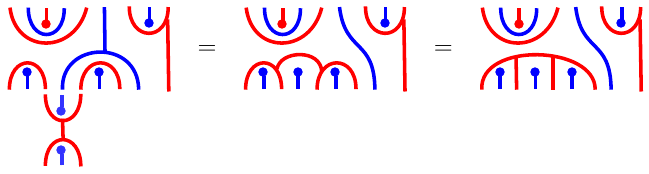}}
      \raisebox{-2\height}{ {\color{black}{.}}}
\end{equation}  

\medskip
There are a few remaining cases to consider, but they are all small variations of the cases already studied 
and so we leave them to the reader.

\medskip
We next consider $ i =0 $. Let $ B $ be the leftmost bottom birdcagecage of $ D $ and
let $ l $ be its leftmost bottom arc. If $ D $ is a boundary dot arc, we have $ U_0 D = 
\alpha_{\ese} D \in A_w  $. Otherwise
we have $ U_0 D =  \alpha_{\ese} D_1  $ where $ D_1 $ is
the birdcagecage obtained from $ D $ by replacing $ l $ by a boundary dot arc, and here we also have 
$ U_0 D \in A_w $. Here is an illustration of this case. 
\begin{equation}
  \raisebox{-.5\height}{\includegraphics[scale=0.8]{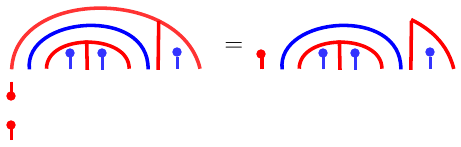}}
      \raisebox{1\height}{ {\color{black}{.}}}  
\end{equation}

\medskip
Finally, we observe that the 
description of the right multiplication $ D U_i $ is completely analogous to the
description of $ U_i D  $, and so we have concluded the proof of 
${\bf b)}$. }}

\medskip
{\color{black}{We now give an alternative proof of $  {A}^{\prime}_w  \subseteq {A}_w $, 
adapting the proof in \cite{LPR} and using one of the main results in \cite{EW}. 
Let $ Q $ be the quotient field of $ R $ and let $ Q{A}_w := Q \otimes_R {A}_w $, 
$ Q{A}^{\prime}_w := Q \otimes_R {A}^{\prime}_w $ and
$ Q \BRnn = Q \otimes_R  \BRnn$. 
Since $ A_w $ and $ \BRnn $ are torsion free
$R$-modules, in fact even free, 
we may view $ A_w $ and $ \BRnn$ as $R$-submodules
of $ Q{A}_w $ and $ Q\BRnn$, via the map $ D \mapsto 1 \otimes_R D $. 
Similarly, we may view $ A^{\prime}_w $ as an $R$-submodule
of $ Q{A}^{\prime}_w $, since $ {A}^{\prime}_w $ is torsion free, being a 
submodule of the free $ R$-module $  \tilde{A}_w $.

\medskip
Now the inclusion $  {A}_w  \subseteq {A}^{\prime}_w $
induces an inclusion $  Q {A}_w \subseteq Q {A}^{\prime}_w $, since $ Q \otimes_R ( \cdot) $
is an exact functor, and the surjection $ \varphi: \BRnn \twoheadrightarrow  A^{\prime}_w $
induces a surjection $  Q \BRnn \twoheadrightarrow  Q A^{\prime}_w $. Combining this with 
${\bf a)}$, we deduce that $  Q {A}_w = Q {A}^{\prime}_w $ since both 
are $Q$-vector spaces of the same dimension $\binom{2n}{n}$. 

\medskip
Let us now show $  {A}^{\prime}_w  \subseteq {A}_w $. As in the first proof
of $  {A}^{\prime}_w  \subseteq {A}_w $, it is for this 
enough to check
that $ U_i D \in  {A}_w $, whenever $ D $ is a light leaves diagram in $  {A}_w $.
Now using $  Q {A}_w = Q {A}^{\prime}_w $, we find elements $ q_k \in Q$ such that 
\begin{equation}\label{compara1}
U_i D = \sum_{ D_k \in A_w} q_k D_k  
\end{equation}
where $ D_k $ runs over the light leaves basis for $ A_w$. On the other hand,
as was shown in Theorem 6.11 of \cite{EW}, the light leaves diagrams $ \{ D_l \} $ for $   \tilde{A}_w $ form an $ R $-basis
for $   \tilde{A}_w $, and so there also exist 
$ r_l \in R $ such that 
\begin{equation}\label{compara2}
 U_i D = \sum_{ D_l  \in   \tilde{A}_w } a_l  D_l .
\end{equation}
Comparing \eqref{compara1} and \eqref{compara1} we deduce that $ q_k \in R $, and so $ U_i D \in A_w $,
as claimed. 

\medskip
We remark that, in a suitable sense, the alternative
proof of ${\bf b)}$ is equivalent to the first proof, since the arguments 
in \cite{EW}, that hold in the setting of a general Coxeter {\color{black}{system}} $(W,S) $, depend on a case-by-case similar to the one carried out in our first proof.}}

\medskip
To prove ${\bf c)}$ we argue essentially as in
\cite{LPR},
although a little extra care has to be exercised since the algebras are defined over a
commutative ring rather than a field. But by \eqref{rankblob}, ${\bf a)}$ and ${\bf b)}$
  we have that $ \varphi $ is a surjective homomorphism between free $R$-modules of the same
  finite {\color{black}{rank}}.
On the other hand, $ R $ is a commutative ring with 1 and 
so indeed $ \varphi $ is an isomorphism, 
as follows from Vasconcelos' Theorem, see \cite{Vasconcelos}.
The proof of the Theorem is finished. 
\end{dem}

\medskip

Suppose that $ \underline{w}= s_{i_1} s_{i_2} \cdots s_{i_n} $. Then
we define $ {\Lambda}_{w} \subseteq \tilde{\Lambda}_{w}  $
as the subset of all 'tails' of $ \tilde{\Lambda}_{w} $, that is 
\begin{equation}\label{tail}
  {\Lambda}_{w} = \{ v \in \tilde{\Lambda}_{w} | v = s_{i_k} s_{i_{k+1}} \cdots s_{i_n} \mbox{ for } k =1,
  \ldots, n\} {\color{black}{.}}
\end{equation}
Note that $ w \in   {\Lambda}_{w} $ but $ 1 \not\in   {\Lambda}_{w} $. 
For example for $ n = 7 $ we have that $ w = {\ese} {\te}{\ese} {\te}{\ese} {\te}{\ese} $ and so 
\begin{equation}   {\Lambda}_{w} = \{ {\ese} {\te}{\ese} {\te}{\ese} {\te}{\ese},
{\te}{\ese} {\te}{\ese} {\te}{\ese}, {\ese} {\te}{\ese} {\te}{\ese},  {\te}{\ese} {\te}{\ese},
{\ese} {\te}{\ese},  {\te}{\ese},  {\ese} \} {\color{black}{.}}
\end{equation}
Let $ {\Lambda}^c_{w} \subseteq \tilde{\Lambda}_{w}  $
be the subset obtained from $ {\Lambda}_{w} $ by deleting the last generator
of each element of $ \tilde{\Lambda}_{w} $. Keeping the example $ n=7 $ we have 
$   {\Lambda}^c_{w} =
\{ {\ese} {\te}{\ese} {\te}{\ese} {\te},
{\te}{\ese} {\te}{\ese} {\te}, {\ese} {\te}{\ese} {\te},  {\te}{\ese} {\te},
{\ese} {\te},  {\te},  1 \} $. With this definition we have 
\begin{equation}\label{taildec}
\tilde{\Lambda}_{w} = {\Lambda}_{w} \, \dot\cup \,{\Lambda}^c_{w} {\color{black}{.}}
\end{equation}
The relevance of ${\Lambda}_{w} $ 
comes from the following Theorem.
\begin{theorem}\label{secondcellular}
  The triple $ ( \Lambda_{w}, {\tab}_w(v), C) $ defines a cellular basis structure on $ {A}_w$.
\end{theorem}  
\begin{dem}
By part ${\bf b})$ of Theorem \ref{mainTheoremSection3new} we know that 
$ A_w $ is a subalgebra of $ \tilde{A}_{w} $ and so the Theorem 
follows immediately from Theorem \ref{secondcellular} and
the fact that 
for $ D_1 , D_2 \in {\tab}_w(v ) $
and $ v \in \tilde{\Lambda}_{w}$ 
we have that $ C_{D_1 D_2 }^v $ belongs to $ A_w $ {\color{black}{if and only if}} 
$ v \in {\Lambda}_{w}$.
\end{dem}  

\medskip
By Theorem {\ref{mainTheoremSection3new}} we know that $ \BRnn \cong A_w $
and from Theorem \ref{firstcellular2} and Theorem \ref{secondcellular} we know that
both  
algebras are cellular. It now seems plausible that the corresponding 
cell modules are isomorphic as well. This is however not automatic since
the cellular structure on a cellular algebra is not unique. Our next aim is to show that the 
cell modules are indeed isomorphic.

\medskip
For this we first need to establish a poset
isomorphism $ \psi: \Lambda_{w}  \cong   \Lambda_{  {\color{black}{\pm }  ( n-1)}}  $. 
The sets $   \Lambda_{w} $ and $ \Lambda_{  {\color{black}{\pm }  ( n-1)}}  $
are both of cardinality $ n $
and the respective order relations are both total, so there is in fact a unique choice of $ \psi $.
It is given by the following Lemma, where $ l(\cdot) $ is the usual Coxeter length function on $ W$. 
\begin{lemma}\phantomsection\label{usualLength}
  Let $ \psi: \Lambda_{w}  \rightarrow   \Lambda_{{\color{black}{ \pm (n-1)}}}   $ be the map defined by 
  \begin{equation}
\psi(v) = \left\{ \begin{array}{ll} l(v)-1 &  {\rm  if }\, \, \,  v  \, \,  \mbox{begins with}\, \,  \ese   \\
  -l(v) & { \rm  if } \, \, \,  v  \, \,  \mbox{begins with}\, \,  \te {\color{black}{.}}
\end{array} \right.
  \end{equation}
  Then $ \psi $ is an isomorphism of posets.
{\color{black}{We denote by $ \varphi $ the inverse $ \varphi =  \psi^{-1}: \Lambda_{ \pm (n-1)}
\rightarrow  \Lambda_{w} $.}}
 \end{lemma}  
(The different meanings of the symbol $ \varphi$, for example as
the algebra isomorphism $ \BRn \cong A_w $, but also as the poset isomorphism
$ \Lambda_{{ \color{black}{\pm ( n-1)}}} \cong \Lambda_{w} $, should not give rise to confusion).
For example, if $ n = 7 $ we have
\begin{equation}
\varphi(6,-6,4,-4,2,-2,0) =   
(  {\ese} {\te}{\ese} {\te}{\ese} {\te}{\ese},
{\te}{\ese} {\te}{\ese} {\te}{\ese}, {\ese} {\te}{\ese} {\te}{\ese},  {\te}{\ese} {\te}{\ese},
{\ese} {\te}{\ese},  {\te}{\ese},  {\ese} ) {\color{black}{.}}
\end{equation}  

For $ v \in \Lambda_w $ 
we now introduce 
$ C_{v} \in A_w$ as the double leaf diagram of the form {\eqref{Csoergel}} or
\eqref{CsoergelAB} that defines $v \in W$, that is

\begin{equation}{\label{CClambday}} 
  C_{v} = \left\{
  \begin{array}{l} 
    \raisebox{-.40\height}{\includegraphics[scale=0.7]{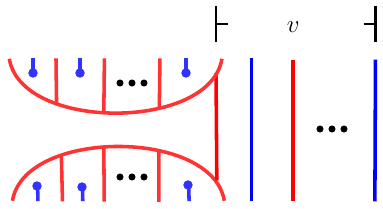}} 
       \\        \\ \\
    \raisebox{-.4\height}{\includegraphics[scale=0.7]{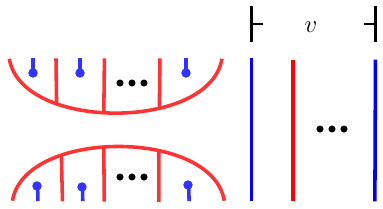}} 
  \end{array}
  \right.
\end{equation}
where 
the vertical lines below $ \raisebox{-.4\height}{\includegraphics[scale=0.7]{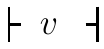}} $ in
each diagram of 
{\eqref{CClambday}} define $ v $. 

We have that $ C_v = \varphi(C_\lambda) $ where $ \varphi(\lambda) = v$, that is 
$  C_v $ corresponds to the blob algebra element $ C_{\lambda} $ defined in {\eqref{tlambda2}}.
Let $  A^{v}  $ and $ A^{ < v} $ be the cell ideals in $ A_w$ given by 
\begin{equation}
  A^{v} := \{ a C_{v } b \, | \, a,b \in A_w \}  \mbox{ and }  A^{< v} :=
\{ a C_{u} b \, | \, a,b \in A_w  \mbox{ and } u < v \} 
\end{equation}
and define the cell module
\begin{equation}
\Delta_w(v) := A_w \overline{C}_{v} \subset A^{ v}/A^{< v}
\end{equation}
where $ \overline{C}_{v} = {C}_{v} +A^{< v}$. 
Let $ D_v $ be the half-diagram corresponding to $ C_v $. 
Then we have the following Theorem. 
\begin{theorem}
$ \Delta_w(v) $ is free over $ R $ with basis
$ \{ \overline{C}_{D, D_{v}}\, | \, D \in {\tab}_w(v) \} $
where $ \overline{C}_{D, D_{v}} := {C}_{D, D_{v}} +
A^{< v} $. 
\end{theorem}
\begin{dem}
This follows from the algorithm given in the proof of Theorem 
{\ref{mainTheoremSection3new}}.
\end{dem}  

\medskip
We then finally obtain the main result of this section.

\begin{theorem}\label{finallymodules}
  Suppose that $ \varphi(\lambda) = v $.
Then the isomorphism $ \varphi: \BRnn \rightarrow  {A}_w$ induces an isomorphism 
$  \varphi: \Delta^{\Blob}_{n-1}(\lambda) \rightarrow \Delta_w(v)$
of cell modules.  
\end{theorem}  
\begin{dem}
  Since $ \varphi(C_\lambda) = C_v $ we have that $ \varphi $
  induces ideal isomorphisms $ A^{\lambda} \cong A^v $ and $ A^{< \lambda} \cong A^{< v}$ and 
  hence $ A^{\lambda} /A^{< \lambda} \cong A^v / A^{< v} $. But from this we deduce
  that $ \Delta_w(v) $ and $ \Delta^{\Blob}_{n-1}(\lambda) $ are isomorphic under $ \varphi$, as claimed.
\end{dem}

\medskip
For a general cellular algebra $ \cal A$ over $ \Bbbk     $ with cell modules
$ \{ \Delta(\lambda) \, | \, \lambda \in \Lambda \} $ there is a canonical bilinear 
form $ \langle \cdot, \cdot
\rangle_{\lambda} $ on $ \Delta(\lambda) $ that plays an important role for the representation
theory of $ \cal A$. Let $ \Lambda_0 = \{\lambda \in \Lambda \, | \,
\langle \cdot, \cdot \rangle_{\lambda}\neq 0 \} $ and define for $ \lambda \in \Lambda_0 $ the
$ \cal A$-module $ L(\lambda) :=  {\color{black}{\Delta}}(\lambda) / {\rm rad } \langle \cdot, \cdot \rangle_{\lambda} $
where $ {\rm rad } $ is the radical of $ \langle \cdot, \cdot \rangle_{\lambda} $ in the usual sense
of bilinear forms. It is an $ \cal A$-submodule of $ \Delta(\lambda) $ because
$ \langle \cdot, \cdot \rangle_{\lambda} $ is $ \cal A $-invariant, that is
$ \langle xy, z \rangle_{\lambda} = \langle x, y^{\ast} z \rangle_{\lambda} $
for all $ x \in \cal A $ and $ y, z \in  \Delta(\lambda) $,
{\color{black}{where $ \ast$ is the antihomomorphism of $\cal A $ given in
    Definition \ref{defcellularalgebra}.}}
In the case where $ \Bbbk $ is a field, the $ \cal A $-modules in the set 
$ \{ L(\lambda) \, | \, \lambda \in \Lambda_0 \} $ are all irreducible,
and each $ \cal A $-irreducible module occurs exactly once in the set,
{\color{black}{see Theorem 3.4 in \cite{GL}}}. 

\medskip
Let $ \langle \cdot, \cdot \rangle^w_{n, v} $ be the bilinear form on $ \Delta_{w}(v) $ and
let $ \langle \cdot, \cdot \rangle^{\B}_{n-1, \lambda} $ be the bilinear form on $ \Delta^{\Blob}_{n-1}(\lambda) $. 
Then we have the following Theorem.  
\begin{theorem}\phantomsection\label{equivalentforms}
  $ \langle \cdot, \cdot \rangle^{\B}_{n-1, \lambda} $ and $ \langle \cdot, \cdot \rangle^w_{n, v} $
  are equivalent under
  $ \varphi $, in other words
\begin{equation}
  \langle \varphi(s), \varphi(t) \rangle^w_{n, v} =
   \langle s, t \rangle^{\B}_{n-1, \lambda} \, \, \, \mbox{for } s,t \in \Delta^{\Blob}_{n-1}(\lambda) 
  \end{equation}
where $ \varphi(\lambda) = v$.
\end{theorem}  
\begin{dem}
A bilinear and invariant form $  \langle \cdot, \cdot \rangle $ on 
$ \Delta_w(v) $ corresponds to an 
$ A_w$-homomorphism
$ \Delta^{\Blob}_{ n-1}(\lambda) \rightarrow \Delta_{ n-1}^{\Blob, \ast}(\lambda)$
where $ \Delta^{\Blob, \ast}_{ n-1}(\lambda) $ is the dual of $ \Delta^{\Blob}_{ n-1}(\lambda)$.
For $ Q $ the fraction field of $ R$ and $\overline{Q} $ its algebraic closure we 
set $ \Delta^{\overline{Q}}_{ n-1}(\lambda) := \Delta^{\Blob}_{ n-1}(\lambda) \otimes_R \overline{Q} $.
Then $ \Delta^{\overline{Q}}_{ n-1}(\lambda)  $ is irreducible and so by Schur's Lemma 
$ \Delta^{\overline{Q}}_{ n-1}(\lambda) \rightarrow \Delta_{ n-1}^{\overline{Q},\ast}(\lambda)$ is
unique up to a scalar. Hence $  \langle \cdot, \cdot \rangle $ is unique up to a scalar $ \mu$, 
that is
\begin{equation}
\langle \psi(a), \psi(b) \rangle^{\B}_{n-1, \lambda} = \mu
   \langle a, b \rangle^w_{n, v} \, \, \, \mbox{for } a,b \in \Delta_{w, v}(\lambda)
\end{equation}  
But using $ a=b = C_v$, one {\color{black}{checks}} that $ \mu =1$ and so the Theorem follows.
\end{dem}  

\medskip
The purpose of the present paper is to study the form $ \langle \cdot, \cdot \rangle^w_{n, v} $. 
In view of Theorem
\ref{equivalentforms}, we can instead study the form 
$ \langle \cdot, \cdot \rangle^{\B}_{n-1, \lambda} $, which turns out to be easier to handle.

\section{Restriction of $ \Delta^{\Blob}_{ n}(\lambda) $ to $ \TRn$ }\label{restriction}
As already mentioned, there is an embedding $ \TRn \subseteq \BRn $ which at the diagrammatic level
is {\color{black}{an}} embedding of Temperley-Lieb diagrams in blob diagrams. 
This gives rise to a restriction functor
${\rm Res} = {\rm Res}^{\BRn}_\TRn  $ from $ {\BRn} $-modules to $ {\TRn} $-modules.
In this section we study the application of ${\rm Res} $ on $ \Delta^{\Blob}_{ n}(\lambda) $.

\medskip
Recall from Theorem \ref{cellblob} that the 
cell module $ \Delta^{\Blob}_n(\lambda) $ for $ \BRn$ (resp. $ \Delta^{\TL}_n(\lambda) $ for $ \TRn$)
has basis 
$ \{ \overline{C}_{s t_{\lambda}}\, | \, s \in {\tab}(\lambda) \} $.
From now on,
if $ \lambda \ge 0 $ 
we shall identify $  \overline{C}_{s t_{\lambda}}  $ with $ s $ so that
we consider the basis for $ \Delta^{\Blob}_n(\lambda) $ 
to consist of half-diagrams. 
For example, the basis of the $ \BRcinco$-module
$ \Delta^{\Blob}_5(1) $ consists of the following half-diagrams
\begin{equation}{\label{examplehalf}}
   \begin{array}{l}
    \raisebox{-.2\height}{\includegraphics[scale=0.7]{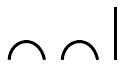}}, \, \, 
    \raisebox{-.2\height}{\includegraphics[scale=0.7]{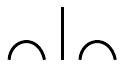}}, \, \, 
    \raisebox{-.2\height}{\includegraphics[scale=0.7]{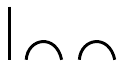}}, \, \,
 \raisebox{-.2\height}{\includegraphics[scale=0.7]{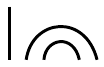}}, \, \,    
 \raisebox{-.2\height}{\includegraphics[scale=0.7]{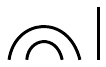}}, \\
 \raisebox{-.2\height}{\includegraphics[scale=0.7]{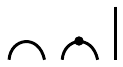}}, \, \,
 \raisebox{-.2\height}{\includegraphics[scale=0.7]{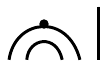}}, \, \,
 \raisebox{-.2\height}{\includegraphics[scale=0.7]{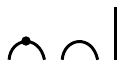}}, \, \,
 \raisebox{-.2\height}{\includegraphics[scale=0.7]{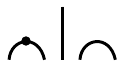}}, \\
      \raisebox{-.2\height}{\includegraphics[scale=0.7]{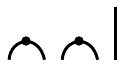}}
      \raisebox{-2\height}{ {\color{black}{.}}}
   \end{array}    
\end{equation}
If $ \lambda < 0 $ 
we still identify $  \overline{C}_{s t_{\lambda}}  $ with $ s $, but with the leftmost propagating
line marked.  
For example, the basis of the $ \BRcinco$-module
$ \Delta^{\Blob}_5(-3) $ consists of the following half-diagrams
\begin{equation}{\label{examplehalfneg}}
   \begin{array}{l}
    \raisebox{-.2\height}{\includegraphics[scale=0.7]{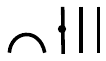}}, \, \, 
    \raisebox{-.2\height}{\includegraphics[scale=0.7]{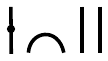}}, \, \, 
    \raisebox{-.2\height}{\includegraphics[scale=0.7]{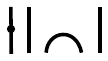}}, \, \,
        \raisebox{-.2\height}{\includegraphics[scale=0.7]{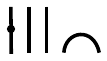}}, \, \, 
        \\
        \raisebox{-.2\height}{\includegraphics[scale=0.7]{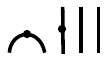}}
      \raisebox{-2\height}{ {\color{black}{.}}}
        \,  \, \, 
       \end{array}    
\end{equation}
Finally, for the $\TL_{\color{black}{n}}$-module $ \Delta^{{\color{black}{ \TL}}}_{{\color{black}{n}}}(\lambda) $
we once again identify
$  \overline{C}_{s t_{\lambda}}  $ with $ s $. 
Thus, for example the basis of the $ \TRcinco$-module $ \Delta^{\TL}_5(1) $ consists of the half-diagrams
that appear in the first row of \eqref{examplehalf}.

\medskip
In terms of these identifications the action of $ D \in \BRn$
(resp. $D \in \TRn$) on $ D_1 \in \Delta_n^\B(\lambda) $ (resp. $ D_1 \in \Delta_n^\TL(\lambda)) $ is given by
concatenation with $D_1$ on top of $D$, followed by the same
reduction process of extra blobs and internal loops, marked
or unmarked, that we gave for $ \BRn$ (resp. $\TRn$) itself.
If the result of this does not belong to the span of half-diagrams for
$  \Delta_n^\B(\lambda) $ (resp. $  \Delta_n^\TL(\lambda) $), 
we have $DD_1 = 0.$

\medskip
Suppose that $ \lambda \in \Lambda_{\pm n} $ and set $ k := \dfrac{n-|\lambda |}{2}$. We 
then define a filtration
$0 = \F^{-1}(\lambda) \subset  \F^0(\lambda) \subset \cdots \subset \F^k(\lambda) = 
\Delta^{\Blob}_n(\lambda) $ of $ \Delta^{\Blob}_n(\lambda) $ via 
\begin{equation}\label{ffiltration}
  \F^i(\lambda) := \left\{
\begin{array}{ll}
  \spa_R \{ s \, | \, s \in {\tab}(\lambda) \mbox{ has } i \mbox{ or less blobs} \} \ & \mbox{ if } \lambda \ge 0 \\
  \spa_R \{ s \, | \, s \in {\tab}(\lambda) \mbox{ has } i+1 \mbox{ or less blobs} \} \ & \mbox{ if }
  \lambda < 0 {\color{black}{.}}
\end{array} \right.
\end{equation}
For example, for $ \lambda $ as in {\eqref{examplehalf}} we have that 
$ \F^0(\lambda) $ is the span of the diagrams of the first row, 
$ \F^1(\lambda) $ is the span of the diagrams of the first two rows and $ \F^2(\lambda) =
\Delta^{\B}_5(1) $.

\medskip
The following result is the analogue of Lemma 8.2 from \cite{blob positive}. 
\begin{lemma}\phantomsection\label{homomorphismf}
  \begin{description}
   \item[a)] $ \F^i(\lambda) $ is a $ \TRn$-submodule of $ {\rm Res} \Delta^{\Blob}_n(\lambda) $. 
  \item[b)] There is a homomorphism of $\TRn$-modules $\pi_i:  \F^i(\lambda) \rightarrow
    \Delta^{\TL}_n(|\lambda|+2i) $ that induces an isomorphism
    \begin{equation}
      \overline{\pi}_i:  \F^i(\lambda)/\F^{i-1}(\lambda)  \cong    \Delta^{\TL}_n(|\lambda|+2i)
  { {\color{black}{.}}}
    \end{equation}
  \end{description}
\end{lemma}  
\begin{dem} ${ \bf a)}$ follows from the fact that the action of $ \TRn$ does not produce new blobs.
To show ${ \bf b)}$, we use the map $ \pi_i: \F^i(\lambda) \rightarrow 
\Delta^{\TL}_n(|\lambda|+2i) $ that transforms a marked southern arc to two propagating lines, and
removes the mark on any propagating line. For example, for $ \Delta^{\Blob}_5(1) $ we have
that {\color{black}{$ \pi_1 $}} transforms the diagrams of the second row of {\eqref{examplehalf}} to
the following diagrams
\begin{equation}{\label{examplehalftransform}}
 \raisebox{-.2\height}{\includegraphics[scale=0.7]{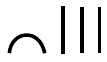}}, \, \,
  \raisebox{-.2\height}{\includegraphics[scale=0.7]{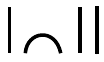}}, \, \,
 \raisebox{-.2\height}{\includegraphics[scale=0.7]{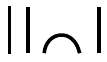}}, \, \,
 \raisebox{-.2\height}{\includegraphics[scale=0.7]{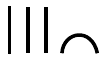}} \!\!
      \raisebox{-2\height}{ {\color{black}{.}}}
\end{equation}
As in \cite{blob positive}, one readily checks that $ \pi_i $ is a homomorphism
of $ \TL_{{\color{black}{n}}}$-modules,
that induces an isomorphism
$\overline{\pi}_i:  \F^i(\lambda)/\F^{i-1}(\lambda)  \cong    \Delta^{\TL}_n(|\lambda|+2i)$.
\end{dem}  

\medskip
\noindent

Our next goal is to construct sections for the $\pi_i$'s from Lemma \ref{homomorphismf}. For this
we need to recall the {\it Jones-Wenzl idempotent} $ \JWn$ for $ \TRn$, see
\cite{Jo} and \cite{Wenzl}. It is defined as the unique element $ \JWn \in \TRn $ that
satisfies the conditions
\begin{equation}\label{conditionsJW}
{\rm coef}_1(\JWn) = 1 \mbox{  and  }  \UU_i \JWn = \JWn \UU_i=  0 \mbox{ for all } i=1,2,\ldots,n-1 
\end{equation}
where $  {\rm coef}_1(\JWn) $ denotes the coefficient of $ 1 $ when $  \JWn $ is expanded in the diagram
basis for $ \TRn$. We use the standard rectangle notation to indicate $ \JWn$, as follows 
\begin{equation}
\JWn= 
\raisebox{-.45\height}{\includegraphics[scale=0.9]{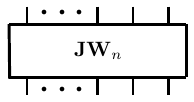}} \in \TRn
{ {\color{black}{.}}}
\end{equation}
For example, we have that 
\begin{equation}\label{JWrectangle}
  \raisebox{-.35\height}{\includegraphics[scale=0.9]{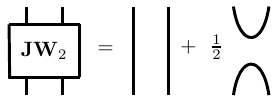}} 
\end{equation}
and
\begin{equation}
  \raisebox{-.35\height}{\includegraphics[scale=0.9]{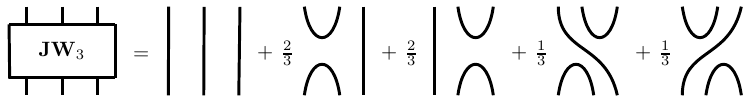}} 
      \raisebox{-12\height}{ {\color{black}{.}}}
\end{equation}
Apart from special cases, see for example Corollary 3.7 of \cite{GL2}, 
there is no closed formula for the coefficient of
a diagram in $ \JWn$, but there are several, equivalent, recursive formulas that can be used
to calculate $ \JWn$, for example 
\begin{equation}\label{firstrecursionJW}
  \raisebox{-.45\height}{\includegraphics[scale=0.9]{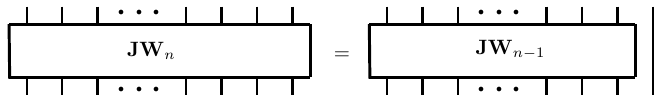}} +\frac{n-1}{n}
    \raisebox{-.45\height}{\includegraphics[scale=0.9]{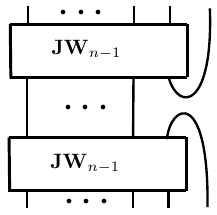}}
      \raisebox{-36\height}{ {\color{black}{.}}}
\end{equation}
In the next section we shall explain another recursive formula for $ \JWn$ that turns out to
be useful for our key calculations. 

Important properties of the $\JWn^{\! \!\prime}$s are their idempotency, as already mentioned,
and their invariance under vertical reflection as well as horizontal reflection, that is $ \ast$. Furthermore,
they satisfy the following absorption property
\begin{equation}
  \raisebox{-.5\height}{\includegraphics[scale=0.9]{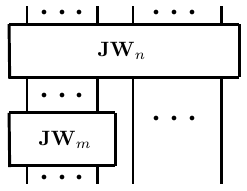}}  \,\,\,= \,\,\,
    \raisebox{-.5\height}{\includegraphics[scale=0.9]{JWn.pdf}} 
\end{equation}
where $ m \le n$

\medskip
We now return to the $ \BRn$-module $\Delta_n^{\Blob}(\lambda) $ and its filtration
$ \{ \F^{i}(\lambda) \} $. For $ i =0, \ldots, k $ we define $ e_{i}^{\lambda} \in \F^{i}(\lambda) $
as the following element
\begin{equation}
  e_{i}^{\lambda} = \left\{
  \begin{array}{cc}
    \raisebox{-.2\height}{\includegraphics[scale=0.9]{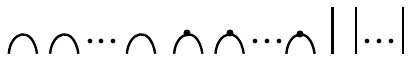}} & \mbox{ if } \lambda \ge 0 \\
    \raisebox{-.2\height}{\includegraphics[scale=0.9]{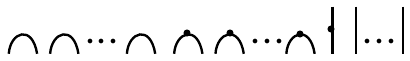}} & \mbox{ if } \lambda < 0
  \end{array}
    \right.
\end{equation}
that is, the number of blobs on $   e_{i}^{\lambda} $ is $ i $ if $ \lambda \ge 0 $ and $ i+1 $
if $ \lambda < 0$.
In general, for $D $ any Temperley-Lieb diagram, we define $1^j D$ as the left concatenation
of $ j $ vertical lines on $D$, and we extend this definition linearly to the Temperley-Lieb
algebra itself. For example, 
for $ j = n-| \lambda | -2i $ we have $ 1^{j} \mathbf{JW}_{| \lambda| +2i } \in \TRn$ which
we depict as follows 
\begin{equation}
  1^{j} \mathbf{JW}_{| \lambda| +2i } =
    \raisebox{-.38\height}{\includegraphics[scale=0.9]{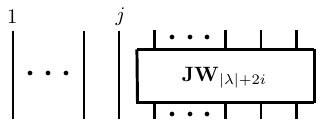}} 
      \raisebox{-17\height}{ {\color{black}{.}}}
\end{equation}
{\color{black}{With this choice of $ j $ we}} define elements $   f_{i}^{\lambda} \in \F^{i}(\lambda) $
as follows 
\begin{equation}\label{cases}
  f_{i}^{\lambda} :=    (1^{j} \mathbf{JW}_{| \lambda| +2i })  e_{i}^{\lambda}  =\left\{
  \begin{array}{cc}
    \raisebox{-.2\height}{\includegraphics[scale=0.9]{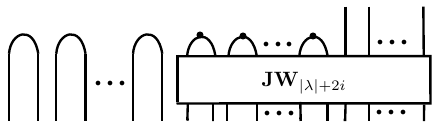}} & \mbox{ if } \lambda \ge 0 \\
    \raisebox{-.2\height}{\includegraphics[scale=0.9]{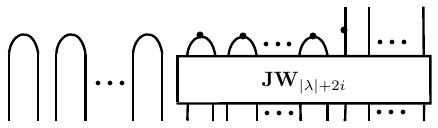}} & \mbox{ if } \lambda < 0 {\color{black}{.}}
  \end{array}
    \right.
\end{equation}
%\cancel{\color{black}{We have the following Lemma.}}
\begin{lemma}\label{wehavethefollowingLemma}
Let $\pi_i:  \F^i(\lambda) \rightarrow
\Delta^{\TL}_n(|\lambda|+2i) $ be the homomorphism from Lemma \ref{homomorphismf}.
%\cancel{\color{black}{Then}}
We have that
\begin{equation}
  \pi_i(   e_{i}^{\lambda} ) = \pi_i(   f_{i}^{\lambda} ) {\color{black}{.}}
\end{equation}
\end{lemma}
\begin{dem} 
  We show that if $D \neq 1 $ is any diagram appearing in the expansion
  of $  \mathbf{JW}_{| \lambda| +2i}$ then $ (1^j D)  e_{i}^{\lambda} \in \F^{i-1}(\lambda) $, where
  $ j = n-| \lambda | -2i $, from which the proof of the Lemma follows
  since $ {\rm ker}\, \pi_i = \F^{i-1}(\lambda) $ and 
  ${\rm coef}_1( \mathbf{JW}_{| \lambda| +2i}) = 1$.
  To prove this claim
we first consider $ \lambda \ge 0$. Any $ D \neq 1 $ in the expansion of 
$  \mathbf{JW}_{| \lambda| +2i}$ contains at least one arc connecting two northern points. If that arc
only involves blobbed arcs in \eqref{cases}, illustrated with red below, or only
involves vertical lines in \eqref{cases}, illustrated with blue below, then
$ (1^j D)e_i^{\lambda} $ has strictly less blobs than $ e_i^{\lambda}$ or has
strictly less 
propagating lines than $ e_i^{\lambda}$, proving 
the claim in these cases.
\begin{equation}
\raisebox{-.5\height}{\includegraphics[scale=0.9]{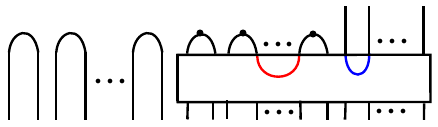}} 
      \raisebox{-23\height}{ {\color{black}{.}}}
\end{equation}
If the arc involves both the blobbed arcs in \eqref{cases} and the vertical lines in \eqref{cases},
then either we will be in the previous case or 
there will be a, possibly different, arc that involves the last blobbed arc and the first vertical line in 
\eqref{cases}, illustrated with blue below. But then 
$ (1^j D ) e_{i}^{\lambda}  $ has a vertical blobbed line  and is zero in $ \Delta_n^{\Blob}(\lambda)$, finishing the proof of the case
$ \lambda \ge 0$. 
\begin{equation}
\raisebox{-.5\height}{\includegraphics[scale=0.9]{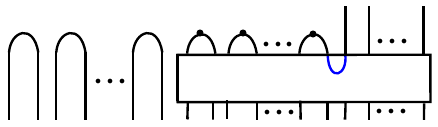}} 
      \raisebox{-23\height}{ {\color{black}{.}}}
\end{equation}
The case $ \lambda < 0$ is shown in a similar way.
\end{dem}  

\medskip
We now define the $ \TRn$-module $ S_i(\lambda)  $ via
\begin{equation}
  S_i(\lambda) := \TRn f_i^\lambda \subseteq \F^{i}(\lambda)
 {\color{black}{.}}
\end{equation}
  
\medskip
Recall the bilinear form $ \langle \cdot, \cdot \rangle^{\Blob}_{n,\lambda} $ on
$ \Delta^{\Blob}_n(\lambda) $.
In terms of half-diagrams $ D, D_1 $ for $ \Delta^{\Blob}_n(\lambda) $, 
we have that $ \langle D, D_1 \rangle_{n,\lambda}^{\Blob} $ is given by expanding  
$ D^{\ast} D_1 $ in terms of the diagram basis for $ \mathbb{B}^{x,y}_{| \lambda | } $ and
taking the coefficient of $ 1 $ if $ \lambda > 0 $,
the coefficient of $ \emptyset $ if $ \lambda = 0 $ and 
the coefficient of $ \UU_0 $ if
$ \lambda <0 $.
For example, for $ \langle \cdot, \cdot \rangle^{\Blob}_{5,1} $ we have
\begin{equation}
\left\langle  \raisebox{-.35\height}{\includegraphics[scale=0.7]{halfblobH.pdf}}, 
\raisebox{-.35\height}{\includegraphics[scale=0.7]{halfblobH.pdf}} \right\rangle_{ 5,1}
= {\rm coef}_1 \! \left( \! \! \raisebox{-.35\height}{\includegraphics[scale=0.7]{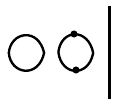}} \right)
= -2 xy {\color{black}{.}}
\end{equation}  
There is a similar description of the bilinear form $ \langle \cdot, \cdot \rangle^{\TL}_{n,\lambda} $ on
$ \Delta^{\TL}_n(\lambda) $. With this notation we can now prove the following Theorem.
\begin{theorem}\phantomsection\label{withthisnotation}
    \begin{description}
   \item[a)]  We have $ (e_j^{\lambda})^{\ast} \TRn  f_i^{\lambda}=0 $ for $ j <i$.
\item[b)]
     We have $ \langle \F^{i-1}(\lambda), S_i(\lambda)  \rangle^{\Blob}_{n,\lambda} =0 $.  
\item[c)]
     We have $ S_i(\lambda) \cap \F^{i-1}(\lambda) =0$.
   \item[d)]  The $ \TRn$-modules $S_i(\lambda)  $ and $ \Delta^{\TL}_n(|\lambda| +2i) $ are isomorphic.
\end{description}
  \end{theorem}  
\begin{dem}
To show $ {\bf a})$ we must check that the following diagram is zero for every diagram $ D$ for $  \TRn$.
\begin{equation}\label{mustshowzero}
\raisebox{-.45\height}{\includegraphics[scale=0.9]{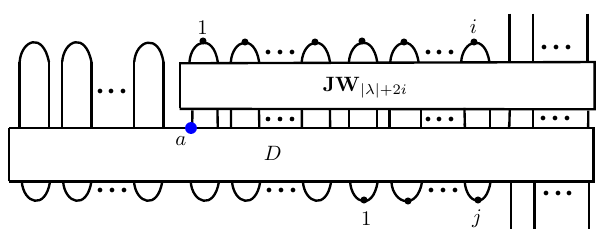}}
      \raisebox{-40\height}{ {\color{black}{.}}}
\end{equation}
Let $ a := n-|\lambda| -2i+1$ and consider all possible cases for the line $L $ leaving
the $a$'th northern point of $D$, indicated with blue
{\color{black}{in \eqref{mustshowzero}}}.

If $L$ {\color{black}{pairs $a$ with}} a northern point {\color{black}{of $D$ which is located}} to
the right of $ a$, then
$L$ {\color{black}{is}} a northern arc and so we get immediately by \eqref{conditionsJW}
that \eqref{mustshowzero} is zero. 
If $L$ {\color{black}{pairs $a $ with}}
a southern point
{\color{black}{of $D$ which is located}} 
strictly to the right of $a$, then the area to the right of
$ L $ has strictly more northern than southern points and so there will be
a northern arc to the right of $a$. We then conclude once again by \eqref{conditionsJW}
that \eqref{mustshowzero} is zero.
This is the situation indicated in the figure below.
\begin{equation}
\raisebox{-.45\height}{\includegraphics[scale=0.9]{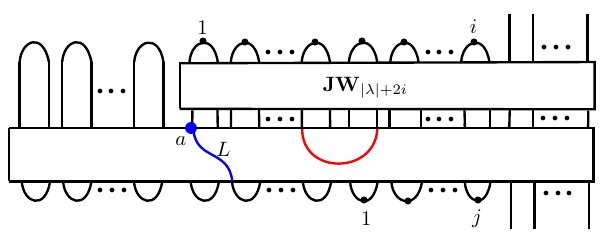}}=0
{ {\color{black}{.}}}
\end{equation}
If $L$ {\color{black}{pairs $a$ with}}
a southern point {\color{black}{which is located either directly below or}}
to the left of $a$, then it will be a left
endpoint of one of the southern arcs of $ D $, 
%\cancel{\color{black}{and to the left of $ a$}}
since
the parities of the numbers of northern and southern points of $ D$ to the left of
{\color{black}{$ L$}} must be 
the same. If the right endpoint of that arc is connected to a northern point of $ D $, then once again
by \eqref{conditionsJW} we get 
that \eqref{mustshowzero} is zero, so let us assume that it is connected to a southern point of $D$
via a line $A_1$. 
That southern point must be a left endpoint of an arc below $D$
since otherwise the number of 
southern points below $A_1$ {\color{black}{would}} be odd. If the arc is unmarked and its right endpoint 
is joined to a northern point of $D$,
we get once again by \eqref{conditionsJW} that 
\eqref{mustshowzero} is
zero, so let us suppose that it is joined to another southern point of $ D $ via a line $A_2$.
Repeating the previous argument, the right endpoint of $ A_2$ must be the left endpoint of an arc
whose right endpoint, in case the arc is unmarked, is connected by a line $A_3 $ to another southern
point of $D$. Repeating this argument, we produce a series of southern lines $A_1, A_2, \ldots,A_k$, that
finally ends up in {\color{black}{either an endpoint of}}
one of the blobbed southern arcs below $D$, {\color{black}{or in an endpoint of one of the vertical
lines below $D$.}} But since $ j<i $ we then 
conclude that not all northern points of $ D$  to the right of $ a$ can be
endpoints of 
propagating lines, and so we conclude
by \eqref{conditionsJW} that 
\eqref{mustshowzero} is
zero. Below we indicate the argument.  
\begin{equation}
  \raisebox{-.45\height}{\includegraphics[scale=0.9]{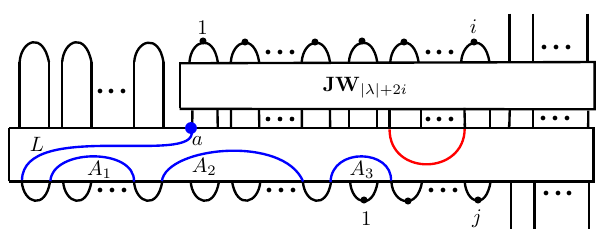}}=0
 {\color{black}{.}}  
\end{equation}
Finally, if $L$ {\color{black}{pairs $a $ with a northern point of $ D$ to the left of $a $}},
the argument is essentially the same
as in the previous case. We indicate it as follows
\begin{equation}
\raisebox{-.45\height}{\includegraphics[scale=0.9]{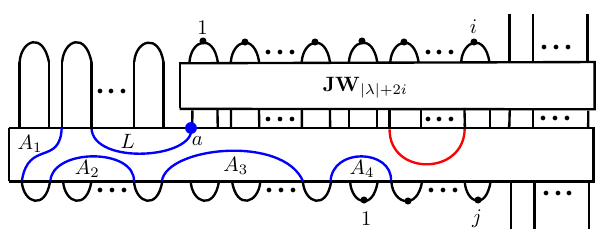}}=0 {\color{black}{.}}
\end{equation}
To show $ {\bf b}) $, we first observe that, 
as {\color{black}{a}} $ \TRn$-module,  
$ \F^{i-1}(\lambda) $ is generated
by $ \{ e_j^{\lambda} \, | \   j <i \}$, as follows from Lemma \ref{homomorphismf}.
In view of this, $  {\bf b}) $ follows from $ {\bf a}) $ and the definition of
$ \langle \cdot, \cdot \rangle^{\B}_{n,\lambda} $.

We next show $ {\bf c}) $. Suppose that there is 
$ s \in S_i(\lambda) \cap \F^{i-1}(\lambda) \setminus \{0\}$.
To get the desired contradiction, we first 
claim that
the restriction of $ \langle \cdot, \cdot \rangle^{\B}_{n,\lambda} $ to $ \F^{i-1}(\lambda) $
is non-degenerate. Indeed, any diagram in $ \F^{i-1}(\lambda) $ can be viewed
as a diagram for $ \Delta_n^{\TL}(|\lambda|)$,
decorated with certain blobs, 
see for example
{\eqref{examplehalf}} and {\eqref{examplehalfneg}}. 
For $ D $ a diagram for $  \F^{i-1}(\lambda) $ we denote by $ \TL(D) $ the associated diagram
for $ \Delta^{\TL}_n(|\lambda|)$, obtained by removing the blobs.
Then via the specialization 
$ x=1, y=-2$ given in \eqref{istheTL}, we have that
\begin{equation}\label{wededuce}
  \left(\langle D, D_1 \rangle^{\B}_{ n,\lambda}\right)_{| x=1, y=-2} =
   \langle \TL(D) , \TL(D_1) \rangle^{\TL}_{n,|\lambda|}
\end{equation}
as one checks from the definitions. Since $ \langle \cdot , \cdot \rangle^{\TL}_{n,|\lambda|} $
is non-degenerate, we now deduce from \eqref{wededuce}
that also the restriction
of $ \langle \cdot, \cdot \rangle^{\B}_{n,\lambda} $ to $ \F^{i-1}(\lambda) $
is non-degenerate, as claimed. Hence, for 
$ s \in S_i(\lambda) \cap \F^{i-1}(\lambda) \setminus \{0\}$ there exists an $s_1 \in \F^{i-1}(\lambda) $ such
that $ \langle s, s_1 \rangle^{\B}_{n,\lambda} \neq 0$, which is in contradiction with $ {\bf b}) $.
This proves $ {\bf c}) $.

To show $ {\bf d}) $ we consider the composition
\begin{equation}
  f_i: S_i(\lambda) \subseteq \F^{i}(\lambda) \longrightarrow \F^{i}(\lambda)/ \F^{i-1}(\lambda) \cong 
\Delta^{\TL}_n(|\lambda|+2i)
\end{equation}
where the last isomorphism is given in 
Lemma \ref{homomorphismf}.
In view of Lemma \ref{wehavethefollowingLemma}, we get that $ f_i $ is surjective. On the other
hand, the kernel of $ f_i $ is 
$ S_i(\lambda) \cap \F^{i-1}(\lambda) $
which is zero by $ {\bf c}) $, and so $ f_i $ is also injective.
The Theorem is proved.
\end{dem}

\begin{corollary}
Set as before $ k := \dfrac{n - | \lambda|}{2}$. Then,
with respect to $ \langle \cdot, \cdot \rangle^{\B}_{n,\lambda} $,
there is an orthogonal direct sum decomposition of ${\rm Res} \, \Delta^{\Blob}_n(\lambda) $, as follows
  \begin{equation}
{\rm Res} \,\Delta^{\Blob}_n(\lambda)  = S_1(\lambda) \oplus \ldots \oplus S_k(\lambda) {\color{black}{.}}
  \end{equation}
\end{corollary}

\begin{dem}  
Combining Lemma \ref{wehavethefollowingLemma} with 
$ {\bf c}) $ of Theorem \ref{withthisnotation}, we get that
\begin{equation}
{\rm Res}\, \F^{i}(\lambda) = S_i(\lambda) \oplus \F^{i-1}(\lambda)
\end{equation}  
and the Corollary follows by induction on this formula.
\end{dem}  

\medskip
The Corollary allows us to diagonalize $ \langle \cdot, \cdot \rangle^{\B}_{n,\lambda} $ as follows.
The restriction of $ \langle \cdot, \cdot \rangle^{\B}_{n,\lambda} $ to
$ S_i(\lambda) $ defines a $\TRn$-invariant bilinear form on $ S_i(\lambda) $ which by Schur's Lemma,
arguing as in the proof of  Theorem \ref{equivalentforms}, must be equivalent to
$\langle \cdot, \cdot \rangle^{\TL}_{n,|\lambda|+2i} $, that is
for $ s, t \in S_i(\lambda) $ we have that 
\begin{equation}\label{comb1}
  \left\langle s, t \right\rangle^{\B}_{n,\lambda} = c_i(x,y) \left\langle f_i(s), f_i(t)  \right\rangle^{\TL}_{n,|\lambda|+2i}
\end{equation}  
for some $ c_i(x,y) \in R$. On the other hand, by our choice of ground field $\comu$,
we have that $\langle \cdot, \cdot \rangle^{\TL}_{n,|\lambda|+2i} $ is equivalent to the standard bilinear form
on $ \Delta^{\TL}_n(|\lambda|+2i) $
given by the identity matrix. In other words, there is an $R$-basis $ f_1^i, f_2^i, \ldots, f_m^i $ for
$ \Delta^{\TL}_n(|\lambda|+2i) $ such that
\begin{equation}\label{comb2}
\langle f_k^i, f_l^i \rangle^{\TL}_{n,|\lambda|+2i} = \delta_{kl} 
\end{equation}  
where $ \delta_{kl}  $ is the Kronecker delta.
Combining \eqref{comb1} and \eqref{comb2} we then conclude that also 
$ \langle \cdot, \cdot \rangle^{\B}_{n,\lambda} $ can be diagonalized. To be precise, there is an $R$-basis
$ {\mathcal B } = \{ b_1, \ldots, b_m \} $ for $ \Delta^{\Blob}_n(\lambda) $ such that
the matrix $ M_{n, \lambda}^{\B}(x,y) = M_{n, \lambda}^{\B}
:= \left(\langle b_i, b_j \rangle^{\B}_{n,\lambda} \right)_{i, j =1,\ldots, m } $
for $ \langle \cdot, \cdot \rangle^{\B}_{n,\lambda} $ 
has the following form
\begin{equation}\label{matrix}
M_{n, \lambda}^{\B}= \begin{pmatrix}
c_0(x,y)  I_{d_0}  & 0 & \cdots & 0\\
0 & c_1(x,y)  I_{d_1} & \cdots & 0\\
\vdots & \vdots & \ddots & \vdots\\
0 & 0 & \cdots & c_k(x,y)  I_{d_k}
\end{pmatrix}
\end{equation}
where $ d_i  = {\rm dim} \,  \Delta^{\TL}_n(|\lambda|+2i) $ and
$ I_{d_i} $ is the $d_i \times d_i $-identity matrix, whereas the $0$'s are $0$-matrices of appropriate
dimensions.

\section{Determination of $ c_i(x,y)$ }\label{deter}
The purpose of this section is to calculate the $ c_i(x,y)$'s from
\eqref{comb1} and \eqref{matrix}. This is the key calculation of the paper. 
Quite surprisingly, the result turns out to be given in terms of nice expressions involving 
the {\it positive roots} for $W$. 

\medskip
For $ i=1,2, \ldots, $ we define $ \alpha_{x,i}, 
\alpha_{y,i} \in \mathfrak{h}^{\ast} $ via
\begin{equation}\label{roots}
  \alpha_{x,i}:= ix+(i-1)y = i \alpha_{\ese}+ (i-1)\alpha_\te, \, \, \, \, \, \, 
   \alpha_{y,i}:= iy+(i-1)x = i \alpha_{\te}+ (i-1)\alpha_\ese  {\color{black}{.}}
\end{equation}
By our choice of realization for $ W $, 
the following formulas for 
$   \alpha_{x,i} $ and $ \alpha_{y,i} $ hold,
{\color{black}{as one proves by induction using \eqref{W-action}.}}
\begin{equation}\label{thisshowsroots}
\begin{array}{llll}
\alpha_{x,2i-1} =  ({\color{red} s}  {\color{blue} t})^{i-1}  \alpha_{\ese},   &
\alpha_{y,2i} = {\color{blue} t} ({\color{red} s}  {\color{blue} t})^{i-1}  \alpha_{\ese},   &
\alpha_{y,2i-1} =   ({\color{blue} t}  {\color{red} s})^{i -1} \alpha_{\te},  &
 \alpha_{x,2i} = {\color{red} s} ({\color{blue} t}  {\color{red} s})^{i-1}  \alpha_{\te}  {\color{black}{.}}
  \end{array}   
\end{equation}
For a general Coxeter system $ (W,S)$ with realization $ \mathfrak{h} $ we define 
a root $ \beta \in \mathfrak{h}^{\ast}$  to be any element of the form $ \beta = w \alpha $ where
$ w \in W $ and $ \alpha \in  \mathfrak{h}^{\ast}$ is a simple root. The formulas in
\eqref{thisshowsroots}
show that $ \{  \pm  \alpha_{x,i},  \pm  \alpha_{y,i} \, | \, i=1,2, \ldots \} $ is the set of all roots for $W$,
with $ \alpha_{x,1} =x =  \alpha_{\ese} $ and $ \alpha_{y,1} =y =  \alpha_{\te} $ being the simple
roots. 

\medskip
With this notation we can now formulate the following Theorem.
\begin{theorem}\phantomsection\label{cifactorA}
  \begin{description}
\item[a)]  
  Suppose that $ \lambda \ge 0 $ and that $ f_i^{\lambda} $ is as in \eqref{cases}
but with $ j = 0 $, that is 
$ f_i^{\lambda} = f_k^{\lambda} $ where
\begin{equation}
 f_k^{\lambda}= \raisebox{-.38\height}{\includegraphics[scale=0.9]{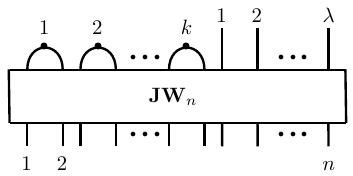}}
      \raisebox{-22\height}{ {\color{black}{.}}}
 \end{equation}
Then we have that
\begin{equation}\label{satisfies the formula}
  \langle f_k^{\lambda}, f_k^{\lambda} \rangle^{\B}_{n,\lambda}  = \dfrac{1}{\binom{n}{k}}
\, \left(  \alpha_{x,\lambda +2 } \,  \alpha_{x,\lambda +3} \cdots \alpha_{x,\lambda +k+1} \right)
\alpha_{y,1} \,  \alpha_{y,2} \cdots \alpha_{y,k}
\end{equation}
where $ n=2k+\lambda$. 
(Note that the
right hand side of \eqref{satisfies the formula} contains $k$ factors $ \alpha_{x,l} $ for consecutive 
$ l $'s, and also $k$ factors $ \alpha_{y,l} $, for consecutive $ l $'s.
For $ k= 0 $, it is set equal to $ 1$. Note also that $ \binom{n}{k}= \dim  \Delta^{\Blob}_n(\lambda) $).
\item[b)]
  Let $ c_i(x,y)  $ be as in \eqref{comb1} and \eqref{matrix}. Then up
  to multiplication by a nonzero scalar in $\comu$, we have that
\begin{equation}
c_i(x,y) = (  \alpha_{x,\lambda +2 }\,  \alpha_{x,\lambda +3} \cdots \alpha_{x,\lambda +i+1} )
  \alpha_{y,1} \,  \alpha_{y,2} \cdots \alpha_{y,i}
\end{equation}
(where, as before, the product is set equal to $1 $ if $i=0$).
\end{description}
 \end{theorem}  
\begin{dem}
  We first prove $ {\bf a}) $. For simplicity, we set $ \beta_{k, \lambda} :=
  \langle f_k^{\lambda}, f_k^{\lambda} \rangle^{\B}_{n,\lambda} $ and must
therefore 
  show that $ \beta_{k, \lambda} $
  satisfies the formula given in \eqref{satisfies the formula}. 
By definition $ \beta_{k, \lambda} $ is the coefficient of $ 1$, 
or the coefficient of $ \emptyset $, if $ \lambda >0  $ or if $ \lambda =0  $, 
of the following diagram
 \begin{equation}
 \raisebox{-.38\height}{\includegraphics[scale=0.9]{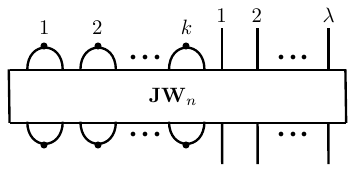}} 
      \raisebox{-23\height}{ {\color{black}{.}}}
 \end{equation}
For example, in view of \eqref{JWrectangle} we have that 
\begin{equation}\label{inviewbeta}
\beta_{1,0}=
   \raisebox{-.49\height}{\includegraphics[scale=1]{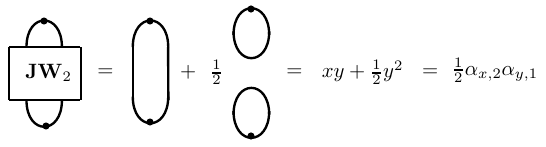}} \! \! \! \! 
      \raisebox{-2\height}{ {\color{black}{.}}}
\end{equation}

We claim that $ \beta_{k, \lambda}$ satisfies the following 'deformation' of the Pascal triangle
recursive formula
\begin{equation}\label{Pascalusing}
  \beta_{1, 0} = \frac{1}{2} \alpha_{x,2}\alpha_{y,1}, \, \, \, \, 
  \beta_{0, 1} = 1, \, \, \, \,
  \beta_{k, \lambda} = \beta_{k, \lambda-1} + \frac{k^2 (\alpha_{y, k})^2}{n(n-1)}  \beta_{k-1, \lambda}
  {\color{black}{.}}
  \end{equation}
From this the proof of the formula \eqref{satisfies the formula}
in ${\bf a})$
follows by induction on $ N:= k +\lambda$
as follows. The basis of the induction $ N=1 $ is immediate from \eqref{inviewbeta} and
the definitions, so let us assume that 
\eqref{satisfies the formula} holds for $  N-1 $ and check it for $N$, using 
\eqref{Pascalusing}. We get
\begin{equation}\label{firstcal}
  \begin{aligned}
    \beta_{k, \lambda} =
\dfrac{1}{\binom{n-1}{k}}
\, \left(  \alpha_{x,\lambda +1 } \,  \alpha_{x,\lambda +2} \cdots \alpha_{x,\lambda +k} \right)
\alpha_{y,1} \,  \alpha_{y,2} \cdots \alpha_{y,k} \\ +
 \dfrac{k^2 (\alpha_{y, k})^2}{n(n-1) \binom{n-2}{k-1}} 
\, \left(  \alpha_{x,\lambda +2 } \,  \alpha_{x,\lambda +3} \cdots \alpha_{x,\lambda +k} \right)
\alpha_{y,1} \,  \alpha_{y,2} \cdots \alpha_{y,k-1} \\ = \left(
\dfrac{\alpha_{x,\lambda +1} }{\binom{n-1}{k}}+ \dfrac{k^2 \alpha_{y, k}}{n(n-1) \binom{n-2}{k-1}} 
 \right)
\left(  \alpha_{x,\lambda +2} \cdots \alpha_{x,\lambda +k} \right)
\alpha_{y,1} \,  \alpha_{y,2} \cdots \alpha_{y,k} \\
= \dfrac{1}{(n-k)\binom{n}{k}}\left(
n \alpha_{x,\lambda +1} + k \alpha_{y, k}
 \right)
\left(  \alpha_{x,\lambda +2} \cdots \alpha_{x,\lambda +k} \right)
\alpha_{y,1} \,  \alpha_{y,2} \cdots \alpha_{y,k}    {\color{black}{.}}
  \end{aligned}    
\end{equation}
On the other hand, using $ n=2k +\lambda $ we get that
\begin{equation}\label{firstcalA}
  \begin{aligned}
    n \alpha_{x,\lambda +1} + k \alpha_{y, k} = (2k +\lambda) \alpha_{x,\lambda +1} + k \alpha_{y, k} \\
    = (2k +\lambda) \left(\lambda+1) x + \lambda y \right) +
    k\left(ky + (k-1)x \right) \\
    = (\lambda+k) \left( (\lambda+k+1)x+ (\lambda+k)\right)y
    = (\lambda+k) \alpha_{x, \lambda+k+1} \\ = (n-k) \alpha_{x, \lambda+k+1}   {\color{black}{.}}
\end{aligned}    
\end{equation}
Inserting this in the last expression of \eqref{firstcal} we get that
\begin{equation}
  \beta_{k, \lambda} =
\dfrac{1}{\binom{n}{k}}
\, \left(  \alpha_{x,\lambda +2 } \,  \alpha_{x,\lambda +3} \cdots \alpha_{x,\lambda +k+1} \right)
\alpha_{y,1} \,  \alpha_{y,2} \cdots \alpha_{y,k}    
\end{equation}
proving the inductive step.

\medskip
In order to prove the claim \eqref{Pascalusing}
we first introduce the following diagrammatic notation for $ \beta_{k, \lambda}$
  \begin{equation}
 \beta_{k, \lambda} =   \raisebox{-.38\height}{\includegraphics[scale=0.9]{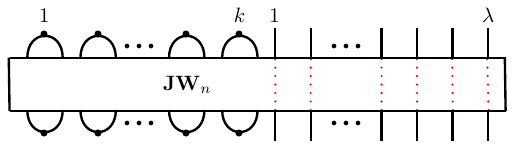}} 
  \end{equation}
  where the numbers above the diagram indicate the cardinalities of the arcs and the vertical
  lines. 
We next recall the following recursive formula for calculating $ \JWn$, that was already alluded to above
  \vspace{-0.2cm} 
\begin{equation}\label{secondrecursion}
  \raisebox{-.42\height}{\includegraphics[scale=0.9]{JWrecursion1A.pdf}} + \, \, \, 
 \mathlarger{\sum}_{j=1}^{n-1}\,  \dfrac{j}{n}\, \,
    \raisebox{-.3\height}{\includegraphics[scale=0.9]{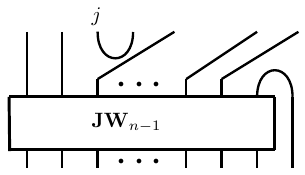}} 
\end{equation}
where the number $j$ indicates the position of the arc.
As a matter of fact, \eqref{secondrecursion} follows from by repeated applications of
\eqref{firstrecursionJW}. 
Let us use it to show the recursive formula \eqref{Pascalusing} for $ \beta_{k, \lambda}$. 

\medskip
Concatenating on top and on bottom 
with the blobbed arcs, the first term of 
\eqref{secondrecursion} becomes
  \begin{equation}
   \raisebox{-.4\height}{\includegraphics[scale=0.9]{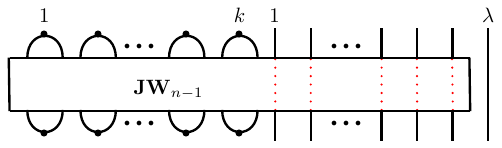}} =  \beta_{k, \lambda-1} {\color{black}{.}}
  \end{equation}
  We next consider the contribution to $ \beta_{k, \lambda}  $ from the terms of the sum 
in \eqref{secondrecursion}. 
We first observe that concatenating on top and on bottom 
with the blobbed arcs, only 
the terms in the sum in \eqref{secondrecursion} where $ j \le 2k $ can contribute to
$ \beta_{k, \lambda}  $ since otherwise the concatenation has the form 
  \begin{equation}\label{sinceotherwise}
   \raisebox{-.4\height}{\includegraphics[scale=1.0]{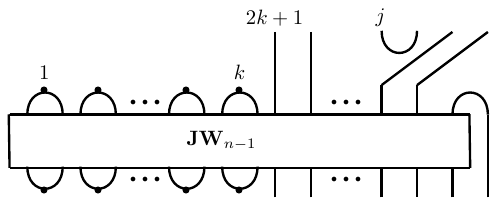}}
  \end{equation}
  in which the $j$'th northern point is connected to another northern point.

We next consider the 
contributions for $ j=1,3,\ldots, 2k-1 $. Apart from the coefficients, they are of the form 
indicated below (in the cases  $ j=1 $ and $j=3$). 
  \begin{equation}
   \raisebox{-.4\height}{\includegraphics[scale=1.0]{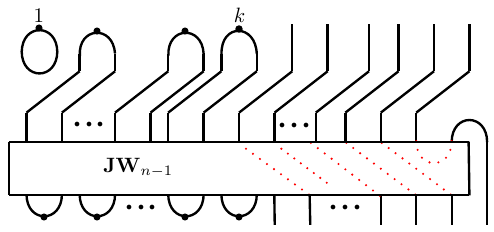}}
  \end{equation}
  \begin{equation}
   \raisebox{-.4\height}{\includegraphics[scale=1.0]{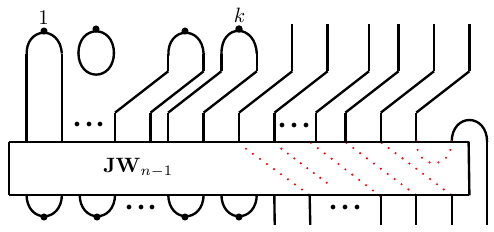}}
  \end{equation}
and are in fact all equal to $ y \beta_{k, \lambda}^{\prime} $ where $\beta_{k, \lambda}^{\prime}  $ is the diagram 
 \begin{equation}\label{betaprime}
   \raisebox{-.4\height}{\includegraphics[scale=1.0]{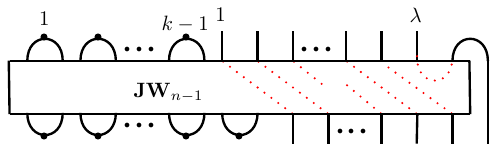}}
      \raisebox{-25\height}{ {\color{black}{.}}}
 \end{equation}
 In these diagrams, the dashed lines once again refer to the summands of the Jones-Wenzl
 elements where the points are connected as indicated.

We then consider the contributions for $ j=2,4,\ldots, 2k-2 $. 
They are all of the form indicated below (for $ j=2 $ and $ j=4$)
 \begin{equation}
   \raisebox{-.4\height}{\includegraphics[scale=1.0]{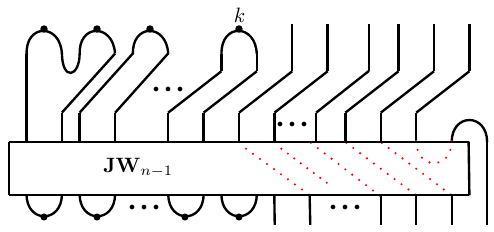}}
  \end{equation}
 \begin{equation}
   \raisebox{-.4\height}{\includegraphics[scale=1.0]{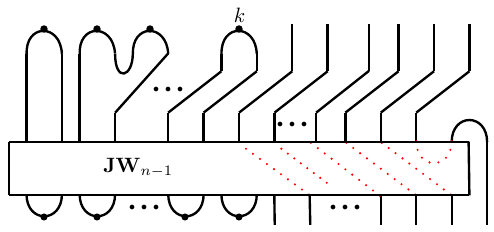}}
  \end{equation}
 and are all equal to $ x \beta_{k, \lambda}^{\prime}  $ where $ \beta_{k, \lambda}^{\prime}  $
 is as before in \eqref{betaprime}.
Finally, for $ j=2k $ there is no contribution since the corresponding diagram is as follows
 \begin{equation}
   \raisebox{-.4\height}{\includegraphics[scale=1.0]{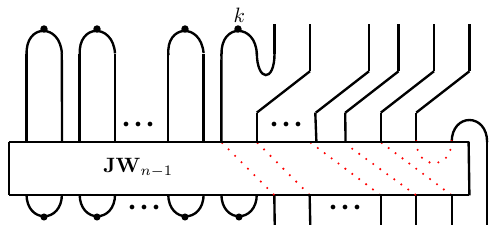}}
      \raisebox{-39\height}{ {\color{black}{.}}}
 \end{equation}
 Summing up and taking into account the coefficients
 $ \frac{j}{n}$ appearing in \eqref{secondrecursion} we get from all this that 
 \begin{equation}\label{porfinA}
   \beta_{k, \lambda}  = \beta_{k, \lambda-1} + k \dfrac{(k-1) x+ ky}{n}  \beta_{k, \lambda}^{\prime}
   = \beta_{k, \lambda-1} + \dfrac{k}{n} \alpha_{y, k}   \beta_{k, \lambda}^{\prime}   {\color{black}{.}}
 \end{equation}
 We are therefore faced with the problem of calculating $  \beta_{k, \lambda}^{\prime} $, that is 
  \begin{equation}
   \raisebox{-.4\height}{\includegraphics[scale=1.0]{betaE.pdf}}
      \raisebox{-24\height}{ {\color{black}{.}}}
  \end{equation}
  For this we first observe that by the symmetry properties of the Jones-Wenzl idempotents,
  we have that $ \beta_{k, \lambda}^{\prime} $ is equal to 
  \begin{equation}\label{betaprimeA}
( \beta_{k, \lambda}^{\prime})^{\ast}=   \raisebox{-.4\height}{\includegraphics[scale=1.0]{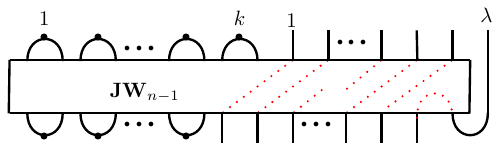}}
      \raisebox{-22\height}{ {\color{black}{.}}}
  \end{equation}
  We now expand $ \JWnn $ in \eqref{betaprimeA} using \eqref{secondrecursion}.
  The first term of \eqref{secondrecursion} does not contribute to 
  \eqref{betaprimeA} so let us consider the contribution of the $j$'th term of the
  sum of \eqref{secondrecursion}, where arguing as
  in \eqref{sinceotherwise} we see that only $ j \le 2k $ can contribute.
  Once again, there is a dependency on the parity of $ j $. If $ j=1,3, \ldots, 2k-1 $
  the contributions are of the form indicated below (for $j=1,3 $ and disregarding the coefficients)
  \begin{equation}
 \raisebox{-.4\height}{\includegraphics[scale=1.0]{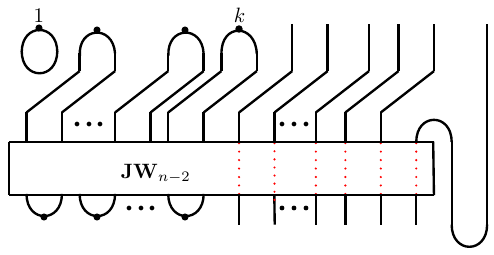}}
  \end{equation}
  \begin{equation}
 \raisebox{-.4\height}{\includegraphics[scale=1.0]{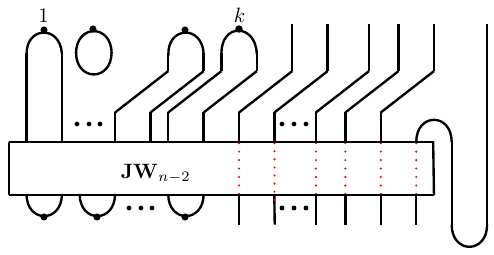}}
  \end{equation}
and are in fact all equal to $ y \beta_{k-1, \lambda} $. 
Similarly, for $ j = 2, 4, \ldots, 2k-2$ we get contributions of the form
  \begin{equation}
 \raisebox{-.4\height}{\includegraphics[scale=1.0]{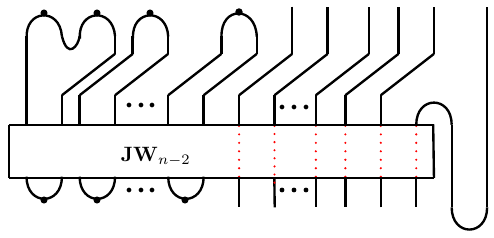}}
  \end{equation}
  \begin{equation}
 \raisebox{-.4\height}{\includegraphics[scale=1.0]{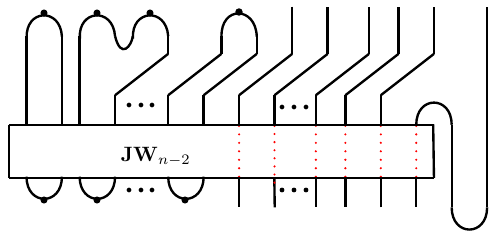}}
  \end{equation}
  all equal to $ x \beta_{k-1, \lambda} $. Once again, there is no contribution for $  j = 2k$. 
  Hence, taking into account the coefficients
  $ \frac{j}{n-1}$, we get that 
\begin{equation}\label{porfin}
  \beta_{k, \lambda}^{\prime} = 
\dfrac{k}{n-1} \alpha_{y, k}   \beta_{k-1, \lambda}   {\color{black}{.}}
\end{equation}  
Combining \eqref{porfin} and \eqref{porfinA} we arrive at the promised recursive formula 
\eqref{Pascalusing} for the $ \beta_{k, \lambda}$'s. This proves $ {\bf a})$.
The proof of $ {\bf b})$ is immediate from $ {\bf a})$ and the definitions.
\end{dem}

\medskip
Using Theorem \ref{cifactorA} we can now calculate the matrix $M_{n, \lambda}^{\B} $
for $ \langle \cdot, \cdot \rangle^{\B}_{n,\lambda} $, see \eqref{matrix}.  
We illustrate it on $M_{5, 1}^{\B} $. Recall that the diagram basis for
$\Delta_{5}^{\B}(1) $ is given in {\eqref{examplehalf}} and so we get from the Theorem that 
\begin{equation}\label{matrixAA}
M_{5, 1}^{\B}= \begin{pmatrix}
 I_{5}  & 0  & 0\\
0 & \alpha_{x,3} \alpha_{y,1}  I_{4} & 0\\
0 & 0 & \alpha_{x,3} \alpha_{x,4}  \alpha_{y,1} \alpha_{y,2}   I_1
\end{pmatrix}
\end{equation}
where the $0$'s are $0$-matrices of appropriate
dimensions.

\medskip
We now consider the situation where $ \lambda < 0$. We have the following Theorem.

\begin{theorem}\phantomsection\label{cifactorB}
  \begin{description}
\item[a)]  
  Suppose that $ \lambda < 0 $ and that $ f_i^{\lambda} $ is as in \eqref{cases}
but with $ j = 0 $, that is 
$ f_i^{\lambda} = f_k^{\lambda} $ where
\begin{equation}
 f_k^{\lambda}= \raisebox{-.38\height}{\includegraphics[scale=0.9]{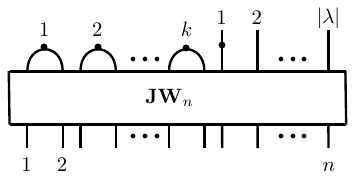}}
      \raisebox{-22\height}{ {\color{black}{.}}}
\end{equation}
Then we have that
\begin{equation}\label{satisfies the formulaA}
  \langle f_k^{\lambda}, f_k^{\lambda} \rangle^{\B}_{n,\lambda}  = \dfrac{1}{\binom{n}{k}}
\, \left(  \alpha_{x,1 } \,  \alpha_{x,2} \cdots \alpha_{x,k+1} \right)
\alpha_{y,1+|\lambda|} \,  \alpha_{y,2+|\lambda|} \cdots \alpha_{y,k+|\lambda|}
\end{equation}
where $ n=2k+|\lambda|$. 
(Note that the
right hand side of \eqref{satisfies the formula} contains $k+1$ factors $ \alpha_{x,l} $ for consecutive 
$ l $'s, but $k$ factors $ \alpha_{y,l} $, for consecutive $ l $'s.
For $ k= 0 $, it is set equal to $ \alpha_{x,1}=x$.
Note also that $ \binom{n}{k}= \dim  \Delta^{\Blob}_n(\lambda) $).
\item[b)]
  Let $ c_i(x,y)  $ be as in \eqref{comb1} and \eqref{matrix}. Then up
  to multiplication by a nonzero scalar in $\comu$, we have that
  \begin{equation}
    c_i(x,y) =
 \left(  \alpha_{x,1 } \,  \alpha_{x,2} \cdots \alpha_{x,i+1} \right)
\alpha_{y,1+|\lambda|} \,  \alpha_{y,2+|\lambda|} \cdots \alpha_{y,i+|\lambda|}
\end{equation}
(where, as before, the product is set equal to $x $ if $i=0$).
\end{description}
 \end{theorem}  
\begin{dem}
  The proof is essentially the same as the proof of Theorem \ref{cifactorA}. We leave the details
  to the reader.
\end{dem}

\medskip
Let us illustrate Theorem \ref{cifactorB} on $M_{5, -1}^{\B} $ and $M_{5, -3}^{\B} $. Recall that
the diagram
basis for
$\Delta_{5}^{\B}(-1) $ is obtained from
{\eqref{examplehalf}} by marking the leftmost propagating
line of each diagram, whereas the diagram basis for $\Delta_{5}^{\B}(-3) $ is given
in {\eqref{examplehalfneg}}. We have 
\begin{equation}\label{matrixA}
M_{5, -1}^{\B}= \begin{pmatrix}
\alpha_{x,1} I_{5}  & 0  & 0\\
0 & \alpha_{x,1} \alpha_{x,2} \alpha_{y,2}  I_{4} & 0\\
0 & 0 & \alpha_{x,1} \alpha_{x,2}  \alpha_{x,3}  \alpha_{y,2} \alpha_{y,3}   I_1
\end{pmatrix}, \, \, \,
M_{5,- 3}^{\B}= \begin{pmatrix}
\alpha_{x,1}  I_{4}  & 0  \\
 0 & \alpha_{x,1} \alpha_{x,2}  \alpha_{y,4}  I_{1}  
\end{pmatrix}  {\color{black}{.}}
\end{equation}

\section{Characterization of $ c_i(x,y)$ in terms of Bruhat order on $W$}\label{charof}
We saw in the Theorems \ref{cifactorA} and \ref{cifactorB} that 
$ c_i(x,y) $ has a factorization in terms of roots for $W$. In this section we describe
the reflections in $ W$ that correspond to these roots. It turns out that these reflections can be
described nicely in terms of the Bruhat order on $W$. 

\medskip
Let $ \beta= w \alpha$ be a root for $ W$ where $ \alpha $ is a simple root. 
Then we define
%\cancel{\color{black}{in a unique way}}
the {\it reflection} $ s_{\beta} \in W $ associated with
$ \beta $ via the formula
\begin{equation}
s_{\beta} := w s_{\alpha} w^{-1}
\end{equation}
where $ s_{\alpha}  \in S$ is the generator associated with $\alpha$.
{\color{black}{It is shown in section 5.7 of \cite{H1} that $ s_{\beta} $
only depends on $ \beta$, not 
on the particular choices of $ w $ and $ \alpha$ such that $ \beta = w \alpha$.}}

\medskip
For our $W$, the reflections $   s_{x,i} $ and $ s_{y,i} $ for
the roots $   \alpha_{x,i} $ and $\alpha_{y,i} $ are given by the formulas of the following Lemma.
\begin{lemma}\label{formulasforroots}
Let $  \alpha_{x,i} $ and $  \alpha_{y,i} $ be the positive roots for $ W $ introduced in \eqref{roots}
and let $  s_{x,i} $ and $  s_{y,i} $ be the associated reflections, for
$ i=1,2, \ldots $ 
Then we have that
\begin{equation}\label{thisshowsrootsref}
s_{x,i} =  ({\color{red} s}  {\color{blue} t})^{i-1} {\color{red} s}, \, \, \, \, \, \, \, \, \, \, \, \, 
s_{y,i} = {\color{blue} t} ({\color{red} s}  {\color{blue} t})^{i-1}     {\color{black}{.}}
\end{equation}
\end{lemma}  
\begin{proof}
This is immediate from \eqref{roots} and the definitions.
\end{proof}  

We now have the following Lemma.
\begin{lemma}\label{lemma43}
  Let $ v \in \Lambda_w $ and let $ \lambda := \varphi(v) \in \Lambda_{ {\color{black}{\pm (n-1)}}} $ where
  $ \varphi: \Lambda_w \rightarrow \Lambda_{ {\color{black}{\pm (n-1)}}} $ is the function defined in
  Lemma \ref{usualLength}. Set $ k = \dfrac{n-1- |\lambda|}{2}$.
  \begin{description}
\item[a)] Suppose that $ \lambda \ge 0$. Then the set of reflections $ s_{\alpha} $ in $W $ satisfying
  $ v < s_{\alpha } v \le w $ is exactly 
\begin{equation}
  \{
  s_{x,\lambda +2 }, \,  s_{x,\lambda +3}, \,   \cdots  \, , s_{x,\lambda +k+1} \} \cup \{
  s_{y,1}, \,  s_{y,2}, \cdots, s_{y,k} \} 
\end{equation}  
obtained by transforming the roots of the factors of \eqref{satisfies the formula}
to reflections.

\item[b)] Suppose that $ \lambda < 0$.
Then the set of reflections $ s_{\alpha} $ in $W $ satisfying
  $ v < s_{\alpha } v \le w $ is exactly 
\begin{equation}
  \{
  s_{x,1 }, \,  s_{x,2}, \,   \cdots  \, , s_{x,k+1} \} \cup \{
  s_{y,1 + | \lambda| }, \,  s_{y, 2 + | \lambda|}, \cdots, s_{y,k + | \lambda|} \} 
\end{equation}  
obtained by 
transforming the roots of the factors of \eqref{satisfies the formulaA}  
to reflections.
\end{description}
\end{lemma}  
  \begin{proof}
  Let us prove $ {\bf a})$. Since $ \lambda \ge 0 $ we have that
  $v $ begins with $ \ese$ and that $ \lambda = l(v ) -1 $. Viewing $ v $ as a 'tail' of $ w $, see
  \eqref{tail}, there are $k $ instances in $ w$ of $ \ese$ to the left of $v$ and also
  $k $ instances of $ \te $ to the left of $v$. Multiplying $ v $ on the right of the reflections
  $ s_{y,1}, \,  s_{y,2}, \cdots, s_{y,k} $ gives the tails from these $\te$'s,
as illustrated below in \eqref{firstillustration} for $ n = 20 $ and $ \lambda =9$
\begin{equation}\label{firstillustration}
  \raisebox{-.38\height}{\includegraphics[scale=1]{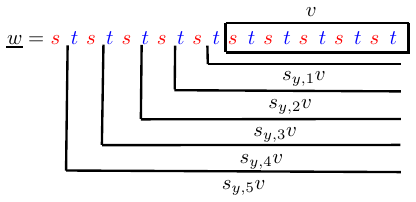}}
\end{equation}
and multiplying 
  $ v $ on the right of the reflections $ s_{x,\lambda +2 }, \,  s_{x,\lambda +3}, \,   \cdots  \, , s_{x,\lambda +k+1} $,
gives the tails from the $\ese$'s, upon deleting the last generator of $ w$, as illustrated below
\begin{equation}\label{secondillustration}
  \raisebox{-.38\height}{\includegraphics[scale=1]{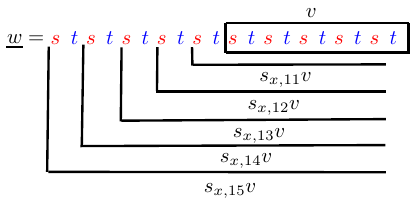}}
      \raisebox{-25\height}{ {\color{black}{.}}}
\end{equation}
These products all satisfy the conditions $ v < s_{\alpha } v \le w $ of the Lemma, and one also checks
that they are the only ones satisfying the conditions. 
This shows $ {\bf a})$, and $ {\bf b})$ is shown the same way.
\end{proof}  

\section{Graded Jantzen filtrations and sum formulas  }\label{Jantzen}
In this final section we study the representation theory of $ \tilde{A}_w $ and
its specialization $ \tilde{A}_w^{\comu}:= \tilde{A}_w \otimes_R \mathbb \comu $,
where the $ R $-algebra structure on $ \comu $ is given by 
mapping $ \alpha^\vee_\ese $ and $   \alpha^\vee_\te  $ to $ 0$.
Note that $ \tilde{A}_w^{\comu} $ 
has already been studied in the literature, in fact for general $ (W,S) $ for example in
\cite{steen} or in 
\cite{Plaza}.

\medskip
As already mentioned, $\tilde{A}_w $ is a cellular algebra and hence also $ \tilde{A}_w^{\comu} $ 
is a cellular algebra, with 
cell modules 
$ \Delta_w^{\comu}(y )  := \Delta_w(y) \otimes_R \comu$ for $ y \in \tilde{\Lambda}_w$. 
Moreover, as already indicated, $ \tilde{A}_w $ and $ \tilde{A}_w^{\comu} $
are $ \Z$-graded algebras with degree
function $ \deg $ given in \eqref{defin endo BS}. In fact  
they are {\it graded} cellular algebras, that is they fit into the following 
definition first formulated by Hu and Mathas, see \cite{hu-mathas}. 

\begin{definition}\label{gradedcel}
{\color{black}{Suppose that $ \Bbbk $ is a commutative ring with identity}} and
that $ \cal A $ is a $ \Bbbk $-algebra which is cellular on 
the triple $ (\Lambda, {\tab},  C ) $. Suppose moreover that $ \cal A $ is a $ \Z$-graded
algebra via $ { \cal A} = \oplus_{i \in \Z} A_i $. Then we say that $ \cal A $ is a 
$\Z$-graded cellular algebra if for
each $ \lambda \in \Lambda $ there is a function $ \deg: \tab(\lambda) \rightarrow \Z$ 
such that for $ \s, \T \in \tab(\lambda) $ we have that 
$ C_{\s \T} \in {\cal A}_{ \deg(\s) + \deg(\T) } $. 
\end{definition}

We shall in general refer to $\Z$-graded cellular algebras simply 
as {\it graded cellular algebras}.

\medskip
The same degree function $ \deg $ that was used for $ \tilde{A}_w $ also makes 
$   {A}_w $ and $ {A}_w^{\comu} $ into $\Z$-graded cellular algebras.
On the other hand, to make the blob-algebras $ \BRn $ and $ \NB$ fit into 
Definition \ref{gradedcel} we use the function $ \deg:\tab(\lambda) \rightarrow \Z $ given by 
\begin{equation}
  \deg(t):\tab(\lambda) \rightarrow \Z, \deg(t) = \left\{\begin{array}{ll}
  2 \{\mbox{number of blobs in $t$} \}& \mbox{ if    }
   \lambda \ge 0 \\
  2 \{\mbox{number of blobs in $t$} \} +1
& \mbox{     if     }   \lambda < 0  \end{array} \right.
\end{equation}  
where $ \tab(\lambda) $ refers to blob half-diagrams as in the paragraph prior to 
equation {\eqref{isomorphism3}}. Thus, when viewing $  \tab(\lambda) $ as the basis
elements for $ \Delta_n^{\B}(\lambda) $ as in {\eqref{examplehalf}} and
{\eqref{examplehalfneg}}, the function $ \deg $ assigns degree 1 to a blob on a propagating line,
and degree 2 to a blob on a non-propagating line.

\medskip
For $ \cal A $ a graded cellular algebra
over $ \Bbbk $ we let 
$ {\cal A}$-mod be the category of $ \cal A $-modules which 
are free over $ \Bbbk $ with finite basis consisting of homogeneous elements. 
For $ M $ 
in $ {\cal A}$-mod we define its {\it graded {\color{black}{rank}}} $ \rankq M \in \Z[q, q^{-1}]$ via
\begin{equation}
\rankq M := \sum_{i \in \Z} \rank M_i q^i 
\end{equation}
where $ \rank M_i $ is the number of basis elements for $ M $ that have degree $ i $.

\medskip
We shall use the notation
$ \tilde{A}_w^{gr} $ and $ \tilde{A}_w^{gr, \comu} $ when referring to 
$ \tilde{A}_w $ and $ \tilde{A}_w^{ \comu} $ as graded cellular algebras and similarly, 
for $ v \in \tilde{\Lambda}_w$, we shall
use the notation $ \Delta_w^{gr}(v ) $ and
$ \Delta_w^{gr, \comu}(v ) $ for the graded cell modules for $ \tilde{ A}_w^{gr} $ and $ \tilde{A}_w^{gr, \comu} $.
On the blob-algebra side, we shall use the notation $\BRgr$ and $ \NBgr $ for
$ \BRn $ and $ \NB $, when viewed as graded cellular algebras,
and shall for $ \lambda \in \Lambda_{ {\color{black}{\pm n}}}  $ use the notation $ \Delta_n^{gr, \Blob}( \lambda) $ and
$ \Delta_n^{gr, \Blob, \comu}(\lambda ) $ for the graded cell modules for
$\BRgr$ and $ \NBgr $.

\medskip
The proof of Theorem \refeq{mainTheoremSection3new} shows that
$\BRnngr  \cong A_w^{gr} $ and so also $\NBnngr  \cong A_w^{gr,\comu} $. 
{\color{black}{Recall the map $ \varphi:
\Lambda_{ \pm (n-1)}
\rightarrow  \Lambda_{w} $ from Lemma \ref{usualLength}}   }
Then, similarly, the proof of Theorem \ref{finallymodules} shows that 
$  \Delta^{gr, \Blob}_{n-1}(\lambda) \cong \Delta^{gr}_w(v)$
where
$ \varphi(\lambda) = v $ and so also $  \Delta^{gr, \Blob, \comu}_{n-1}(\lambda) \cong \Delta^{gr, \comu}_w(v)$.
However, for $ v $ not belonging  
to the image of $ \varphi$, that is for $ v \in  {\Lambda}^c_{w}  $, see 
\eqref{taildec}, 
we need to work a little bit to get an analogous description
of $ \Delta^{gr}_w(v) $ and $  \Delta^{gr, \comu}_w(v)  $.

\medskip
For $ \cal A $ a graded cellular algebra and
$ M = \oplus_{i \in \Z} M_i $ a graded $ \cal A $-module we define the {\it graded shift} $ M[k] $ of
$ M $ as the graded $ \cal A $-module given by 
\begin{equation}  M[k] = \oplus_{i \in \Z} M[k]_i \,\,\, \,  \mbox{where} \,\,\, \,
  M[k]_i:= M_{i-k}       { {\color{black}{.}}}
\end{equation} 
The first part of the following Lemma has just been mentioned, but we still include it 
for later reference. The second part of the Lemma essentially says that the $v$'s in $ \Lambda_w^c $
do not give rise to 'new' cell modules for $A_w^{gr} $. 
\begin{lemma}\phantomsection\label{lemmashift}
  Suppose that $ \underline{w} = \underline{w}_1s^{\prime} $, that
  is $ s^{\prime} $ is the last $S$-generator of $ \underline{w}$. 
\begin{description}
\item[a)] For $ v= \varphi(\lambda)  $ we have that 
$\Delta^{gr}_w(v) \cong  \Delta^{gr, \Blob}_{n-1}(\lambda) ${\color{black}{.}}
\item[b)]
For $ v \in \Lambda_w^c $ set $ v_1 := v s^{\prime} $, that is $ l(v_1) = l(v) +1 $. Then 
$\Delta^{gr}_w(v) \cong
 \Delta^{gr}_w(v_1 )[1] 
 \cong
\Delta^{gr, \Blob}_{n-1}(\lambda)[1]  
$ where
$ \varphi(\lambda) = v_1 $. 
\end{description}  
\end{lemma}  
\begin{dem}
  As mentioned, we only need to prove {\bf b)}. Let $ D_1 $ be the following diagram
\begin{equation}
D_1 :=  \raisebox{-.45\height}{\includegraphics[scale=0.8]{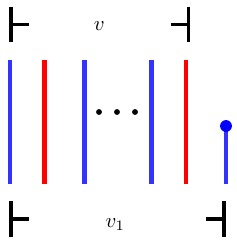}}
\end{equation}
(where we suppose that $ s^{\prime} $ is blue). 
For any diagram $ D $ for $ \Delta^{gr}_w(v_1 )$, we define $ f(D) := DD_1 $, that is 
$ f(D) $ is obtained from $ D $ by multiplying on top with $ D_1 $.
Then $ f $ induces an $R$-isomorphism $ \Delta_w(v_1) \cong  \Delta_w(v) $
which is also a module isomorphism 
since left multiplication
commutes with right multiplication. But 
$ D_1 $ is of degree 1, and so we get that 
$ f:  \Delta^{gr}_w(v_1 ) \rightarrow  \Delta^{gr}_w(v )[-1]$
and hence 
$ \Delta^{gr}_w(v_1 )[1] \cong  \Delta^{gr}_w(v )$. 
Combining this with {\bf a)} we obtain {\bf b)}. 
\end{dem}

\begin{lemma}\phantomsection\label{characterformula}
\begin{description}
\item[a)]  Suppose that $ \lambda \in \Lambda_{ {\color{black}{\pm (n-1)}}}  $ with $\lambda \ge 0 $. Then 
  \begin{equation} \rankq \Delta_{n-1}^{gr, \B}(\lambda) =\rankq \Delta_{n-1}^{gr, \B, \comu}(\lambda) = 
    \sum_{i=0}^{\frac{n-1-\lambda}{2}}
    \rank \Delta_{n-1}^{\TL}(\lambda+2i) q^{2i}   {\color{black}{.}}
  \end{equation}
\item[b)]  Suppose that $ \lambda \in \Lambda_{ {\color{black}{\pm (n-1)}}}  $ with $ \lambda < 0 $. Then 
  \begin{equation} \rankq \Delta_{n-1}^{gr, \B}(\lambda) =
 \rankq \Delta_{n-1}^{gr, \B, \comu}(\lambda) = 
 q \rankq \Delta_{n-1}^{gr, \B}(-\lambda)  {\color{black}{.}}
  \end{equation} 
  \end{description}
\end{lemma}
\begin{dem}
This follows immediately from Lemma \ref{homomorphismf}. 
\end{dem}

\medskip
For $ \lambda \in  \Lambda_{ {\color{black}{\pm (n-1)}}}  $ 
let $ \{b_1, b_2, \ldots, b_m  \} $ be the $ R $-basis for $ \Delta_{n-1}^{ \B}(\lambda) $
obtained from \eqref{matrix} and 
Theorem \ref{cifactorA} if $ \lambda \ge 0 $ or from Theorem \ref{cifactorB} if
$ \lambda <0  $. According to these Theorems 
$ \langle b_i , b_i  \rangle^{\B}_{n-1,\lambda} $ is a product of positive roots for $ W$
and so $ \{b_1, b_2, \ldots, b_m  \} $ consists of homogeneous elements
since $ \langle \cdot , \cdot  \rangle^{\B}_{n-1,\lambda} $
is homogeneous. The degree of $ b_i $ is equal to number of roots appearing
in $ \langle b_i , b_i  \rangle^{\B}_{n-1,\lambda} $ according to Theorem \ref{cifactorA}
and \ref{cifactorB}.

\medskip
Let now $ \alpha $ be a positive root for $ W$. We then  
introduce the following $ \BRgr$-submodule of $ \Delta^{gr, \B}_{n-1}(\lambda) $
\begin{equation}\label{filtrablob}
  \Delta^{gr , \alpha}_{n-1}(\lambda):=   \{ a \in   \Delta^{gr, \B}_{n-1}(\lambda)  \, | \,
  \alpha \mbox{ divides} \, \,  \langle a, b \rangle_{n-1, \lambda}^{\B}
  \mbox{ for all } b \in  \Delta^{gr, \B}_{n-1}(\lambda)  \}   {\color{black}{.}}
\end{equation}
From the above remarks we have that $   \Delta^{gr , \alpha}_{n-1}(\lambda) $ is a free over $ R $
with basis  
\begin{equation}
 \{ b_i   \, | \,
\alpha \mbox{ is a factor of }  \langle b_i, b_i \rangle_{n-1, \lambda}^{\B}  \}   {\color{black}{.}}
\end{equation}
The proof of the following Theorem is a 
compilation of the results from the previous 
sections. 
\begin{theorem}\label{firstmain}
  Supposing that $ v = \varphi(\lambda) $ and that
  $ v < s_{\alpha } v \le w $ we have that 
\begin{equation}\label{gradedsumdim}
  \rankq \Delta_{n-1}^{gr, \alpha  }(\lambda) = \rankq \Delta^{gr}_{w}(s_{\alpha}v) [l(s_{\alpha} v)-l(v)]
   {\color{black}{.}}
\end{equation}
Otherwise, if $ v < s_{\alpha } v \le w $ is not satisfied, we have $   \Delta_{n-1}^{gr, \alpha  }(\lambda)  =0$.
\end{theorem}
\begin{dem}
  Let $ k := \frac{n-1-\lambda}{2}$. 
  Let us first consider the case where $  \lambda \ge 0 $ and 
  $ \alpha = \alpha_{y,i}$ for some $ i=1,2, \ldots, k $, hence 
  $ s_{\alpha} = {\color{blue} t} ({\color{red} s}  {\color{blue} t})^{i-1} $,
  see Lemma \ref{thisshowsrootsref}. 
By Lemma \ref{lemmashift} we have that 
$\Delta^{gr, \Blob}_{n-1}(\lambda) \cong \Delta^{gr}_w(v)      $.
The distinct roots $ \alpha_{x,j} $ and $ \alpha_{y,j} $ are
irreducible and unassociated elements of $ R $ and so it follows from the description
in Theorem \ref{cifactorA} of the matrix $M_{n-1, \lambda}^{\B} $ in 
\eqref{matrix} that 

\begin{equation}\label{calculationA}
  \rankq \Delta^{gr, \alpha}_{n-1}(\lambda) = \sum_{j=i}^{k}
  \rank \Delta_{n-1}^{\TL}(\lambda+2j) q^{2j } =
q^{2i}   \sum_{j=0}^{k-i }
\rank \Delta_{n-1}^{\TL}(\lambda+2i + 2j) q^{2j }  {\color{black}{.}}
  \end{equation}
Using Lemma \ref{characterformula} and
Lemma \ref{usualLength}
we get that \eqref{calculationA} is equal to
\begin{equation}\label{butthisis}
  q^{2i}   \rankq \Delta_{n-1}^{gr, \B}(\lambda+2i ) =
  q^{2i-1}   \rankq \Delta_{n-1}^{gr, \B}(-\lambda-2i ) = 
 q^{2i-1}   \rankq   \Delta^{gr}_w({\color{blue} t} ({\color{red} s}  {\color{blue} t})^{i-1}v)  = 
 q^{2i-1}   \rankq   \Delta^{gr}_w(  s_{\alpha}v)   {\color{black}{.}}
\end{equation}
But $ v $ begins with ${\color{red} s} $ since $ \lambda \ge 0 $ and 
so the last expression of \eqref{butthisis} is $\rankq   \Delta^{gr}_w(  s_{\alpha}v)
[l(s_{\alpha} v)-l(v)] $ which shows the Theorem in this case.
For an illustration of 
$ s_{\alpha} v $, see \eqref{firstillustration}. 

\medskip
Let us now consider the case where still $ \lambda \ge 0$, 
but $ \alpha = \alpha_{x,\lambda + 1 + i }$ and hence
$ s_{\alpha} =  ({\color{red} s}  {\color{blue} t})^{\lambda+i} {\color{red} s}$, 
see Lemma \ref{thisshowsrootsref}. 
Then $  \alpha$ appears in the same blocks
as $ \alpha_{y,i} $ did in the previous case, and so we get  
from \eqref{calculationA} and
\eqref{butthisis} 
that 
\begin{equation}\label{calculationAB}
  \rankq \Delta^{gr, \alpha}_{n-1}(\lambda) =
  q^{2i-1}   \rankq   \Delta^{gr}_{w}({\color{blue} t} ({\color{red} s}  {\color{blue} t})^{i-1}v) =
q^{2i}   \rankq   \Delta^{gr}_w( ({\color{red} s}  {\color{blue} t})^{i}v)  
\end{equation}
where we used Lemma \ref{usualLength} and 
{\bf b)} of Lemma \ref{characterformula} for the second equality. 
On the other hand, 
writing $ v= u s^{\prime} $ with $ l(u)+1 = l(v) $ we get from
Lemma \ref{lemma43} and 
{\bf b)} of Lemma \ref{lemmashift} that
\begin{equation}\label{compare2}
 \rankq   \Delta^{gr}_w( s_{\alpha} v) =    \rankq   \Delta^{gr}_w( ({\color{red} s}  {\color{blue} t})^{i}u)
  =  q  \rankq   \Delta^{gr}_w( ({\color{red} s}  {\color{blue} t})^{i}v)     {\color{black}{.}}
\end{equation}
Comparing \eqref{calculationAB} and \eqref{compare2} we get 
\begin{equation}
\rankq \Delta^{gr, \alpha}_{n-1}(\lambda) =
q^{2i-1}\rankq   \Delta^{gr}_w( s_{\alpha} v)
= \rankq   \Delta^{gr}_w( s_{\alpha} v)
  [l(s_{\alpha} v)-l(v)] 
\end{equation}
which shows the Theorem in this case as well. 
The remaining cases of the Theorem are proved with similar techniques. 
\end{dem}

\medskip
Our next aim is to generalize
Theorem \ref{firstmain} to $  \Delta_w^{gr}(v ) $.
This is immediate if $ v \in  \Lambda_w $ since in that case $ v \in im \, \varphi $, whereas
for $ v \in   \Lambda_w^c $ we have to work a little bit. 
For $ \alpha $ a positive root we first generalize
\eqref{filtrablob} in order to get a graded submodule
$  \Delta^{gr , \alpha}_{w}(v) $ of 
$ \Delta^{gr}_{w}(v) $.
\begin{equation}\label{filtraSoergel}
  \Delta^{gr , \alpha}_{w}(v):=   \{ a \in   \Delta^{gr}_{w}(v)  \, | \,
  \alpha \mbox{ divides} \, \,  \langle a, b \rangle^w_{n, v}
  \mbox{ for all } b \in  \Delta^{gr}_{w}(v)  \}   {\color{black}{.}}
\end{equation}
We then have the following Theorem. 
\begin{theorem}\label{secondmaincompilation}
Let $ \alpha $ be a positive root for $ W$. If $ v < s_{\alpha } v \le w $ then  
  \begin{equation}\label{gradedsumdimB}
  \rankq \Delta_{w}^{gr, \alpha  }(v) = \rankq \Delta^{gr}_{w}(s_{\alpha}v) [l(s_{\alpha} v)-l(v)]
\end{equation}
and otherwise $   \Delta_{w}^{gr, \alpha  }(v)  =0$.
\end{theorem}
\begin{dem}
  As already mentioned, if $ v \in  \Lambda_w $
  the result follows immediately from Theorem \ref{firstmain}, so let us assume
  that $ v \in  \Lambda_w^c$.  As before we write $ \underline{w} = \underline{w}_1s^{\prime} $, where
  $ s^{\prime} $ is the last $S$-generator of $ \underline{w}$ and set $ v_1 = v s^{\prime}$. Then
  $ l(v_1) = l(v)+1$ and $ v_1 \in \Lambda_w$. Let $D$ be a diagram basis element
  for $\Delta_{w}(v) $. Then $ D $ has nonempty zone C and we define $ D_1 $ to be
  the diagram basis element for $\Delta_{w}(v_1) $ which is obtained from $ D$ by making the last
  non-hanging birdcage, corresponding to zone C, hanging.
  Then $ D \mapsto D_1 $ is a bijection between the diagram basis for 
  $\Delta_{w}(v) $ and the diagram basis for $\Delta_{w}(v_1) $. On the other hand, using the
  definition
  of the bilinear forms we have that
\begin{equation}
  \langle D, C \rangle^w_{n, v} = 
  v(\alpha^{\prime})  \langle D_1, C_1 \rangle^w_{n, v_1}
\end{equation}  
where $ \alpha^{\prime} $ is the root corresponding to $ s^{\prime} $ and so  
the matrix for $ \langle \cdot, \cdot \rangle^w_{n, v} $  has the diagonalized form 
$  v(\alpha^{\prime})  M_{n-1, \lambda}^{\B} $ where $ \lambda = \varphi(v_1) $. 
We now assume that $ \lambda \ge 0  $ such that $ M_{n-1, \lambda}^{\B} $ is given by
Theorem \ref{cifactorA}. One then checks that $ v(\alpha^{\prime}) = \alpha_{x, \lambda+1}$. 
Moreover, the set of reflections $ s_{\alpha} $ in $W $ satisfying
  $ v < s_{\alpha } v \le w $ is the union of  $ s_{x,\lambda +1} $ with the set of reflections
$ s_{\alpha} $ satisfying
  $ v_1 < s_{\alpha } v_1 \le w $, and hence, 
in view of Lemma \ref{lemma43}, it is equal to 
\begin{equation}
  \{
 s_{x,\lambda +1},  s_{x,\lambda +2 }, \,  s_{x,\lambda +3}, \,   \cdots  \, , s_{x,\lambda +k+1} \} \cup \{
 s_{y,1}, \,  s_{y,2}, \cdots, s_{y,k} \}
\end{equation}
which is exactly the set of reflections 
obtained by transforming the roots of the factors of the diagonal elements of 
of $  v(\alpha^{\prime})  M_{n-1, \lambda}^{\B} $ to reflections.

Let us now show \eqref{gradedsumdimB}. If $ \alpha $ is a factor of 
a diagonal element of $  M_{n-1, \lambda}^{\B}$ we are reduced to the previous case treated
in Theorem \ref{firstmain}, as follows
\begin{equation}
  \rankq \Delta_{w}^{gr, \alpha  }(v) =   q \rankq   \Delta^{gr, \alpha}_{n-1}(\lambda) = 
  q  \rankq \Delta^{gr}_{w}(s_{\alpha}v_1) [l(s_{\alpha} v_1)-l(v_1)] =
  \rankq \Delta^{gr}_{w}(s_{\alpha}v) [l(s_{\alpha} v)-l(v)]
\end{equation}
where we used Lemma \ref{lemmashift} for the last step. 
On the other hand, if 
$ \alpha = \alpha_{x,\lambda +1} $ we have that $ s_{\alpha} v_1=v $ and so we get 
\begin{equation}
  \rankq \Delta_{w}^{gr, \alpha  }(v) =   q \rankq   \Delta^{gr}_{n-1}(\lambda) = 
  \rankq \Delta^{gr}_{w}(s_{\alpha}v_1) [l(s_{\alpha} v)-l(v)] 
\end{equation}
which shows \eqref{gradedsumdimB} in this case as well. The cases where $ \lambda < 0 $ are treated
with similar techniques.
\end{dem}

\medskip
$  \Delta^{gr , \alpha}_{w}(v) $ is free over $ R $ and hence
we get immediately a specialized version of Theorem \ref{secondmaincompilation}.
\begin{corollary}\label{maincor}
Defining $  \Delta^{gr , \alpha, \comu}_{w}(v) :=   \Delta^{gr , \alpha}_{w}(v) \otimes_R \comu $ 
we have that
\begin{equation}
  \rankq \Delta_{w}^{gr, \alpha, \comu  }(v) = \rankq \Delta^{gr, \comu}_{w}(s_{\alpha}v) [l(s_{\alpha} v)-l(v)]
   {\color{black}{.}}
\end{equation}  
\end{corollary}

The definition of $ \Delta^{gr , \alpha}_{w}(v) $ in \eqref{filtraSoergel}
is reminiscent of the Jantzen filtration for Verma modules. 
To make this analogy even stronger we let $ R_1:= \comu [{ \mathbf x}]$ 
and define $ \Delta^{gr , {\mathbf x}}_w(v) := \Delta^{gr }_{w}(v) \otimes_R R_1  $
where $ R_1 $ is made into an $R$-module via $ \alpha_\ese \mapsto {\mathbf x} $ and 
$ \alpha_\te \mapsto {\mathbf x} $. Then $ \{ \Delta^{gr , {\mathbf x}}_w(v) | v \in \tilde{\Lambda}_w \}$
are graded cell modules for $ \tilde{A}_w^{gr, \mathbf x} :=  \tilde{A}^{gr}_w \otimes_R R_1 $.  
In $ R_1 $ all the roots $ \alpha_{x, i } $ and $ \alpha_{y, i } $ are non-zero scalar multiples
of $ {\mathbf x} $ and so we define for any $ k =1,2,\ldots$ 
\begin{equation}\label{filtraSoergelx}
  \Delta^{gr , k, {\mathbf x}}_{w}(v):=   \{ a \in   \Delta^{gr, \mathbf x}_{w}(v)  \, | \,
  {\mathbf x}^k \mbox{ divides} \, \,  \langle a, b \rangle^w_{n, v}
  \mbox{ for all } b \in  \Delta^{gr, \mathbf x }_{w}(v)  \}
\end{equation}
\begin{equation}\label{filtraSoergelxA}
  \Delta^{gr , k, \comu}_{w}(v):=   \pi (   \Delta^{gr ,  {\mathbf x}, k }_{w}(v) ) 
\end{equation}
where $ \pi:  \Delta^{gr ,  {\mathbf x} }_{w}(v)  \rightarrow \Delta^{gr ,  {\mathbf x} }_{w}(v)  \otimes_{R_1} \comu
$ is the quotient map: here $ \comu $ is made into an $R_1$-module via
${ \mathbf x} \mapsto 0$. 
Then $  \Delta^{gr , \comu}_{w}(v)  \supseteq  \Delta^{gr , 1, \comu}_{w}(v) \supseteq 
\Delta^{gr , 2, \comu}_{w}(v) \supseteq  \ldots $ is a filtration of graded submodules of 
$  \Delta^{gr , \comu}_{w}(v)  $ and we have the following Corollary to Theorem \ref{secondmaincompilation}.
  
\begin{corollary}\phantomsection\label{maincorB2}
\begin{description}
\item[a)] $ \Delta^{gr ,  \comu}_{w}(v) / \Delta^{gr , 1, \comu}_{w}(v) $ is irreducible or zero. 
\item[b)] The following graded analogue of Jantzen's sum formula holds
\begin{equation}  \sum_{k>0} \rankq \Delta_{w}^{gr, k, \comu  }(v) =
  \sum_{\substack{ \alpha > 0 \\  v <s_{\alpha} v \le w}}  \rankq \Delta^{gr, \comu}_{w}(s_{\alpha}v) [l(s_{\alpha} v)-l(v)]
\end{equation}
where $ \alpha > 0 $ refers to the positive roots of $ W$. 
 \end{description}
\end{corollary}

As pointed out in the introduction, analogues of ungraded Jantzen filtrations with
associated sum formulas
exist in many module categories of Lie type and 
give information on the irreducible modules
for the category in question. 
But although graded representation theories in Lie theory have been known
for a long time and would be very useful for calculating decomposition numbers,
to our knowledge
graded sum formulas have so far not been available. The virtue of 
Corollary \ref{maincorB2} is to show the possible form of graded sum formulas in graded
representation theory.

\medskip

It should be noted that in the present $ \NB $-situation, the irreducible modules
can be read off from
Theorem \ref{cifactorA} and 
Theorem \ref{cifactorB} and are in fact 
Temperley-Lieb 
cell modules. 
See also \cite{Plaza1} for a different approach to
this.

\medskip
We believe that the equalities in 
Theorem \ref{secondmaincompilation}
and 
Corollary \ref{maincor} are valid on module level, but have so far not been able to prove so.
But in the remainder of the article
we indicate how to generalize them 
to {\it enriched} Grothendieck group level.
The methods for this are essentially generalizations to the graded case of the methods
in \cite{steen}, where the corresponding ungraded case is treated.

\medskip
Let $ \cal A $ be a graded cellular algebra
over $ \comu $.
Let $ \langle {\cal A}{\rm-mod} \rangle_{q} $ be the {\it enriched} Grothendieck group 
for $ \cal A$, that is $ \langle {\cal A}{\rm-mod} \rangle_{q} $
is the Abelian group generated by symbols $ \langle M \rangle_q $, for $ M $ running over the
modules in $ \langle {\cal A}{\rm-mod} \rangle_{q} $,
subject to the relations $  \langle M \rangle_q =  \langle M_1 \rangle_q +  \langle M_2 \rangle_q $ whenever
there is a short exact sequence $ 0 \rightarrow M_1 \rightarrow M \rightarrow M_2 \rightarrow 0$
in $ \langle {\cal A}{\rm-mod} \rangle_{q} $. 
The grading shift in $ {\cal A}-{\rm mod} $
induces a grading
shift in $ \langle {\cal A}{\rm-mod} \rangle_{q} $
via $ \langle M \rangle_q[ k] := \langle M [k] \rangle_q$ and so we get a 
$ \Z[q, q^{-1}] $-structure on $ {\cal A}{\rm-mod} $ via
$ \langle M \rangle_q +\langle N \rangle_q := \langle M \oplus N \rangle_q$ and 
$ q^{k} \langle M \rangle_q  := \langle M \rangle_q[k]$.

\medskip
The following is a natural generalization of the definition of a cellular category, see \cite{Wes},
to the $ \Z$-graded case. 
\begin{definition}{\label{cellularcat}}
  Let $ \Bbbk $ be a commutative ring with identity and let $ \cal C $ be a $ \Bbbk $-linear
  $\mathbb Z$-graded category, that is for objects $ m,n $ in $ \cal C $ we have a decomposition 
  \begin{equation}\label{sumgrad} {\rm Hom}_{\cal C}(m,n) = \oplus_{i \in \Z}{\rm Hom}_{\cal C}(m,n)_i  {\color{black}{.}}
\end{equation}    
Suppose that $ \cal C $ is endowed with a duality 
$ \ast$. 
Then $ \cal C $ is called a $\Z$-graded cellular category if there exists a poset $ \Lambda$ and for each 
$ \lambda \in \Lambda $ and each object $ n $ in $ \cal C $ a finite set $ \tab(n, \lambda)  $
which is decomposed as 
$ \tab(n, \lambda) = \cupdot_{ i \in \Z}\tab(n, \lambda)_i $
together with a map 
$ \tab(m, \lambda) \times  \tab(n, \lambda)
\rightarrow {\rm Hom}_{\cal C}(m,n), (S,T) \mapsto C_{ST}^{\lambda} $,
satisfying $ C_{ST}^{\lambda} \in {\rm Hom}_{\cal C}(m,n)_{i+j} $ if $ S \in \tab(n, \lambda)_i $
and $ T \in \tab(n, \lambda)_j $. These data satisfy that 
$(C_{ST}^{\lambda})^{\ast} = C_{TS}^{\lambda}$
and that 
\begin{equation}
  \{ C_{ST}^{\lambda} | S \in \tab(m, \lambda), T \in \tab(n, \lambda), \lambda \in \Lambda \} \mbox{ is a
    homogeneous $\Bbbk$-basis for } Hom_{\cal C}(m,n)  \\
\end{equation}
and for all $ a \in Hom_{\cal C}(n,p)_i, S \in \tab(m, \lambda)_j, T \in \tab(n, \lambda)_k $ 
\begin{equation}\label{multstruc}
a  C_{ST}^{\lambda} = \sum_{ S^{\prime} \in \tab(p, \lambda)_{i+j}} r_{a}(S^{\prime}, S) C_{ S^{\prime}, T}^{\lambda}  \mbox{ mod } A^{\lambda}_{i+j+k}
\end{equation}
where 
$ A^{\lambda}  $ is the span of $\{  C_{ST}^{\mu } | \mu < \lambda, S \in \tab(m, \mu), T \in \tab(p,\mu) \} $.
\end{definition}

\medskip
This following simple fact was already mentioned
in \cite{steen}, in the ungraded case. Let $ \cal C$ be a graded cellular category and 
let $ A $ be a finite subset of the objects of $ \cal C$. 
Define $ {\rm End}_{\cal C}(A) $ as the direct sum 
\begin{equation}
 {\rm End}_{\cal C}(A)  := \oplus_{ m,n \in A } {\rm Hom}_{\cal C}(m,n)  {\color{black}{.}}
\end{equation}
Then $ {\rm End}_{\cal C}(A) $ has a $\Bbbk$-algebra structure as follows
\begin{equation}
g \cdot f := \left\{\begin{array}{ll} g\circ f  & \mbox{ if } f \in {\rm Hom}_{\cal C}(m,n), g \in {\rm Hom}_{\cal C}(n,p) \mbox{ for some } m,n,p \\
0 & \mbox{ otherwise} 
\end{array}
\right. 
\end{equation}
and, in view of \eqref{sumgrad} and \eqref{multstruc}, this is a graded $ \Bbbk $-algebra structure.
Moreover, we have the following Theorem.
\begin{theorem}\label{52}
Let $ \cal C $ be a graded cellular category and let $ A $ be a finite subset of the objects for $ \cal C$. 
Define for $ \lambda \in \Lambda $ the set $ \tab(\lambda) := \cup_{ n \in A } \tab(n, \lambda) $.
Let for $ S \in \tab(\lambda), T \in \tab(\lambda) $ the element $ C_{ST}^{\lambda} \in    {\rm End}_{\cal C}(A)  $ be
defined as the inclusion of $ C_{ST}^{\lambda} \in {\rm Hom}_{\cal C}(m,n) $ in $ {\rm End}_{\cal C}(A)  $. 
Then these data define a graded cellular algebra structure on $ {\rm End}_{\cal C}(A) $. 
\end{theorem}
\begin{dem}
Just as in the ungraded case considered in \cite{steen}, this follows immediately from the definitions.
\end{dem}

\medskip
For a general Coxeter system $(W,S)$, it was shown in \cite{EW} that the diagrammatic
Soergel categories $ \tildeSoergelcat $ and $ {\mathcal D}_{(W,S)}^{\comu} $, see
Definition \ref{defin endo BS} and
Remark \ref{cyclotomic}, 
are graded cellular categories in the sense of Definition 
\ref{cellularcat}. 

\medskip
Let us indicate the ingredients that make the categories
$ \tildeSoergelcat $ and $\tildeSoergelcatC  $, see Definition \ref{defin endo BS} and
Remark \ref{cyclotomic}, 
fit into Definition \ref{cellularcat}, 
for our choice
of $ (W,S) $.
In case of $ \tildeSoergelcat $, 
we use for $ \Bbbk $ the ring $ R $, and for the objects 
and morphisms we use the objects 
and morphisms given in 
Definition \ref{defin endo BS}.
For the poset $ \Lambda $ we use $ W$ itself, endowed with the
Bruhat order poset structure.
For $ \underline{w} \in \Exp  $ starting with $ {\color{red} s} $ and in reduced form, 
that is $  \underline{w} = w $, 
we use for $ \tab(\underline{w}, v)  $ the
set of birdcages $ \tab_w(v)  $.
For $ \underline{w} \in \Exp  $ starting with $ {\color{blue} t} $ and in reduced form, 
we use for $ \tab_w(v)  $ the corresponding set of birdcages
$ \tab_w(v)  $. 
For $ \underline{w} \in \Exp$ a general object
in $  \tildeSoergelcat $ not in reduced form, 
we 
use the general light leaves 
construction from \cite{EW}. Since we do not need the details of
this construction we skip it at this point. 
In case of $ \tildeSoergelcatC $ we use the same ingredients as for
$ \tildeSoergelcat $, except that for $ \Bbbk $ we replace
$ R $ by $ \comu$.

\medskip
Let us now fix a finite subset $ W_0 \subseteq W $ of $ W $ such that
$ v \in W_0, u \le v \implies u \in W_0$, that is $W_0 $ is an ideal in $W$. For
each $ z \in W_0 $ we let $ \underline{z} $ be its (unique) reduced expression. We then set 
$ \underline{W}_0 := \{ \underline{z} \, | \,  z \in W_0 \} \subseteq \Exp $ and define
\begin{equation} \tilde{A}^{gr}_{W_0} := \mbox{End}_{\tildeSoergelcat} (\underline{W}_0) \,\, \,\,  \mbox{and} \,\, \,\,
\tilde{A}_{W_0}^{gr, \comu} := \mbox{End}_{\tildeSoergelcatC} (\underline{W}_0)  {\color{black}{.}}
\end{equation}
Note that in this setup we recover the graded cellular algebras
$ \tilde{A}^{gr}_w$ and $ \tilde{A}^{gr, \comu}_w$ 
by taking $ W_0:= \{w\} $.

\medskip
We now have the following Theorem, which is
our main reason for changing to the categorical setting.
\begin{theorem}\label{54}
$ \tilde{A}_{W_0}^{gr, \comu}$ is a graded quasi-hereditary algebra over $ \comu $.
\end{theorem}
\begin{dem}
  The proof of Theorem 8.5 in \cite{steen}, corresponding to the ungraded setting, carries over to
the present graded setting. 
 \end{dem}

\medskip
Let $  \Delta^{gr}_{W_0}( v) $ and $  \Delta^{gr, \comu}_{W_0}( v) $
be the graded cell modules for 
$ \tilde{A}_{W_0}^{gr}$ and $ \tilde{A}_{W_0}^{gr, \comu}$. There is an $ R$-module decomposition
\begin{equation}\label{orto}
  \Delta^{gr}_{W_0}( v) = \oplus_{ \underline{z} \in \underline{W}_0}   \Delta^{gr}_{z}( v)
\end{equation}
where we use ${z}  $ instead of ${w}  $ 
to indicate that $\underline{z}  $ may begin with
$ {\color{red} s} $ as well as 
$ {\color{blue} t} $. There is a similar decomposition for $   \Delta^{gr, \comu}_{W_0}( v) $. Let 
  $ \langle \cdot, \cdot \rangle^{W_0}_{v} $ be the bilinear form on $  \Delta^{gr}_{W_0}( v) $.
  It is orthogonal with respect to the decomposition in \eqref{orto}. Mimicking what we did
  for $  \Delta^{gr}_{w}( v) $ we choose $ \alpha $ a root for $ W$ and define for
  $   \Delta^{gr, \alpha }_{W_0}(v) $ via 
\begin{equation}
  \Delta^{gr, \alpha }_{W_0}(v) := \{ a \in  \Delta^{gr}_{W_0}(v)  \, | \, \alpha \mbox{ divides }
  \langle a, b \rangle_{v}^{W_0}   \mbox{ for all } b \in   \Delta^{gr}_{W_0}(v) \}  
\end{equation}  
and set 
\begin{equation}
  \Delta^{gr, \alpha, \comu }_{W_0}(v):=   \pi(\Delta^{gr, \alpha }_{W_0}(v))
\end{equation}  
where $ \pi $ is before. 
Following \cite{steen}, we define for $ z \in W_0 $ 
projection maps $ \varphi_{{z}} $ 
as follows
\begin{equation}
  \varphi_{z}:
  \langle \tilde{A}_{W_0}^{gr, \comu} -{\rm mod} \rangle_{q} \rightarrow
  \langle \tilde{A}_{z}^{gr, \comu}-{\rm mod} \rangle_{q}, \, \, \, \,   \langle \Delta^{gr, \comu }_{W_0}(v)
  \rangle_q  \mapsto
          \langle \Delta^{gr, \comu }_{z}(v) \rangle_q 
\end{equation}
and arguing as in \cite{steen}, we get the following compatibility 
at Grothendieck group level
\begin{equation}\label{compatibi}
  \varphi_{z}( \langle  \Delta^{gr, \alpha, \comu }_{W_0}(v) \rangle_q )  =
  \langle \Delta^{gr, \alpha, \comu }_{z}(v) \rangle_q  {\color{black}{.}}
\end{equation}
We have 
natural homomorphisms of $ \Z[q,q^{-1}] $-modules
\begin{equation}
\begin{array}{c}
  \rankWZeroq: 
  \langle \tilde{A}_{W_0}^{gr, \comu}-{\rm mod} \rangle_{q} \rightarrow \Z[q, q^{-1}], \, \, \, 
  \langle M \rangle_q \mapsto \rankq M  \\ \\
  \rankzq : 
  \langle \tilde{A}_{z}^{gr, \comu}-{\rm mod} \rangle_{q} \rightarrow \Z[q, q^{-1}], \, \, \, 
  \langle M \rangle_q \mapsto \rankq M    {\color{black}{.}}
\end{array}
\end{equation}
Let $ \Phi:   \langle \tilde{A}_{W_0}^{gr, \comu}-{\rm mod} \rangle_{q} \rightarrow
\oplus_{ \underline{z} \in \underline{W}_0 } \Z[q, q^{-1}] $ be the $ \Z[q,q^{-1}] $-homomorphism 
whose $\underline{z}$'th coordinate is equal to the composite map $ \rankzq \circ \, \varphi_{z} $. 
With this notation we have the following Theorem. 
\begin{theorem}\label{secondlast} 
$ \Phi:   \langle \tilde{A}_{W_0}^{gr, \comu}-{\rm mod} \rangle_{q} \rightarrow
  \oplus_{ \underline{z} \in \underline{W}_0 } \Z[q, q^{-1}] $ is an isomorphism of
  $ \Z[q,q^{-1}] $-modules.
\end{theorem}
\begin{dem}
  Since $ \langle \tilde{A}_{W_0}^{gr, \comu}-{\rm mod} \rangle_{q} $ and
  $ \oplus_{ \underline{z} \in \underline{W}_0 } \Z[q, q^{-1}] $ are free $ \Z[q, q^{-1}] $-modules
  of the same {\color{black}{rank}},
  it is enough to show that $ \Phi $ is surjective, by Vasconcelos' Theorem
  once again, see \cite{Vasconcelos}. Let $ f =  \sum_{ \underline{z} \in \underline{W}_0 } f_{ \underline{z}} \in
  \oplus_{ \underline{z} \in W_0 } \Z[q, q^{-1}]  $ and choose $ f_{ \underline{z}_0} $ nonzero and 
  $ \underline{z}_0 $ minimal with respect to this condition. The $ \underline{z}_0 $'th component of
  $ \Phi( \langle \Delta^{gr, \comu }_{W_0}(z_0) \rangle_q) $
  is $\rankzeroq( \Delta_{ z_0}( z_0)) =1 $ and so the 
  $ {z}_0 $'th component of
  $ \Phi( f_{ \underline{z}_0} \langle \Delta^{gr, \comu }_{W_0}(z_0) \rangle_q) $ is $ f_{ \underline{z}_0} $. 
  On the other hand, the $z$'th component of $ \Phi( \Delta^{gr, \comu }_{W_0}(z_0)) $ is
$\rankzq \langle \Delta_{ z}( z_0) \rangle_q $ which is 
nonzero only if $ z_0  \le z $. Hence, we can use induction on
$  f -  \Phi( f_{ \underline{z}_0} \langle \Delta^{gr, \comu }_{W_0}(z_0) \rangle_q) $ 
and get that $ f  \in {\rm im}  \Phi $, as claimed. 
\end{dem}

\medskip

We can now prove the promised Grothendieck group extension of 
Corollary \ref{maincor}.

\begin{corollary}\label{mimickingA}
  For $ v \in W_0 $
  we have $ \Delta^{gr, \alpha ,\comu }_{W_0}(v) = 0 $
  unless $ w \ge s_{\alpha} v > v $. If
$ w \ge s_{\alpha} v > v $ then 
\begin{equation}\label{bothsides}
  \langle \Delta^{gr, \alpha, \comu}_{W_0}(v) \rangle_q   =
  \langle \Delta^{gr,  \comu}_{W_0}(s_{\alpha}v )[l(s_{\alpha}v) -l(v)]
  \rangle_q    {\color{black}{.}}
\end{equation}  
\end{corollary}  
\begin{dem}
We apply $ \Phi $ to \eqref{bothsides} and check that both sides are equal. 
Using \eqref{compatibi} and Corollary \ref{maincor} we get that the $\underline{z}$'th component
of the left hand side is $ \rankq \Delta^{gr,  \comu }_{z}(s_{\alpha}v) [l(s_{\alpha} v)-l(v)] $ 
which by definition of $ \Phi $ coincides with the right hand side. We then use
Theorem \ref{secondlast} 
to conclude the proof. 
\end{dem}  

\medskip

Finally, the Grothendieck group extension of Corollary \ref{maincorB2}
is proved with the same techniques, upon changing the ground ring for the category
$ \tildeSoergelcat $ from $ R $ to 
$ R_1$. The cell modules for $ \tildeSoergelcat $ are called $ \Delta^{gr, \mathbf x}_{W_0}(v) $ and we define
for $ k = 1,2,\ldots $ 
\begin{equation}
  \Delta^{gr, k, \mathbf x}_{W_0}(v) := \{ a \in  \Delta^{gr, \mathbf x}_{W_0}(v)  \, | \,
  \langle a, b \rangle_{v}^{W_0}  \in {\mathbf x^k} R_1   \mbox{ for all } b \in   \Delta^{gr, \mathbf x}_{W_0}(v) \}  
\end{equation}  
and set 
\begin{equation}
  \Delta^{gr, k , \comu }_{W_0}(v):=   \pi(   \Delta^{gr, k, \mathbf x}_{W_0}(v)   )  {\color{black}{.}}
\end{equation}  
Mimicking the proof of Corollary \ref{mimickingA} we then have the following generalization of
Corollary \ref{maincorB2}.

\begin{corollary}\phantomsection\label{mimicking2}
The following graded analogue of Jantzen's sum formula holds:
\begin{equation}  \sum_{k>0} \langle \Delta_{W_0}^{gr, k, \comu  }(v) \rangle_q =
  \sum_{\substack{ \alpha > 0 \\  v <s_{\alpha} v \le w}}
  \langle  \Delta^{gr, \comu}_{W_0}(s_{\alpha}v)[l(s_{\alpha}v) -l(v)] \rangle_q   {\color{black}{.}}
\end{equation}  
 \end{corollary}

\bibliographystyle{myalpha}
 
\bibliography{mybibfile}

\begin{thebibliography}{X}



  
\bibitem{A} H. H. Andersen, {\it A sum formula for tilting filltrations}, 
Journal of Pure and Applied Algebra {\bf 152} (2000), 17-40.

\bibitem{A1} H. H. Andersen, {\it Filtrations of cohomology modules for Chevalley groups}, 
Annales scientifiques de l'{\'E}.N.S. $4^e$ s{\'e}rie, tome {\bf 16}, $\rm{n}^o$ 4 (1983), 495-528.



  
\bibitem{AJS}
 H. H. Andersen, J. C. Jantzen, W. Soergel, 
 \textit{Representations of quantum groups at a $p$-th root of unity and
   of semisimple groups in characteristic $p$: independence of $p$}, 
Astérisque {\bf 220} (1994) Paris: Société Mathématique de France, p. 321.  

{\color{black}{
\bibitem{BGS}
A. Beilinson, V. Ginzburg, W. Soergel, 
\textit{Koszul duality patterns in representation theory}, 
J. Am. Math. Soc. {\bf 9}(2) (1996), 473-527.}}
  
  

\bibitem{BCH} C. Bowman, A. Cox, A. Hazi, \textit{Path isomorphisms between quiver Hecke and diagrammatic Bott-Samelson endomorphism algebras}, arXiv:2005.02825, to appear in Advances in Mathematics.  


\bibitem{BCHM} C. Bowman, A. Cox, A. Hazi, D. Michailidis, \textit{
Path combinatorics and light leaves for quiver Hecke algebras.}
Math. Z. {\bf 300}(3) (2022), 2167-2203.




\bibitem{brundan-klesc} J. Brundan, A. Kleshchev,
\textit{Blocks of cyclotomic Hecke algebras and Khovanov-Lauda algebras}, Invent. Math. {\bf 178} (2009), 451-484.


\bibitem{BLS} G. Burrull, N. Libedinsky, P. Sentinelli, 
\textit{$p$-Jones-Wenzl idempotents}, 
Advances in Mathematics {\bf 352} (2019), 246–264.




\bibitem{blob positive} A. Cox, J. Graham, P. Martin, {\it The blob algebra in positive characteristic},
Journal of Algebra, {\bf 266}(2) (2003), 584-635.



\bibitem{EL} B. Elias, N. Libedinsky, {\it  Indecomposable Soergel bimodules for universal Coxeter groups.
  With an appendix by Ben Webster}, Trans. Am. Math. Soc.
{\bf 369}(6) (2017), 3883-3910. 



  
\bibitem{EW} B. Elias, G. Williamson, {\it Soergel calculus},  Representation Theory
{\bf 20} (2016), 295-374.


\bibitem{EhrigStroppel}
M. Ehrig, C. Stroppel, \textit{
Koszul gradings on Brauer algebras}, 
Int. Math. Res. Not. {\bf 2016} (13) (2016), 3970-4011. 



\bibitem{FL}
P. Fiebig, M. Lanini, \textit{
  The combinatorial category of Andersen, Jantzen and Soergel and filtered moment graph sheaves},
Abh. Math. Semin. Univ. Hamb. {\bf 86}(2) (2016), 203-212.



{\color{black}{
\bibitem{GJSV}
A. M. Gainutdinov, J. L. Jacobsen, H. Saleur, R. Vasseur, \textit{A physical approach to the
classification of indecomposable Virasoro representations from the blob algebra}, Nuclear Physics
B, {\bf 873}(3) 2013, 614–681. 
}}




\bibitem{GL} J. J. Graham, G. I. Lehrer,
  \textit{Cellular algebras}, Inventiones Mathematicae {\bf 123} (1996), 1-34.

\bibitem{GL2} J. J. Graham, G. I. Lehrer,
  \textit{The representation theory of affine Temperley-Lieb algebras},
  Enseign. Math., II. S\'er. {\bf 44}(3-4) (1998), 173-218.


\bibitem{HMP} A. Hazi, P. Martin, A. Parker, \textit{Indecomposable tilting modules for the blob algebra},
Journal of Algebra {\bf 568} (2021), 273-313. 

\bibitem{HRW} M. Hogancamp, D. E. V. Rose, P. Wedrich,
  \textit{A Kirby color for Khovanov homology}, arXiv:2210.05640,
{\color{black}{to appear in Journal of the European Mathematical Society. }}


{\color{black}{
\bibitem{Hu} J. Hu, {\it BGG Category $\mathcal O$ and $\mathbb Z$-Graded Representation Theory},
in: East China Normal University Scientific Reports {\bf 16}, 
Forty Years of Algebraic Groups, Algebraic Geometry, and Representation Theory in China. }}


\bibitem{hu-mathas} J. Hu, A. Mathas, \textit{Graded cellular bases for the cyclotomic Khovanov-Lauda-Rouquier
algebras of type $A$}, Adv. Math., {\bf 225} (2010), 598-642.


  
\bibitem{H1} J. E. Humphreys, {\it Reflection groups and Coxeter groups}, volume {\bf 29} of Cambridge Studies in
Advanced Mathematics. Cambridge University Press, Cambridge, 1990.


\bibitem{H} J. Humphreys, {\it Representations of semisimple Lie algebras in the BGG category O}, (2008), 
Graduate Studies in Mathematics 94, Providence, R.I.: American Mathematical Society, ISBN 978-0-8218-4678-0. 




  


\bibitem{Jo} V. F. R. Jones, {\it Index for subfactors}, Invent. Math. {\bf 72}(1) (1983) 1–25.


\bibitem{KhovanovLauda} M. Khovanov, A. Lauda, \textit{A diagrammatic approach to categorification of quantum groups I}, Represent. Theory {\bf 13}
(2009), 309-347.

  
\bibitem{LibedinskyGentle} N. Libedinsky, {\it 
  Gentle introduction to Soergel bimodules I: The basics}, {\color{black}{Sao Paulo Journal of
  Mathematical Sciences, {\bf 13}(2) (2019), 499-538.}}


  


\bibitem{LPR} D. Lobos, D. Plaza, S. Ryom-Hansen,
\textit{The nil-blob algebra: an incarnation of type $\tilde{A}_1$
Soergel calculus and of the truncated blob algebra}, 
Journal Algebra {\bf 570} (2021),  297-365. 




\bibitem{Lobos-Ryom-Hansen} D. Lobos, S. Ryom-Hansen,
  \textit{Graded cellular basis and Jucys-Murphy elements
  for generalized blob algebras},
{\color{black}{Journal of Pure and Applied Algebra}, 
{\bf 224}(7) (2020), 106277, 1-40.}




\bibitem{Mat-Sal} P. P. Martin, H. Saleur, \textit{The blob algebra and the periodic Temperley-Lieb algebra},
Lett. Math. Phys. {\bf 30} (1994), 189-206.


{\color{black}{
\bibitem{Mar-Wood}
P. P. Martin, D. Woodcock, \textit{On the structure of the blob algebra}, 
J. Algebra {\bf 225}(2) (2000), 957-988. 
}}


\bibitem{MathasJan} A. Mathas, 
\textit{Positive Jantzen sum formulas for cyclotomic Hecke algebras}, 
Math. Z. {\bf 301}(3) (2022), 2617-2658. 

\bibitem{Michailidis}
D. Michailidis, 
\textit{Bases and BGG resolutions of simple modules of Temperley-Lieb algebras of type B}, 
Int. Math. Res. Not. {\bf 2022}(12) (2022), 9357-9412. 


\bibitem{Plaza1} D. Plaza, \textit{Graded decomposition numbers for the blob algebra}, 
J. Algebra {\bf 394}, (2013), 182-206.


\bibitem{Plaza} D. Plaza, \textit{Graded cellularity and the monotonicity conjecture}, 
Journal Algebra {\bf 473} (2017), 324-351.



\bibitem{PlazaRyom}
D. Plaza,  S. Ryom-Hansen, \textit{Graded cellular bases for Temperley-Lieb algebras of type A and B},
Journal of Algebraic Combinatorics, {\bf 40}(1) (2014), 137-177.


\bibitem{Russel} H. Russell, \textit{The Bar-Natan skein module of the solid torus
  and the homology of $(n, n)$ Springer varieties}, Geom.
Dedicata {\bf 142} (2009), 71-89. 

  

\bibitem{steen} S. Ryom-Hansen, \textit{Jucys-Murphy elements for Soergel bimodules},
  Journal of Algebra, {\bf 551} (2020), 154-190.

{\color{black}{
\bibitem{Shan} P. Shan, 
\textit{Graded decomposition matrices of $v$-Schur algebras via Jantzen filtration}, 
Represent. Theory {\bf 16} (2012), 212-269.}}



  
\bibitem{Stroppel} C. Stroppel,
\textit{Category $ \mathcal O$: gradings and translation functors}, J. Algebra {\bf 268} (1) (2003), 301-326.

  
\bibitem{Vasconcelos} W. V. Vasconcelos, \textit{On finitely generated flat modules}, Trans. Amer. Math. Soc.,
  {\bf 138} (1969),  505-512.

\bibitem{Wenzl} H. Wenzl, {\it On sequences of projections}, C. R. Math. Rep. Acad. Sci. Can. 9 {\bf 1}
  (1987), 5–9.


\bibitem{Wes} B. W. Westbury,  {\it Invariant tensors and cellular categories},
  Journal of Algebra, {\bf 321}(11) 2009, 3563-3567.


\end{thebibliography}

\noindent
mhernandez@inst-mat.utalca.cl, Universidad de Talca, Chile.\newline
steen@inst-mat.utalca.cl, Universidad de Talca, Chile.

%\section*{References}
%\bibliographystyle{myalpha} 
%\bibliography{mybibfile}

\end{document}